\documentclass[a4paper,10pt]{amsart}

\usepackage{etoolbox}
\usepackage{setspace} 
\usepackage{amsthm} 
\allowdisplaybreaks
\usepackage[only,llbracket,rrbracket]{stmaryrd} 
\usepackage{latexsym}
\usepackage{amsmath}
\usepackage{amssymb}
\usepackage{amscd}
\usepackage{mathrsfs} 
\usepackage{dsfont} 
\usepackage{xcolor}

\usepackage[backref,pagebackref,breaklinks,linktoc=all,hyperfootnotes=false]{hyperref} 
\hypersetup{
	colorlinks  =	true,
	linkcolor	=	{green!50!black},
	citecolor	=	{green!50!black},
	urlcolor    =   {green!50!black}
}

\usepackage{multibib} 

\newcites{GitHub}{\small GitHub Repositories}

\usepackage{mathrsfs}
\usepackage[utf8x]{inputenc}
\usepackage{enumitem} 
\usepackage{ucs}
\usepackage{graphicx}
\usepackage{tabularx}
\usepackage{wrapfig, booktabs} 
\usepackage{pgfplots}
\usepackage{caption}
\usepackage{subfigure}
\usepackage{longtable}
\usepackage{multicol}
\usepackage{mathtools}
\usepackage{tikz-cd}
\usepackage{tikz}

\usepackage{multicol} 
\usepackage{multirow} 

\usepackage{placeins} 

\definecolor{darkgreen}{rgb}{0,0.45,0}
\definecolor{darkred}{rgb}{0.75,0,0}
\definecolor{darkblue}{rgb}{0,0,0.6}

\usepackage[margin=1in]{geometry}
\usepackage[skins]{tcolorbox}
\usepackage{minibox}

\usepackage{float} 	

\usepackage[ruled,vlined]{algorithm2e} 

\makeatletter 
\def\@tocline#1#2#3#4#5#6#7{\relax
	\ifnum #1>\c@tocdepth 
	\else
	\par \addpenalty\@secpenalty\addvspace{#2}%
	\begingroup \hyphenpenalty\@M
	\@ifempty{#4}{%
		\@tempdima\csname r@tocindent\number#1\endcsname\relax
	}{%
		\@tempdima#4\relax
	}%
	\parindent\z@ \leftskip#3\relax \advance\leftskip\@tempdima\relax
	\rightskip\@pnumwidth plus4em \parfillskip-\@pnumwidth
	#5\leavevmode\hskip-\@tempdima
	\ifcase #1
	\or\or \hskip 1em \or \hskip 2em \else \hskip 3em \fi%
	#6\nobreak\relax
	\dotfill\hbox to\@pnumwidth{\@tocpagenum{#7}}\par
	\nobreak
	\endgroup
	\fi}
\makeatother

\date{\today}
\author{Luca Dall'Ava, Aleksander Horawa}
\title{Balanced triple product $p$-adic $L$-functions and Stark points}

\newcommand{\N}{\mathbb{N}}
\newcommand{\zp}{\mathbb{Z}_p}
\newcommand{\cp}{\mathbb{C}_p}
\newcommand{\F}{\mathbb{F}}
\newcommand{\qp}{\mathbb{Q}_p}
\newcommand{\qpbar}{\overline{\qp}}
\newcommand{\ql}{\mathbb{Q}_\ell}
\newcommand{\qll}{\mathbb{Q}_{\ell^2}}
\newcommand{\A}{\mathbb{A}}
\newcommand{\Ac}{\mathcal{A}}

\newcommand{\R}{\mathbb{R}}

\newcommand{\T}{\mathbb{T}}
\newcommand{\Ttilde}{\widetilde{\mathbb{T}}}

\newcommand{\W}{\mathcal{W}}
\newcommand{\D}{\mathscr{D}}
\newcommand{\E}{\mathscr{E}}
\newcommand{\Etilde}{\widetilde{\E}}
\newcommand{\ZZ}{\mathscr{Z}}

\newcommand{\End}{\operatorname{End}}
\newcommand{\Image}{\operatorname{Im}}

\newcommand{\ph}[1]{\phantom{#1}}

\newcommand{\UL}[1]{\langle\varpi_{D_{#1}}\rangle}
\newcommand{\IL}[1]{\langle\sqrt{#1}\rangle_\mathbf{I}}

\newcommand{\Aut}{\operatorname{Aut}}
\newcommand{\disc}{\operatorname{disc}}
\newcommand{\cond}{\operatorname{cond}}

\newcommand{\AL}[1]{\tau_{#1}^D}
\newcommand{\Np}{\mathrm{N}_p}
\newcommand{\Mod}[1]{\ (\mathrm{mod}\ #1)}
\newcommand{\cyc}{\varepsilon_{\textrm cyc}}

\newcommand{\magma}{\texttt{magma}~\cite{Magma}}

\newcommand{\Hom}{\operatorname{Hom}}
\newcommand{\Ind}{\operatorname{Ind}}
\newcommand{\GL}{\operatorname{GL}}
\newcommand{\SL}{\operatorname{SL}}

\newcommand{\Gal}{\operatorname{Gal}}

\newcommand{\Sp}{\operatorname{Sp}}
\newcommand{\Frac}{\operatorname{Frac}}
\newcommand{\Ad}{\operatorname{Ad}}
\newcommand{\St}{\operatorname{St}}
\newcommand{\vol}{\operatorname{vol}}

\newcommand{\Z}{\mathbb{Z}}
\newcommand{\Q}{\mathbb{Q}} 
\newcommand{\C}{\mathbb{C}}
\newcommand{\bal}{\operatorname{bal}}
\renewcommand{\O}{\mathcal{O}}
\newcommand{\iso}{\cong}

\newcommand{\Addresses}{
	{
		\bigskip
        \bigskip
		\footnotesize
        \setstretch{1.2}
    
		\noindent L. Dall'Ava, \textsc{Dipartimento di Matematica, Università degli Studi di Milano,
			Milano, Italy.}\par\nopagebreak
   
		\noindent \textit{E-mail address:} \href{mailto:luca.dallava@unimi.it}{\texttt{luca.dallava@unimi.it}}\par\nopagebreak
  
		\noindent \textit{URL:} \url{https://sites.google.com/view/luca-dallava}
		
		\medskip
		
		\noindent A. Horawa, \textsc{Mathematical Institute, University of Oxford, Oxford, United Kingdom.}\par\nopagebreak
  
		\noindent \textit{E-mail address:} \href{mailto:horawa@maths.ox.ac.uk}{\texttt{horawa@maths.ox.ac.uk}}\par\nopagebreak
        
		\noindent \textit{URL:} \url{https://people.maths.ox.ac.uk/horawa/}
	
    }
}

\theoremstyle{plain}
\newtheorem{theoremaleph}{Theorem}

\newtheorem{conjaleph}[theoremaleph]{Conjecture}

\newtheorem{propositionaleph}[theoremaleph]{Proposition}

\newtheorem{theorem}{Theorem}[section]
\newtheorem{lemma}[theorem]{Lemma} 
\newtheorem{corollary}[theorem]{Corollary}
\newtheorem{proposition}[theorem]{Proposition}
\newtheorem{conjecture}[theorem]{Conjecture}

\theoremstyle{definition}
\newtheorem{definition}[theorem]{Definition}

\newtheorem{remark}[theorem]{Remark}
\newtheorem{example}[theorem]{Example}
\newtheorem*{notation*}{Notation}
\newtheorem{hypothesis}[theorem]{Hypothesis}
\newtheorem{hypothesisalph}{Hypothesis}

\newtheorem{hypothesisalphpr}{Hypothesis}

\numberwithin{equation}{section}

\newcommand{\mat}[4]{\left(\begin{smallmatrix}
		#1 & #2\\
		#3 & #4
	\end{smallmatrix}\right)}
\newcommand{\Mat}[4]{\left(\begin{matrix}
		#1 & #2\\
		#3 & #4
	\end{matrix}\right)}

\newcommand{\lcm}{\mathrm{lcm}}

\pgfplotsset{compat=1.16}

\begin{document}
	
	\begin{abstract}
            Let $E$ be an elliptic curve over $\Q$ and $\varrho_1, \varrho_2 \colon \Gal(H/\Q) \to \GL_2(L)$ be two odd Artin representations. We use $p$-adic methods to investigate the part of the Mordell--Weil group $E(H) \otimes L$ on which the Galois group acts via $\varrho_1 \otimes \varrho_2$. When the rank of the group is two, Darmon--Lauder--Rotger used a dominant triple product $p$-adic $L$-function to study this group, and gave an {\em Elliptic Stark Conjecture} which relates its value outside of the interpolation range to two Stark points and one Stark unit. Our paper achieves a similar goal in the rank one setting. We first generalize Hsieh's construction of a 3-variable balanced triple product $p$-adic $L$-function in order to allow Hida families with classical weight one specializations. We then give an {\em Elliptic Stark Conjecture} relating its value outside of the interpolation range to a Stark point and two Stark units. As a consequence, we give an explicit $p$-adic formula for a point which should conjecturally lie in $E(H) \otimes L$. We prove our conjecture for dihedral representations associated with the same imaginary quadratic field. This requires a generalization of the results of Bertolini--Darmon--Prasanna which we prove in the appendix. \\
            
	\end{abstract}
	
	\maketitle

{
\hypersetup{hidelinks}
\tableofcontents
}

\section{Introduction}\label{section: introduction}

Let $E$ be an elliptic curve over $\Q$ and $\varrho_1, \varrho_2 \colon \Gal(H/\Q) \to \GL_2(L)$ be two odd Artin representations such that $\det \varrho_1 \cdot \det \varrho_2 = 1$. We can consider the part of the Mordell-Weil group $E(H)$ on which the Galois group $\Gal(H/\Q)$ acts via $\varrho_1 \otimes \varrho_2$; formally:
$$E[\varrho_1 \otimes \varrho_2] = \Hom_{\Gal(H/\Q)}(\varrho_1 \otimes \varrho_2, E(H)_L).$$
An equivariant version of the Birch--Swinnerton-Dyer conjecture then predicts that the rank of $E[\varrho_1 \otimes \varrho_2]$ is equal to the order of vanishing of the $L$-function $L(\varrho_E \otimes \varrho_1 \otimes \varrho_2, s)$ at $s = 1$, and the leading term of its Taylor expansion is explicitly related to the elements of $E[\varrho_1 \otimes \varrho_2]$. The goal of this paper is to investigate the group $E[\varrho_1 \otimes \varrho_2]$ using $p$-adic methods when its rank is 1. We start with a brief summary before explaining the details.

When $E[\varrho_1 \otimes \varrho_2]$ has rank 2, the analogous question was considered by Darmon, Lauder, and Rotger:
\begin{enumerate}
    \item They prove that there is a $p$-adic $L$-function $\mathcal L_p^{\varrho_1} (\ell)$ with an interpolation property which holds for integers $\ell \geq 3$.\footnote{In general, there is a 3-variable $p$-adic $L$-function $\mathcal L_p^{\varrho_1}(k, \ell, m)$ with an interpolation property when $\ell$ is {\em dominant}. Here, we take $(k, \ell, m) = (2, \ell, 1)$ for simplicity.} It is associated with the triple $(E, \varrho_1, \varrho_2)$, but it is {\em unbalanced} --- $\varrho_1$ plays a special role in the construction.~\cite{Darmon_Rotger}
    
    \item They conjecture a special value formula
    outside of the interpolation range, the {\em rank two Elliptic Stark Conjecture}:
    \begin{equation}\label{eqn:DLR_ES}
        \mathcal L_p^{\varrho_1}(1) \sim_{L^\times} \frac{ 
 \det \begin{pmatrix} \log_{E, p}(P_{1,1}) & \log_{E, p}(P_{1,2}) \\
 \log_{E, p}(P_{2,1}) & \log_{E, p}(P_{2,2}) \end{pmatrix} }{\log_p(u_1)},
    \end{equation}
    where $P_{i,j} \in E(H)_L$ are points in the $(\varrho_1 \otimes \varrho_2)$-isotypic component and $u_1 \in \O_H^\times$ is a unit in the~$\Ad^0\varrho_1$-isotypic component. \cite[Conjecture~ES]{Darmon_Lauder_Rotger_Stark_points}
    
    \item They prove~\eqref{eqn:DLR_ES} when $\varrho_1$ and $\varrho_2$ are both induced from the same imaginary quadratic field in which $p$ splits. \cite[Theorem 3.3]{Darmon_Lauder_Rotger_Stark_points}
\end{enumerate}

When $E[\varrho_1 \otimes \varrho_2]$ has rank 1, we achieve the same goals in the present paper:
\begin{enumerate}
    \item We prove that there is a $p$-adic $L$-function $\mathcal L_p^{\bal}(\ell)$ with an interpolation property which holds for integers $\ell \geq 2$.\footnote{Again, in general, there is a 3-variable $p$-adic $L$-function $\mathcal L_p^{\bal}(k, \ell, m)$ with an interpolation property when the weights are {\em balanced}. Here, we take $(k, \ell, m) = (2, \ell, \ell)$ for simplicity.} It is associated with the triple ($E$, $\varrho_1$, $\varrho_2$), but in our case it is {\em balanced} --- $\varrho_1$ and $\varrho_2$ play equivalent roles. [Theorem~\ref{thm:A}]
    
    \item We conjecture a special value formula outside of the interpolation range, the {\em rank one Elliptic Stark Conjecture}:
    \begin{equation}\label{eqn:ES_rank1}
        \mathcal L_p^{\bal}(1) \sim_{\sqrt{L^\times}} \frac{\log_{E,p}(P)}{\log_p(u_1)^{1/2} \log_p(u_2)^{1/2}}
    \end{equation}
    where $P \in E(H)_L$ is a point in the ($\varrho_1 \otimes \varrho_2$)-isotypic component and $u_i \in \O_H^\times$ is a unit in the $\Ad^0\varrho_i$-isotypic component, for $i = 1,2$. [Conjecture~\ref{conj:C}]

    \item We prove~\eqref{eqn:ES_rank1} when $\varrho_1$ and $\varrho_2$ are both induced from the same imaginary quadratic field in which $p$ splits. [Theorem~\ref{thm:D}]
\end{enumerate}
As opposed to the rank two setting, our formula (2) can be rewritten to give a $p$-adic analytic formula for a $p$-adic point:
\begin{equation}
	\exp_{E,p}(\mathcal L_p^{\bal}(1) \cdot \log_p(u_1)^{1/2} \cdot \log_p(u_2)^{1/2}) \in E(H_p)_L
\end{equation}
conjecturally lies in $E(H)_L$. 

Bertolini, Seveso, and Venerucci have also been studying the rank one group $E[\varrho_1 \otimes \varrho_2]$ using a balanced $p$-adic $L$-function\footnote{An algorithmic study of this $p$-adic $L$-function, motivated by their work, has been carried out in the PhD thesis of the first-named author~\cite{DallAva2021PhD}.}. Their Oberwolfach report~\cite{BSV} conjectures that there is a canonical multiple $L_p^{\rm can}$ of a $p$-adic $L$-function $L_p$ whose value at $(2,1,1)$ is equal to $\log_{E, p}(P)$. They explain that special cases of this conjecture have been verified in the CM setting. They give a different expression for this $p$-adic $L$-function using an endoscopic lift to ${\rm GSp}_4$, and relate its non-vanishing at $\ell = 1$ to the non-vanishing of an appropriate Selmer class. This is the subject of their forthcoming work with Andreatta~\cite{ABSV}.

In the rest of the introduction, we will state our results precisely and explain the new ideas which allow us to achieve these goals.

\subsection{Construction of the balanced $p$-adic $L$-function}

Let $f$ be the normalized weight two modular form associated with $E$ by the Modularity Theorem, and $g$, $h$ be the normalized weight one modular forms associated with $\varrho_1$, $\varrho_2$, respectively, by the work of Khare and Wintenberger~\cite{Khare_Wintenberger_I, Khare_Wintenberger_II}. The analytic properties of the $L$-function we are interested in are understood in terms of the Garrett triple product $L$-function associated with $f$, $g$, and $h$:
$$L(\varrho_E \otimes \varrho_1 \otimes \varrho_2, s) = L(f \times g \times h, s).$$
Then the completed $L$-function $\Lambda(f \times g \times h, s)$ satisfies a functional equation relating the value at $s$ to the value at $2-s$. The assumption that $\det \varrho_1 \det \varrho_2 = 1$ implies that the global root number $\epsilon(f \times g \times h)$ is $\pm 1$. Since we are interested in the rank one setting, we will assume that:
$$\epsilon(f \times g \times h) = -1$$
(see Hypothesis~\ref{hyp:epsilon-1}). Since the global root number is the product of local root numbers
$$\epsilon(f \times g \times h) = \prod_v \epsilon_v(f \times g \times h),$$
and $\epsilon_\infty(f \times g \times h) = 1$, the assumption amounts to the fact that there is an odd number of primes $v$ such that $\epsilon_v(f \times g \times h) = -1$.

Let $p \geq 5$ be a prime number. We assume that $f$, $g$, $h$ are ordinary and regular at $p$ (see Hypothesis~\ref{hyp:classicality}), and consider Hida families $\mathbf g$, $\mathbf h$ whose weight one specializations are fixed $p$-stabilizations $g_\alpha$, $h_\alpha$ of $g$ and $h$, respectively. Letting $g_\ell$, $h_m$ be the weight $\ell$, $m$ specializations of $\mathbf g$, $\mathbf h$, respectively, we note that 
$$\epsilon(f \times g_\ell \times h_m) = \prod_{v} \epsilon(f \times g_\ell \times h_m) = \prod_{v < \infty}  \epsilon_v(f \times g_\ell \times h_m) \cdot \epsilon_\infty(f \times g_\ell \times h_m) = -\epsilon_\infty(f \times g_\ell \times h_m)$$
by rigidity of automorphic types. In particular, if the weights $(2, \ell, m)$ are {\em balanced}, i.e.\ $2 \leq \ell < 2 + m$, $2 \leq m < 2 + \ell$, and $\ell + m \equiv 0 \pmod 2$, we have that
$$\epsilon(f \times g_\ell \times h_m) = +1,$$
and hence we expect that generically the central $L$-value is non-vanishing:
$$L(f \times g_\ell \times h_m, (\ell + m)/2) \neq 0,$$
and there should be a $p$-adic $L$-function interpolating these values divided by appropriate periods. In comparison, Darmon, Lauder, and Rotger start with $\epsilon(f \times g \times h) = +1$ in the $f$-dominant region and hence obtain $\epsilon(f \times g_\ell \times h_m) = +1$ in the $g$-dominant region, i.e.\ for $\ell > m + 2$.

There are two known constructions of $p$-adic $L$-functions interpolating the central critical $L$-values in the balanced region for three Hida families $\mathbf f$, $\mathbf g$, $\mathbf h$. The first construction was due to Greenberg and Seveso~\cite{GS2019}, and the second more explicit construction was given by Hsieh~\cite{Hsieh2021}. The difference between the constructions was outlined in the latter paper and it seems that Hsieh's approach is more suitable for arithmetic applications, such as the one in the present paper.

However, as observed by the first-named author in~\cite[Section 4.4.4]{DallAva2021Approx}, the ramification assumption~(3) in \cite[Theorem B]{Hsieh2021} implies that the Hida families $\mathbf g$ and $\mathbf h$ cannot both have classical weight one specializations. In particular, as far as we know, there is no suitable balanced triple product $p$-adic $L$-function to study classical points with weights $(2,1,1)$.

Therefore, the first goal of our paper is to extend Hsieh's construction to a setting where the Hida families $\mathbf g$ and $\mathbf h$ do have classical weight one specializations. The final assumption we make (see Hypothesis~\ref{hyp:Ccond_at_most_2}) is:
$$\epsilon_q(f \times g \times h) = -1 \text{ implies that } v_q(N_i) \leq 2,$$
where  $N_1$, $N_2$, $N_3$ are the levels of $f$, $g$, $h$, respectively.
This weakens Hsieh's assumption that $v_q(N_i) = 1$.

Fix an algebraic closure $\overline \Q$ of $\Q$ and field embeddings of $\overline \Q$ into $\C$ and $\C_p$.

\begin{theoremaleph}[Corollary~\ref{cor:interpolation}]\label{thm:A}
    Under the above hypotheses, there exist open admissible neighborhoods $\mathcal U_g$ and $\mathcal U_h$ of the classical weight $1$ in the weight space, and a (square root) balanced triple product $p$-adic $L$-function $\mathcal L_p^{\bal}\colon \mathcal U_g \times \mathcal U_h \to \C_p$, associated with $f$ and the Hida families $\mathbf g$ and $\mathbf h$. It has the following interpolation property for integers $\ell \in \mathcal U_g$ and $m \in \mathcal U_h$ such that $\ell, m \geq 2$ and $2 - \ell \leq m \leq 2+\ell$:
	\begin{align*}
		\mathcal L_p^{\bal}(\ell, m)^2 & = \frac{\Lambda(f  \times g_\ell \times h_m, c)}{\langle f, f \rangle \langle g_\ell, g_\ell \rangle \langle h_m, h_m \rangle} \cdot \frac{ \mathcal E_p^{\bal}(f \times g_\ell \times h_m)}{ \mathcal E_p(g_\ell, \Ad) \cdot \mathcal E_p(h_m, \Ad) } \in \overline \Q,
	\end{align*}
	where $c = (\ell + m)/2$ is the center of the functional equation, and
	\begin{align*}
		\mathcal E_p^{\bal}(f \times g_\ell \times h_m) & = (1 - \alpha_f \beta_g \beta_h p^{-c})(1- \beta_f \alpha_g \beta_h p^{-c}) (1 - \beta_f \beta_g \alpha_h p^{-c}) (1 - \beta_f \beta_g \beta_h p^{-c}), \\
		\mathcal E_p(F_k, \Ad) & =  (1 - \alpha_{F}^{-2} \chi_{F_k}(p) p^{k -1})(1 - \alpha_{F}^{-2} \chi_{F_k}(p) p^{k - 2}) & F_k \in \{g_\ell, h_m\}
	\end{align*}
	are as in~\cite[(1.2)]{Hsieh2021} and~\cite[p.\ 416]{Hsieh2021}, respectively.
\end{theoremaleph}

A more general square root balanced triple product $p$-adic $L$-function for three Hida families $\mathbf f$, $\mathbf g$, $\mathbf h$ is constructed in Theorem~\ref{thm:interpolation}.

Note that the interpolation property for our triple product $p$-adic $L$-function differs from Hsieh~\cite[Theorem B]{Hsieh2021}. The three key differences (which we will expand on momentarily) are:
\begin{itemize}
	\item Hsieh's $p$-adic $L$-function is defined on the entire weight space.
	\item Our interpolation property does not involve Gross periods, but rather just the Petersson norms of the three forms $f$, $g_\ell$, and $h_m$.
	\item The Euler factor at $p$ in Hsieh's $p$-adic $L$-function is only $\mathcal E_p^{\bal}(f \times g_\ell \times h_m)$ instead of the above quotient. Indeed, the adjoint Euler factors above are absorbed by the Gross periods.
\end{itemize}
However, it seems that our $p$-adic $L$-function is the correct one for the eventual arithmetic application in the present work. For example, both the periods and the Euler factor at $p$ closely resemble those in~\cite{Darmon_Rotger} in the unbalanced setting. Moreover, the proof of Theorem~\ref{thm:D} below also only seems to work with this interpolation property. We expand on the differences in Remark~\ref{rmk:difference_to_Hsieh}.

Next, we explain the technical novelty which allows us to loosen the ramification assumption in Hsieh. Ichino's formula~\cite{Ichino2008} for the central value of the triple product $p$-adic $L$-function is on the quaternion algebra $D$ ramified at $v$ such that $\epsilon_v(f \times g_\ell \times h_m) = -1$, and utilizes the Jacquet--Langlands transfers $f^D$, $g^D_\ell$, $h^D_m$ of the three modular forms $f$, $g_\ell$, $h_m$. Under our assumptions, $D$ is ramified at $v = \infty$ and at an odd number of finite primes $q$. 

Let $q$ be an odd prime. As observed by Pizer~\cite{Pizer1980}, if $D$ is ramified at $q$ and $f$ is a (twist-minimal) new cusp form of weight $k \geq 2$, level $q^2$, character $\chi$ of conductor at most $q$, then:
$$\dim S_k^D(q^2, \chi)[f] = 2,$$
i.e.\ there is a two-dimensional space of quaternionic modular forms for $D^\times$ which transfer to the same modular form $f$ (cf.\ Proposition~\ref{prop: multiplicity-2}). This raises two questions:
\begin{enumerate}
	\item For each $(\ell, m)$, can we choose vectors $f^D$, $g^D_\ell$, $h^D_m$ associated with $f$, $g$, $h$ and compute the associated local integrals in Ichino's integral representation?
	\item Do there exist Hida families $\mathbf g^D$, $\mathbf h^D$ associated with the Hida families $\mathbf g$, $\mathbf h$ such that the specializations $g_\ell^D$, $h_m^D$ recover the choices in question (1)?
\end{enumerate}

We give a positive answer to these questions by introducing ``extra Hecke operators'' which recover multiplicity one in the Jacquet--Langlands correspondence.

\begin{propositionaleph}[Proposition~\ref{prop: mult-1 with varpi_d}]\label{prop:B}
There exists an operator $\langle \varpi_{D_q}\rangle$ on $\dim S_k^D(q^2, \chi)$ associated with a choice of local uniformizer $\varpi_{D_q}$ of $D_q$, commuting with the Hecke operators. For each twist-minimal new cusp form $f$, of weight $k \geq 2$, level $q^2$, and character $\chi$ of $\cond(\chi)\leq q$:
	$$S_k^D(q^2, \chi)[f] \iso \underbrace{S_k^D(q^2, \chi)[f]^{\langle\varpi_{D_q}\rangle = + \sqrt{\chi_q(q)}}}_{\dim = 1} \oplus \underbrace{S_k^D(q^2, \chi)[f]^{\langle\varpi_{D_q}\rangle = - \sqrt{\chi_q(q)}}}_{\dim = 1}.$$
    Here, $\chi_q$ is the local component at $q$ of the adèlization of $\chi$.
\end{propositionaleph}

To answer question (1), we can choose a vector in either of the one-dimensional eigenspaces. Section~\ref{Section: Local JL correspondence and test vectors} computes the necessary local integrals in Ichino's formula. Answering question~(2) is then also based on Proposition~\ref{prop:B}, although it is technical and occupies Sections~\ref{Section: Global JL correspondence and test vectors} and~\ref{Section: The JL correspondence in families}. We show that for each choice of sign $+1$ or $-1$, there is a one-dimensional space of Hida families on $D^\times$ associated with the Hida families on $\GL_2$, strengthening the results in \cite{DallAva2021Hida}. Theorem~\ref{thm: multiplicity-1 for lifts of Coleman families} proves this statement also in the case of Coleman families. From it, we deduce a control theorem à la Hida, which is the content of Corollary~\ref{cor: Hida's Control Theorem}. The proof of these results relies on the construction of auxiliary quaternionic eigenvarieties endowed with closed immersions in the classical cuspidal one. In order to construct an explicit $p$-adic $L$-function, we need to identify quaternionic Hida families with (a slight generalization of the) $\Lambda$-adic forms in~\cite{Hsieh2021}; this is done in Proposition~\ref{prop: identification chenevier families and hsieh families}.

Note that our choice of vectors on the quaternion algebra $D^\times$ is only well-defined up to scalars. We also restrict ourselves to working with an admissible affinoid in the weight space, instead of proving a control theorem over the whole weight space. It would be interesting to generalize our control theorem to the whole weight space and answer the natural questions about Gross periods for the different Hida families associated with the choice of eigenvalues of the operators $\langle\varpi_{D_q}\rangle$. This could lead to a definition of an integral $p$-adic $L$-function analogous to Hsieh's.

\subsection{The Elliptic Stark Conjecture}

Having defined the balanced triple product $p$-adic $L$-function $\mathcal L_p^{\bal}$, we turn to studying its value at the {\em BSD point} $(\ell, m) = (1,1)$ which lies outside of the interpolation range. We start by introducing the relevant arithmetic objects.

Let us briefly recall our setup: the triple $(f, g, h)$ corresponds to a triple $(E, \varrho_1, \varrho_2)$ of an elliptic curve $E$ over $\Q$, and two Artin representation $\varrho_i \colon \Gal(H/\Q) \to \GL(V_i)$ for two-dimensional $L$-vector spaces $V_1$ and $V_2$. Recall that:
$$E[V_1 \otimes V_2] = \Hom_{L[\Gal(H/\Q)]}(V_1 \otimes V_2, E(H) \otimes_\Q L),$$
and under the rank one assumption, we may choose a basis:
$$\Phi \colon V_1 \otimes V_2 \to E(H) \otimes_\Q L,$$ 
well-defined up to $L^\times$. 

We fix an embedding $H \hookrightarrow \Q_p^{\rm ur}$ and let $\sigma_p \in \Gal(H/\Q)$ be the associated Frobenius element. Under our classicality and regularity assumptions, we have that $\varrho_i(\sigma_p) = \begin{pmatrix}
	\alpha_i & 0 \\ 
	0 & \beta_i
\end{pmatrix}$ with $\alpha_i \neq \beta_i$. Letting $V_i^\alpha$ be the $\alpha_i$-eigenspace of $V_i$, we consider a non-zero element in the one-dimensional subspace $v_{\alpha \alpha} \in V^{\alpha \alpha} = V_1^{\alpha} \otimes V_2^{\alpha} \subseteq V_1 \otimes V_2$. Finally, we let
$$P_{\alpha \alpha} = \Phi(v_{\alpha \alpha}) \in E(H) \otimes L,$$
which is well-defined up to $L^\times$, but depends on the choice of Frobenius eigenvalues $\alpha_1$ and $\alpha_2$.

Next, we consider the adjoint representations $\Ad^0(\varrho_i) = \Hom^0(V_i, V_i)$. By \cite[Proposition\ 1.5]{Darmon_Lauder_Rotger_Stark_points}, the Stark unit group
$$U_H[\Ad^0 \varrho_i] = \Hom_{L[\Gal(H/\Q)]}( \Ad^0 V_i, \O_H^\times \otimes L)$$
is also of rank one, and we may choose a basis:
$$\Psi_i \colon \Ad^0 V_i \to \O_H^\times \otimes L,$$
well-defined up to $L^\times$. 

The eigenvalues of $\Ad^0\varrho_i(\sigma_p)$ are $1$, $\frac{\alpha_i}{\beta_i}$, $\frac{\beta_i}{\alpha_i}$, and we fix eigenvectors in the $\alpha_i/\beta_i$-eigenspaces:
$$v_i^{\alpha/\beta} \in (\Ad^0 V_i)^{\alpha_i/\beta_i}.$$
Finally, this defines
$$u_{i, \alpha} = \Psi_i(v_i^{\alpha/\beta}) \in \O_H^\times \otimes L,$$
which is well-defined up to $L^\times$, but depends on the choice of Frobenius eigenvalue $\alpha_i$.

Write $r(E, \varrho_{gh})$ for the rank of $E[\varrho_g \otimes \varrho_h]$ which is, conjecturally on BSD, equal to ${\rm ord}_{s=1}L(f \times g \times h, s)$.

\begin{conjaleph}[Rank one Elliptic Stark Conjecture~\ref{conj:ES}]\label{conj:C}
	If $r(E, \varrho_{gh}) > 1$, then $\mathcal L_p^{\bal}(1,1) = 0$. If $r(E, \varrho_{gh}) = 1$, then:
	$$\mathcal L_p^{\bal}(1,1)^2 \sim_{L^\times} \frac{\log_{E, p}( P_{\alpha \alpha} )^2}{\log_p(u_{1, \alpha}) \log_p(u_{2, \alpha})},$$
	where
	\begin{align*}
		\log_p \colon (\O_{H}^\times)_L & \to H_p \otimes L & \text{$p$-adic logarithm}, \\
		\log_{E, p} \colon  E(H)_L & \to H_p \otimes L & \text{$p$-adic formal group logarithm for $E$}.
	\end{align*}
\end{conjaleph}

Both sides of the equality depend on the same choice of $\alpha_1$ and $\alpha_2$. At first glance, it seems that the right hand side depends on $H$ and the choices associated with it, but a simple computation reveals that it does not --- see Remark~\ref{rmk:sanity_check}.

Finally, we prove that this conjecture is true if $\varrho_1$, $\varrho_2$ are both induced from characters of the same imaginary quadratic field. Let $K$ be an an imaginary quadratic field and $p$ be a prime which splits in $K$. Let $\psi_g$, $\psi_h$ be finite order Hecke characters of $K$ such that $(\psi_g \psi_h) \circ N_{K/\Q} = 1$, and consider the theta series $g = \theta_{\psi_g}$, $h = \theta_{\psi_h}$. Let $\psi_1 = \psi_g \psi_h$ and $\psi_2 = \psi_g \psi_h^{\sigma}$. We introduce some explicit assumptions under which our running hypotheses hold:
\begin{enumerate}
	\item $\epsilon_q(f_K, \psi_2) = +1$ for all finite primes $q$ of $K$,
	\item $g$ and $h$ are ordinary and regular,
	\item if $\epsilon_q(f \times g \times h) = -1$, then $v_q(N_1) = 1$ and $v_q(N_2) = v_q(N_3) = 2$; 
\end{enumerate}

\begin{theoremaleph}[Theorem~\ref{thm:CM}]\label{thm:D}
	Under assumptions \emph{(1)--(3)} above, Conjecture~\ref{conj:C} is true.
\end{theoremaleph}

The idea of the proof of this theorem is to factor the triple product $p$-adic $L$-function into Rankin--Selberg and Katz $p$-adic $L$-functions (Theorem~\ref{thm:CM_factorization}). Under our ramification hypotheses, we need a certain generalization of the result of Bertolini--Darmon--Prasanna~\cite{BDP} on Rankin--Selberg $p$-adic $L$-functions. Indeed, their result holds under the Heegner hypothesis, which is not satisfied under assumption (3) above. We loosen the Heegner hypothesis slightly to include the case of interest to us. We prove this result in Appendix~\ref{appendix BDP}.

In Appendix~\ref{section: appendix digest of examples}, we provide a digest of examples to which our conjecture and theorem apply. In future work, we hope to give numerical evidence for Conjecture~\ref{conj:C}, based on the algorithms in~\cite{DallAva2021Approx}. 

\subsection{Organization of the paper}

The main technical innovation of the paper is the study of extra operators on the space of quaternionic modular forms, which facilitate our choice of test vectors for the $p$-adic $L$-function. This occupies the first three sections, which consider the local (Section~\ref{Section: Local JL correspondence and test vectors}), global (Section~\ref{Section: Global JL correspondence and test vectors}), and $p$-adic (Section~\ref{Section: The JL correspondence in families}) Jacquet--Langlands transfers. A reader interested in the construction of the $p$-adic $L$-function may proceed directly to Section~\ref{Section: Balanced triple product $p$-adic $L$-function}, where we use this choice of test vectors to generalize Hsieh's construction. The arithmetic applications are discussed in Section~\ref{section: Application: Elliptic Stark Conjecture in rank one}, which may also be read independently of all the previous sections. Appendix~\ref{appendix BDP} gives generalizations of results of Bertolini--Darmon--Prasanna which are used the proof of Theorem~\ref{thm:D}, while Appendix~\ref{section: appendix digest of examples} contains examples.

\subsection*{Acknowledgments}

We would like to thank Alan Lauder, Matteo Longo, Kimball Martin, James Newton, Kartik Prasanna, and John Voight for many helpful discussions. This work represents a natural continuation of the doctoral work of LD and he is grateful to his advisor Massimo Bertolini for introducing him to the topic of balanced $p$-adic $L$-functions and setting him out on this path. LD is also thankful to Kimball Martin for bringing his attention to the reference~\cite{GrossPrasad1991}. Many of the ideas presented here are inspired by the work of Henri Darmon, Alan Lauder, and Victor Rotger. Their numerical investigation which led to such a precise statement of their rank two {\em Elliptic Stark Conjecture} paved the way for our rank one version of it. In particular, AH is thankful to Victor Rotger for his \href{https://people.maths.ox.ac.uk/horawa/Iwasawa2019-Rotger.pdf}{mini-course} at Iwasawa 2019 in Bordeaux, which inspired him to consider the question in the present work. 

AH was supported by the NSF grant DMS-2001293 and UK Research and Innovation grant MR/V021931/1. LD is grateful to the Università degli Studi di Padova (Research Grant funded by PRIN 2017 “Geometric, algebraic and analytic methods in arithmetic”) and Università degli Studi di Milano (Research fellowship “Motivi, regolatori e geometria aritmetica”) for their financial support.

\section{Local JL correspondence and test vectors}\label{Section: Local JL correspondence and test vectors}

In this section, we consider a finite extension $F$ of $\Q_\ell$ for $\ell \neq 2$ (which will eventually be $\Q_\ell$) and three local representations $\pi_1$, $\pi_2$, $\pi_3$ of $\GL_2(F)$ such that the product of their central characters is trivial. 

We classify when $\epsilon(\pi_1 \otimes \pi_2 \otimes \pi_3) = -1$ and in all cases when $c(\pi_1), c(\pi_2), c(\pi_3) \leq 2$, we compute the relevant local triple product integrals on a definite quaternion algebra over $F$. The new contribution of this section is the computations of these local integrals when $c(\pi_i) = 2$ for some $i$.

Throughout this section, we often work with Weil--Deligne representations $\sigma_i = \sigma(\pi_i)$ associated with $\pi_i$ via the Local Langlands Correspondence~\cite{Bushnell_Henniart}. Their triple product $\epsilon$- and $L$-factors agree with the automorphic ones~\cite[Proposition 2.1]{Harris_Kudla_on_a_conj}. 

For completeness, we briefly recall the explicit Local Langlands Correspondence~\cite[Chapter\ 8]{Bushnell_Henniart} for $\ell \neq 2$. For a character $\chi$ of $K^\times$, for a local field $K$, we write $\xi = \xi(\chi)$ for the associated character of $W_K$ via Class Field Theory.

\begin{enumerate}
    \item If $\pi = \pi(\chi_1, \chi_2)$ is a principal series representation, then $\sigma(\pi) = \xi(\chi_1) \oplus \xi(\chi_2)$ is reducible.
    \item If $\pi = \St \otimes \chi$ is a twist of the Steinberg representation, then $\sigma(\pi) = \Sp(2) \otimes \xi(\chi)$ is a twist of the special representation.
    \item If $\pi = \pi_{\psi}$ is supercuspidal, associated with an {\em admissible pair} $(K, \psi)$, where $K/F$ is a quadratic extension and $\psi$ is a character of $K^\times$, then $\sigma(\pi) = \Ind_{W_K}^{W_F} (\xi(\psi) \Delta_{\psi}^{-1})$ for a character $\Delta_{\psi}$ defined in \cite[Section 34.4]{Bushnell_Henniart}. Then $c(\pi) = 2$ if and only if $\pi$ has {\em depth 0} (or {\em level 0}), i.e.\ $K$ is unramified over $F$ and $\psi$ has conductor 1 ({\em level 0}). In this case $\Delta_{\psi}$ is the unramified quadratic character of $K^\times$.
\end{enumerate}

The representations in (2) and (3) are called {\em discrete series}. 

\subsection{Local $\epsilon$-factors and $L$-factors}

The following proposition classifies all cases when $\epsilon(\sigma_1 \otimes \sigma_2 \otimes \sigma_3) = -1$. 

\begin{proposition}[Prasad]\label{prop:epsilon=-1}
    Let $\sigma_1, \sigma_2, \sigma_3$ be three Weil--Deligne representations of $W_F$ such that the product of their determinants is trivial. Then $\epsilon(\sigma_1 \otimes \sigma_2 \otimes \sigma_3) = -1$ if and only if there is a reordering $\{\sigma_1', \sigma_2', \sigma_3'\} = \{\sigma_1, \sigma_2, \sigma_3\}$ such that one of the following holds:
    \begin{enumerate}
        \item $\sigma_i' \iso \Sp(2) \otimes \det(\sigma_i')$ for $i = 1, 2, 3$;
        \item $\sigma_1' \iso \Sp(2) \otimes \det(\sigma_1')$, $\sigma_2'$ is irreducible and $\sigma_3' \iso (\sigma_2')^\vee \otimes \det(\sigma_1')^{-1}$;
        \item there is a quadratic extension $K/F$ and characters $\xi_1$, $\xi_2$, $\xi_3$ such that $\sigma_i' = \Ind_{W_K}^{W_F} \xi_i$ is irreducible (so $\xi_i^\sigma \neq \xi_i$) and either $\xi_1\xi_2\xi_3 = 1$ or $\xi_1 \xi_2 \xi_3^\sigma = 1$ or $\xi_1\xi_2^{\sigma}\xi_3 = 1$ or $\xi_1\xi_2^\sigma\xi_3^\sigma = 1$.
    \end{enumerate}
\end{proposition}
\begin{proof}
    These results are contained in~\cite[Section 8]{Prasad1990}. See \cite[Proposition 8.6]{Prasad1990} for (1) and \cite[Proposition 8.5]{Prasad1990}~(2). For completeness, we compute $\epsilon(\sigma_1 \otimes \sigma_2 \otimes \sigma_3)$ for $\sigma_i = \Ind_{W_K}^{W_F} \xi_i$; the other possibilities are dealt with similarly. First, note that:
    \begin{align*}
        \sigma_1 \otimes \sigma_2 \otimes \sigma_3 & \iso \left(\Ind_{W_K}^{W_F}(\xi_1 \xi_2) \otimes \sigma_3\right) \oplus  \left(\Ind_{W_K}^{W_F}(\xi_1 \xi_2^\sigma) \otimes \sigma_3 \right)
    \end{align*}
    and hence by standard properties of $\epsilon$-factors:
    \begin{align*}
        \epsilon(\sigma_1 \otimes \sigma_2 \otimes \sigma_3) & = \epsilon\left(\Ind_{W_K}^{W_F}(\xi_1 \xi_2) \otimes \sigma_3 \right) \epsilon\left(\Ind_{W_K}^{W_F}(\xi_1 \xi_2^\sigma) \otimes \sigma_3 \right) \\
        & = \xi_{K/F}(-1) \epsilon(\sigma_3|_{W_K} \otimes \xi_1 \xi_2) \cdot  \xi_{K/F}(-1) \epsilon(\sigma_3|_{W_K} \otimes \xi_1 \xi_2^\sigma) \\
        & = \epsilon(\sigma_3|_{W_K} \otimes \xi_1 \xi_2) \cdot \epsilon(\sigma_3|_{W_K} \otimes \xi_1 \xi_2^\sigma) \\
        & = \det(\sigma_3)(-1) \epsilon(\sigma_3|_{W_K} \otimes \xi_1 \xi_2) \cdot  \det(\sigma_3)(-1) \epsilon(\sigma_3|_{W_K} \otimes \xi_1 \xi_2^\sigma),
    \end{align*}
    where $\xi_{K/F}$ is the quadratic character of $W_F$ associated with $K/F$. Tunnell's Theorem~\cite[Theorem 8.2]{Prasad1990} shows that $\det(\sigma_3)(-1) \epsilon(\sigma_3|_{W_K} \otimes \xi_1 \xi_2) = -1$ if and only if $(\xi_1 \xi_2)^{-1} = \xi_3$ or $(\xi_1 \xi_2)^{-1} = \xi_3^{\sigma}$. Therefore, $\epsilon(\sigma_1 \otimes \sigma_2 \otimes \sigma_3) = -1$ if and only if condition (3) holds after reordering.
\end{proof}

We make the following symplifying hypothesis.

\begin{hypothesis}\label{hyp:cond_at_most_2}
    For all $i$, the conductor $c(\pi_i)$ of $\pi_i$ is at most 2.
\end{hypothesis}

In particular, if $\sigma = \Ind_{W_K}^{W_F} \chi$, then $K/F$ is unramified, the conductor $c(\chi)$ of $\chi$ is 1, and $\chi^{\sigma} \neq \chi$. 

Next, we record the $L$-factors in each of the cases when $\epsilon(\sigma_1 \otimes \sigma_2 \otimes \sigma_3) = -1$.

\begin{proposition}\label{prop:local_L-factors_triple}
    We have that:
    $$L(\sigma_1 \otimes \sigma_2 \otimes \sigma_3, s) = \begin{cases}
        \zeta_F(s+3) \zeta_F(s+2)^2 & \text{case (1)} \\
        \zeta_F(2s+2) & \text{case (2)} \\
        \zeta_F(2s) & \text{case (3)}
    \end{cases}$$
\end{proposition}
\begin{proof}
    In case (1), note that:
    $$\Sp(2) \otimes \Sp(2) \otimes \Sp(2) \iso \Sp(4) \oplus \Sp(2)|\cdot| \oplus \Sp(2)|\cdot|$$
    and $L(\Sp(m), s) = \zeta_\ell(s+m-1)$. Part (2) follows from
    $$L(\Sp(2) \otimes \sigma \otimes \sigma^\vee, s) = L(\sigma \otimes \sigma^\vee, s+1)$$
    and \cite[Corollary 1.3]{Gelbart_Jacquet}, because $\sigma \iso \sigma \otimes \eta$ for an unramified quadratic character~$\eta$. For (3), note that:
    $$\sigma_1 \otimes \sigma_2 \otimes \sigma_3 \iso \Ind_{W_K}^{W_F}(\chi_1 \chi_2 \chi_3) \oplus \Ind_{W_K}^{W_F}(\chi_1 \chi_2 \chi_3^\sigma) \oplus  \Ind_{W_K}^{W_F}(\chi_1 \chi_2^\sigma \chi_3) \oplus \Ind_{W_K}^{W_F}(\chi_1 \chi_2 \chi_3^\sigma).$$
    Under the assumption that $\epsilon(\sigma_1 \otimes \sigma_3 \otimes \sigma_3) = -1$, exactly one of these characters is trivial and the other characters are ramified. This shows that:
    $$L(\sigma_1 \otimes \sigma_2 \otimes \sigma_3) = \zeta_{K}(s) = \zeta_F(2s),$$
    because $K/F$ is unramified.
\end{proof}

We will also need the adjoint $L$-factors.
    
\begin{proposition}\label{prop:local_L-factors_adjoint}
    We have that:
    \begin{align*}
        \Ad(\Sp(2) \otimes \xi) & \iso \Sp(3) |\cdot|, \\
        \Ad( \Ind_{W_K}^{W_F}(\chi) ) & \iso \Ind_{W_K}^{W_F}(\chi/\chi^\sigma) \oplus \chi_{K/F},
    \end{align*}
    and hence:
    \begin{align*}
        L(\Ad(\Sp(2) \otimes \xi), s) & = \zeta_F(s + 1), \\
        L(\Ad( \Ind_{W_K}^{W_F} (\xi) )) & = L(\chi_{K/F}, s) = \frac{\zeta_F(2s)}{\zeta_F(s)}.
    \end{align*}
\end{proposition}

\subsection{Local Jacquet--Langlands correspondence}

In cases when $\epsilon(\pi_1 \otimes \pi_2 \otimes \pi_3) = -1$, the local integral should be non-vanishing on the non-split quaternion algebra $D$ over $F$ (cf.~\cite[Theorem 1.2]{Prasad1990}). In this section, we briefly summarize the local Jacquet--Langlands correspondence following~\cite[Chapter 13]{Bushnell_Henniart}.

The local Jacquet--Langlands correspondence is a bijection 
$$\mathrm{JL} \colon \mathrm{Rep}^{\mathrm{ds}}(\GL_2(F)) \to  \mathrm{Rep}(D^\times)$$ 
between discrete series representations of $\GL_2(F)$ and irreducible smooth admissible representations of $D^\times$. For $\pi = \St \otimes \chi$, we have that:
$$\mathrm{JL}(\St(2) \otimes \chi) = \chi \circ \mathrm{\nu},$$
where $\mathrm{\nu} \colon D^\times \to F^\times$ is the reduced norm.

For supercuspidal representations, an explicit description of the correspondence is given in \cite[Section 56]{Bushnell_Henniart}. We only describe it here for supercuspidal representations of conductor two, i.e.\ $\pi = \pi_{\psi}$ for the unramified quadratic extension $K$ of $F$ and character $\psi$ of $K^\times$ of conductor 1 such that $\psi^\sigma \neq \psi$. 

Recall that $\Delta_{\psi}$ is the unramified quadratic character of $K^\times$ in this case. We then have that:
\begin{itemize}
    \item $\pi_\psi \otimes \Delta_{\psi} \iso \pi_\psi$,
    \item under the local Langlands correspondence, the corresponding representation of $W_F$ is $\Ind_{W_K}^{W_F} (\xi(\psi) \Delta_{\psi})$.
\end{itemize}

We describe the representation 
$$\pi_\psi^D = \mathrm{JL}(\pi_\psi)$$
explicitly. Let $\O_D$ be a maximal order of $D$ and $\varpi_D \in \O_D$ be a uniformizer. We will assume that $\varpi_D^2 = \varpi_F$ for a uniformizer $\varpi_F$ of $F$. There is a filtration on $\O_D^\times$
\begin{equation}\label{Eq.: definition U_D^a}
    U_D^a = 
    \begin{cases}
        \O_D^\times & a = 0 \\
        1 + \varpi_D^a \O_D & a \geq 1. \\
    \end{cases}
\end{equation}
Given an unramified character $\psi$ of $K^\times$, we may extend it to a character $\Psi$ of $K^\times U^1_D$ by letting $U^1_D$ act trivially. Then:
$$\pi_\psi^D \iso {\textnormal{c-Ind}}_{ F^\times U^1_D }^{D^\times} \Psi.$$

Let $\pi$ be a smooth irreducible representation of $D^\times$, which is automatically finite-dimensional because $D^\times/F^\times$ is compact. Thus $\pi|_{U_D^a} = 1$ for $a \gg 0$. We define the {\em conductor} $c(\pi)$ of $\pi$ to be $a + 1$ where $a$ is the smallest integer such that $\pi|_{U_D^a} = 1$. Note that under the assumption that $\psi$ has conductor 1, $\pi_{\psi}^D$ has conductor two. For completess, we verify that all representations of $D^\times$ of conductor two are obtained this way, following~\cite{Carayol1984}. 

Note that $\varpi_K = \varpi_F$ because $K/F$ is an unramfied quadratic extension, and $k = \O_K/\varpi_K \O_K$ is a quadratic extension of $f = \O_F/\varpi_F \O_F$. Finally, since $d = \O_D/\varpi \O_D$ is also a quadratic extensions of $f$, it is isomorphic to $k$. By definition, a representation of conductor two factors through
\begin{equation}\label{Eq.: semidirect product local order}
    D^\times/U^1_D \iso d^\times \rtimes \langle \varpi_D \rangle.
\end{equation}
Note that $\varpi_D \O_D \varpi_D^{-1} = \O_D$ because the maximal order $\O_D \subseteq D$ is unique, and hence conjugation by $\varpi_D$ preserves $d^\times$. Moreover, $\varpi_D^2 = \varpi_F$ is in the center $F^\times$ of $D^\times$ and hence acts trivially. In particular, we have a subgroup:
$d^\times \times \langle \varpi_F \rangle \subseteq d^\times \rtimes \langle \varpi_D \rangle$.

Next, observe that a character $\psi$ of $K^\times$ which is trivial on $U^1_D$ corresponds precisely to a character of $k^\times \times \langle \varpi_F \rangle$. Therefore, we may identify $\pi_\psi^D$ with the inflation of the induction $\Ind_{d^\times \times \langle \varpi_F \rangle}^{d^\times \rtimes \langle \pi_D \rangle} \psi$. This representation is reducible unless $\psi^{\varpi_D} \neq \psi$. Altogether, we get the following result.

\begin{proposition}\label{prop:reps_of_cond2}
    Suppose $\pi$ is an smooth irreducible representation of $D^\times$ and that $c(\pi) = 2$. Then there is a character $\psi$ of $d^\times \times \langle \varpi_F \rangle$ with $\psi^{\varpi_D} \neq \psi$ such that:
    $$\pi = \mathrm{Inf}_{D^\times/U^1_D}^{D^\times} \Ind_{d^\times \times \langle \varpi_F \rangle}^{d^\times \rtimes \langle \varpi_D \rangle} \psi.$$
    In particular, $\pi$ is two-dimensional and in the basis corresponding to the decomposition above, we have:
    \begin{align*}
        \pi(x) & = \begin{pmatrix}
        \psi(x) & \\
        & \psi^{\varpi_D}(x) 
        \end{pmatrix}, & x \in d^\times \times \langle \varpi_F \rangle, \\
        \pi(\varpi_D) & = \begin{pmatrix}  & 1 \\
        \psi(\varpi_F) &  \end{pmatrix}.
    \end{align*}
    Finally, if $\omega$ is the central character of $\pi$, then $\psi|_{F^\times} = \omega$ and the last equality may be written:
    $$\pi(\varpi_D) = \begin{pmatrix}  & 1 \\
        \omega(\varpi_F) &  \end{pmatrix}.$$
\end{proposition}

\begin{proof}
    The first claim then follows from \cite[Section 5.1]{Carayol1984}. The rest of the proposition is immediate.
\end{proof}

We identify the representation $\pi_\psi^D$ with the one described explicitly in Proposition~\ref{prop:reps_of_cond2}.

\subsection{Local integral for zero, two, and three supercuspidal representations of level $\ell^2$}

Let $\pi_1, \pi_2, \pi_3$ be irreducible admissible representations of $\GL_2(F)$ with central characters $\omega_i$ which satisfy $\omega_1 \omega_2 \omega_3 = 1$. Let $\sigma_1, \sigma_2, \sigma_3$ be corresponding representations of $W_F$ such that $\det(\sigma_1) \det(\sigma_2) \det(\sigma_3) = 1$. As above, let $D$ be the non-split quaternion algebra over $F$, and $\pi_i^D$ be the Jacquet--Langlands transfer of $\pi_i$ to $D^\times$. 

Prasad~\cite[Theorem 1.4]{Prasad1990} proves that there exists a non-zero trilinear form on $\pi_1^D \otimes \pi_2^D \otimes \pi_3^D$ if and only if $\epsilon(\sigma_1 \otimes \sigma_2 \otimes \sigma_3) = -1$. There is a natural trilinar form:
\begin{equation}
    I_v'(\phi) = \int\limits_{F^\times \backslash D^\times} \langle \pi(g) \phi,\widetilde \phi \rangle \, dg \qquad \phi \in \pi_1^D \otimes \pi_2^D \otimes \pi_3^D
\end{equation}
and the goal of this section is to choose vectors $\phi$ and compute $I_v'(\phi) \neq 0$ explicitly when $\epsilon(\sigma_1 \otimes \sigma_2 \otimes \sigma_3) = -1$ and $c(\sigma_i) \leq 2$. Here, $\langle -, - \rangle$ is a pairing between $\pi_1^D \otimes \pi_2^D \otimes \pi_3^D$ and its contragredient representation, and $\widetilde \phi$ is the vector dual to $\phi$ under this pairing. Recall that there were three cases (1)--(3) outlined in Proposition~\ref{prop:epsilon=-1} and we will treat each of them separately.

\begin{remark}
    As far as we know, these are the first such results when one of the components of $\pi_1 \times \pi_2 \times \pi_3$ is supercuspidal and $\epsilon(\pi_1 \otimes \pi_2 \otimes \pi_3) = -1$ so the trilinear form is on the quaternion algebra $D^\times$. When $\epsilon(\pi_1 \otimes \pi_2 \otimes \pi_3) = +1$, the trilinear form is on $\GL_2$ and Dimitrov--Nyssen~\cite{Dimitrov2010test} show how to choose vectors in $\pi_1 \times \pi_2 \times \pi_3$ when at least one component is not supercuspidal.
\end{remark}

\subsubsection{Case (1): zero supercuspidal representations}

Suppose $\pi_i = \St(2) \otimes \omega_i$ for $i = 1, 2, 3$. Then:
$$\pi_i^D = \omega_i \circ \mathrm{\nu} \colon D^\times \to \C^\times$$
is one-dimensional and we choose any non-zero vectors $\phi_i \in \pi_i^D$. 

\begin{proposition}\label{prop:local_int_spspsp}
    For characters $\omega_1, \omega_2, \omega_3$ of $F^\times$ such that $\omega_1 \omega_2 \omega_3 = 1$ and any non-zero vectors $\phi_i \in \pi_i^D$, we have that
    $$\frac{I_v'(\phi)}{\langle \phi,\widetilde \phi \rangle} = 2\mu(\O_D^\times)$$
    for $\phi = \phi_1 \times \phi_2 \times \phi_3$.
\end{proposition}
\begin{proof}
    See the proof of \cite[Proposition 4.5]{Woodbury2012explicit}.
\end{proof}

\subsubsection{Case (2): two supercuspidal representations}

Consider representations $\pi_1, \pi_2, \pi_3$ of $\GL_2(F)$ such that:
\begin{enumerate}
    \item $\pi_1$, $\pi_2$ are supercuspidal, $\pi_3 = \St \otimes \omega_3$ is a twist of the Steinberg representation,
    \item the product $\omega_1 \cdot \omega_2 \cdot \omega_3$ of their central characters it trivial,
    \item $\epsilon(\pi_1 \otimes \pi_2 \otimes \pi_3) = -1$, i.e. $\pi_1^\vee = \pi_2 \otimes \omega_3^{-1}.$
\end{enumerate}

We assume that $\pi_1$, $\pi_2$ have conductor 2 and let $\pi_i^D$ be the representation of $D^\times$ corresponding to $\pi_i$ for $i=1,2,3$ under the the Jacquet--Langlands correspondence. Then:
\begin{enumerate}
    \item $\pi_1^D = \pi_\psi$ for some character $\psi$ of $d^\times \times \langle \varpi_F \rangle$ such that $\psi|_{F^\times} = \omega_1$,
    \item $\pi_2^D \iso \pi_{\psi^{-1}} \otimes \omega_3^{-1}$,
    \item $\pi_3^D \iso \omega_3 \circ \mathrm{\nu}$.
\end{enumerate}

\begin{proposition}\label{prop:local_int_scscsp}
    Let $\pi_1, \pi_2, \pi_3$ be as above, $\epsilon_i \in \{\pm 1\}$ for $i = 1,2$, and
    \begin{enumerate}
        \item $\phi_i^{\epsilon_i} \in \pi_i^D$ non-zero such that $\pi_i^D(\varpi_D) \phi_i^{\epsilon_i} = \epsilon_i \sqrt{\omega_i(\ell)} \phi_i^{\epsilon_i}$ (cf. Proposition~\ref{prop:reps_of_cond2}),
        \item $\phi_3 \in \pi_3^D$ nonzero.
    \end{enumerate}
    Then for $\phi^\epsilon = \phi_1^{\epsilon_1} \times \phi_2^{\epsilon_2} \times \phi_3$, we have that:
    $$\frac{I_v'(\phi^\epsilon)}{\langle \phi^\epsilon,\widetilde{\phi^\epsilon} \rangle} = (1 + \epsilon_1 \epsilon_2 \sqrt{\omega_3(\ell)}) \frac{\mu(\O_D^\times)}{2}.$$
\end{proposition}
\begin{proof}
    We simplify the notation throughout the proof and write $\pi_i = \pi_i^D$. We compute, using the above description of the local representations:
    \begin{align*}
        I_v'(\phi^\epsilon) & = \int\limits_{F^\times \backslash D^\times} \langle \pi(g) \phi^\epsilon,\widetilde{\phi^\epsilon} \rangle \, dg \\
        & = \int\limits_{\O_D^\times} \langle \pi(g) \phi^\epsilon,\widetilde{\phi^\epsilon} \rangle \, dg + \int\limits_{\O_D^\times} \langle \pi(g)\pi(\varpi_D) \phi^\epsilon,\widetilde{\phi^\epsilon} \rangle \, dg &  D^\times = F^\times \O_D^\times \cup \varpi_D F^\times \O_D^\times \\
        & = (1 + \epsilon_1 \epsilon_2 \sqrt{\omega_1(\ell) \omega_2(\ell)} \omega_3(\ell)) \sum_{x \in d^\times} \langle \pi(x) \phi^{\epsilon},\widetilde{\phi^\epsilon} \rangle  \mu(U^1_D) \\
        & = (1 + \epsilon_1 \epsilon_2 \sqrt{\omega_3(\ell)}) \mu(U^1_D) \sum_{x \in d^\times} \langle \pi(x) \phi^{\epsilon},\widetilde{\phi^\epsilon} \rangle
    \end{align*}
    Finally, for $x \in d^\times$, we have that:
    \begin{align*}
        \pi_1(x) \phi_1^{\epsilon_1} & = \frac{\psi(x) + \epsilon_1 \psi^{\varpi_D}(x)}{2} \phi^+_1 + \frac{\psi(x) - \epsilon_1 \psi^{\varpi_D}(x)}{2} \phi^-_1, \\
        \pi_2(x) \phi_2^{\epsilon_2} & = \left( \frac{\psi^{-1}(x) + \epsilon_2 \psi^{-1,\varpi_D}(x)}{2} \phi^+_2 + \frac{\psi^{-1}(x) - \epsilon_2 \psi^{-1,\varpi_D}(x)}{2} \phi^-_2 \right) \omega_3(N_{d/k} x)^{-1}, \\
        \pi_3(x) \phi_3 & = \omega_3(N_{d/k} x) \phi_3.
    \end{align*}
    We hence obtain
    $$\sum_{x \in d^\times} \pi(x) \phi^{\epsilon} = S_{++} \phi^{++} + S_{+-} \phi^{+-} + S_{-+} \phi^{-+} + S_{--} \phi^{--}$$ 
    for:
    \begin{align*}
        S_{++} & = \frac{1}{4} \sum_{x \in d^\times} (\psi(x) + \epsilon_1 \psi^{\varpi_D}(x)) (\psi^{-1}(x) + \epsilon_2 \psi^{-1, \varpi_D}(x)), \\
        & = \frac{1}{4} |d^\times|(1 + \epsilon_1 \epsilon_2) \\
        S_{+-} &  = \frac{1}{4} \sum_{x \in d^\times} (\psi(x) + \epsilon_1 \psi^{\varpi_D}(x)) (\psi^{-1}(x) - \epsilon_2 \psi^{-1, \varpi_D}(x)) \\
        & = \frac{1}{4} |d^\times|(1 - \epsilon_1 \epsilon_2), \\
        S_{-+} &  = \frac{1}{4} \sum_{x \in d^\times} (\psi(x) - \epsilon_1 \psi^{\varpi_D}(x)) (\psi^{-1}(x) + \epsilon_2 \psi^{-1, \varpi_D}(x)) \\
        & = \frac{1}{4} |d^\times|(1 - \epsilon_1 \epsilon_2) ,\\
        S_{--} & = \frac{1}{4} \sum_{x \in d^\times} (\psi(x) - \epsilon_1 \psi^{\varpi_D}(x)) (\psi^{-1}(x) - \epsilon_2 \psi^{-1, \varpi_D}(x)) \\
        & = \frac{1}{4} |d^\times|(1 + \epsilon_1 \epsilon_2). \\
    \end{align*}
    This shows that:
    $$\sum_{x \in d^\times} \pi(x) \phi^\epsilon = \frac{1}{2} |d^\times| (\phi^{\epsilon} + \phi^{-\epsilon})$$
    and hence:
    $$I_v'(\phi^\epsilon) = (1 + \epsilon_1 \epsilon_2 \sqrt{\omega_3(\ell)}) \frac{\mu(\O_D^\times)}{2} \langle \phi^{\epsilon} + \phi^{-\epsilon},\widetilde{\phi^{\epsilon}} \rangle = (1 + \epsilon_1 \epsilon_2 \sqrt{\omega_3(\ell)}) \frac{\mu(\O_D^\times)}{2} \langle \phi^{\epsilon},\widetilde{\phi^{\epsilon}} \rangle.$$
    The rest of the results follow.
\end{proof}

\subsubsection{Case (3): three supercuspidal representations}

We next consider the case of three twist-minimal supercuspidal representations $\pi_1, \pi_2, \pi_3$ of $\GL_2(F)$ conductor $2$ such that $\omega_1 \cdot \omega_2 \cdot \omega_3 = 1$; according to Proposition~\ref{prop:reps_of_cond2}:
\begin{equation}
    \pi_i^D = \psi_{\psi_i}^D
\end{equation}
where $\psi_i^{\sigma} \neq \psi_i$ and we note that
\begin{equation}\label{eqn:product_cc_l2}
    \psi_1 \psi_1^{\varpi_D} \psi_2 \psi_2^{\varpi_D} \psi_3 \psi_3^{\varpi_D} = 1.
\end{equation}

Proposition~\ref{prop:epsilon=-1}~(3) classifies when $\epsilon(\pi_1 \otimes \pi_2 \otimes \pi_3) = -1$. We compute the local integrals associated with all possible choices of vectors $\phi_i^\pm \in \pi_i^D = \pi_{\psi_i}^D$ as in Proposition~\ref{prop:local_int_scscsp} and check that $I_v'(\phi) \neq 0$ for some $\phi$ exactly when $\epsilon(\pi_1 \otimes \pi_2 \otimes \pi_3) = -1$.

\begin{proposition}\label{prop:local_integral_scscsc}
    Let $\pi_1, \pi_2, \pi_3$ be as above. For $\epsilon \in \{\pm 1\}^3$ let $\phi^\epsilon = \phi_1^{\epsilon_1} \times \phi_2^{\epsilon_2} \times \phi_3^{\epsilon_3} \in \pi^D = \pi_1^D \times \pi_2^D \times \pi_3^D$. 
    Then:
    $$\frac{I_v'(\phi^\epsilon)}{\langle \phi^\epsilon,\widetilde{\phi^\epsilon} \rangle} = (1 + \epsilon_1 \epsilon_2 \epsilon_3) \frac{\mu(\O_D^\times)}{4}.$$
\end{proposition}
\begin{proof}
    Once again, we simplify the notation to write $\pi_i = \pi_i^D$ etc. We proceed as in the proof of Proposition~\ref{prop:local_int_scscsp}:
     \begin{align*}
        I_v'(\phi^\epsilon) & = \int\limits_{F^\times \backslash D^\times} \langle \pi(g) \phi^\epsilon,\widetilde{\phi^\epsilon} \rangle \, dg \\
        & = \int\limits_{\O_D^\times} \langle \pi(g) \phi^\epsilon,\widetilde{\phi^\epsilon} \rangle \, dg + \int\limits_{\O_D^\times} \langle \pi(g)\pi(\varpi_D) \phi^\epsilon,\widetilde{\phi^\epsilon} \rangle \, dg &  D^\times = F^\times \O_D^\times \cup \varpi_D F^\times \O_D^\times \\
        & = (1 + \epsilon_1 \epsilon_2 \epsilon_3 \sqrt{\omega_1(\ell) \omega_2(\ell) \omega_3(\ell)}) \sum_{x \in d^\times} \langle \pi(x) \phi^{\epsilon},\widetilde{\phi^\epsilon} \rangle  \mu(U^1_D) \\
        & = (1 + \epsilon_1 \epsilon_2 \epsilon_3) \mu(U^1_D) \sum_{x \in d^\times} \langle \pi(x) \phi^{\epsilon},\widetilde{\phi^\epsilon} \rangle.
    \end{align*}
    
    Next, we need to compute $\sum\limits_{x \in d^\times} \langle \pi(x) \phi^\epsilon,\widetilde{\phi^\epsilon} \rangle$. For $x \in d^\times$ and $i = 1,2,3$:
    \begin{align*}
        \pi_i(x) \phi_i^{\epsilon_i} & = \frac{\psi_i(x) + \epsilon_i \psi_i^{\varpi_D}(x)}{2} \phi^+_i + \frac{\psi_i(x) - \epsilon_i\psi_i^{\varpi_D}(x)}{2} \phi_i^-.
    \end{align*}
    Therefore:
    \begin{align*}
        \pi(x) \phi^\epsilon & = \sum_{\eta \in \{\pm 1\}^3} s_{\eta\epsilon} \phi^\eta, \\
        s_{\eta\epsilon} & = \frac{1}{8} \prod_{i=1}^3 (\psi_i(x) + \epsilon_i \eta_i \psi_i^{\varpi_D}(x)) \\
        & = \frac{1}{8} \sum_{ \delta \in \{0,1\}^3 }  (\epsilon_1 \eta_1)^{\delta_1} (\epsilon_2 \eta_2)^{\delta_2} (\epsilon_3 \eta_3)^{\delta_3}\psi_1^{\varpi_D^{\delta_1}}(x) \psi_2^{\varpi_D^{\delta_2}}(x) \psi_3^{\varpi_D^{\delta_3}}(x).
    \end{align*}
    and hence:
    $$\langle \pi(x) \phi^\epsilon,\widetilde{\phi^\epsilon} \rangle = s_{\epsilon\epsilon} \langle \phi^\epsilon,\widetilde{\phi^\epsilon} \rangle = \frac{1}{8} \sum_{\delta \in \{0,1\}^3} \psi_1^{\varpi_D^{\delta_1}}(x) \psi_2^{\varpi_D^{\delta_2}}(x) \psi_3^{\varpi_D^{\delta_3}}(x) \langle \phi^\epsilon,\widetilde{\phi^\epsilon} \rangle.$$
    Altogether, we have that:
    \begin{align*}
        \frac{\sum\limits_{x \in d^\times} \langle \pi(x) \phi^\epsilon,\widetilde{\phi^\epsilon} \rangle}{\langle \phi^\epsilon,\widetilde{\phi^\epsilon} \rangle} & = \frac{1}{8} \cdot |d^\times| \cdot \left|\left\{ \delta \in \{0,1\}^3 \ \middle| \ \psi_1^{\varpi_D^{\delta_1}} \psi_2^{\varpi_D^{\delta_2}} \psi_3^{\varpi_D^{\delta_3}} = 1 \right\} \right|.
    \end{align*}
    
    Without loss of generality, suppose that $\psi_1 \psi_2 \psi_3 = 1$. We claim that then:
    $$S = \left\{ \delta \in \{0,1\}^3 \ \middle| \ \psi_1^{\varpi_D^{\delta_1}} \psi_2^{\varpi_D^{\delta_2}} \psi_3^{\varpi_D^{\delta_3}} = 1 \right\} = \{ (0,0,0), (1,1,1) \}.$$
    By assumption, $(0,0,0) \in S$ and by equation~\eqref{eqn:product_cc_l2} also $(1,1,1) \in S$. To verify that $S$ cannot be larger, suppose without loss of generality that $(1,0,0) \in S$, i.e.\ $\psi_1^{\varpi_1} \psi_2 \psi_3 = 1$. Then $\psi_1^{\varpi_1} = \psi_1$, but this contradicts the admissibility of the pair $(K, \psi_1)$. 
\end{proof}

\begin{remark}
    It would be interesting to treat the case of supercuspidal representations of higher conductor as well, but this would take us too far afield from the ultimate arithmetic goals of the paper.
\end{remark}

\section{Global JL correspondence and test vectors}\label{Section: Global JL correspondence and test vectors}

\noindent We are ready to study the global consequences of Section \ref{Section: Local JL correspondence and test vectors}. We focus our attention on quaternionic modular forms with level structure given by orders which are \emph{residually inert} at the primes where the quaternion algebra ramifies; the local theory considered in the previous section allows a precise understanding of such forms.

From now on, $D$ denotes a quaternion algebra over $\Q$ (and not a local quaternion algebra as in the previous section). For simplicity of exposition, until Section \ref{section: Pairings}, we restrict ourselves to the case where the quaternion algebra is ramified exactly at one odd prime $\ell$ and at infinity. However, everything we state in this section generalizes to any definite quaternion $\Q$-algebra (see also Remark~\ref{rmk: W in families}); in particular, the results in Section \ref{section: Pairings} and Sections \ref{Section: The JL correspondence in families}--\ref{section: Application: Elliptic Stark Conjecture in rank one} deal with the general situation.

For any place $v$ of $\Q$, we denote $D_v=D\otimes_\Q\Q_v$, where we understand $\Q_\infty$ to be $\R$. Similarly, for any order $R\subset D$ and any finite place $v$, we denote $R_v=R\otimes_\Z\Z_v$. As a last piece of notation, we set $D(\A)= D\otimes_\Q \A$, for $\A$ the adèles of $\Q$, $\widehat{D}= D\otimes_\Q \A_{f}$, for $\A_{f}$ the finite adèles, and $\widehat{R}= R \otimes_\Z \widehat{\Z}$, for $\widehat{\Z}$ the profinite completion of the integers.

\subsection{A remark on the structure of quaternion algebras at ramified primes}\label{Section: A remark on the structure of quaternion algebras at ramified primes}

We begin by recalling a general formalism to deal with definite quaternion algebras over a local field, which will make our exposition clearer and independent of the choices of the uniformizers. Most of the content of this section can be found in \cite[Section 13]{Voight2021Book}. Let $D_\ell$ be a quaternion division algebra over $\ql$ and let $\ql(\varpi_\ell)$ and $\qll$ be, respectively, one of the two ramified quadratic extension of $\ql$ and the unique unramified quadratic one; as $D_\ell$ is division, there exists embeddings  of these two fields in $D_\ell$. Consider $\nu_\ell:D_\ell\longrightarrow\ql$ to be the reduced norm map at $\ell$. We extend the $\ell$-adic valuation $v_\ell$ of $\ql$, to the division algebra: $w_\ell=\tfrac{1}{2}\cdot v_\ell\circ \nu_\ell$. We fix $$\O_{D_\ell}=\{x\in D_\ell \mid w_\ell(x)\in \Z_\ell\}$$ and denote by $\varpi_{D_\ell}\in\O_{D_\ell}$ a uniformizer of $D_\ell$, namely an element in $\O_{D_\ell}$ with valuation $w_\ell(\varpi_{D_\ell})=1/2$; it is not difficult to notice that one can take $\varpi_{D_\ell}=\varpi_{\ell}$. Therefore, we can decompose the division algebra $D_\ell$ as
\begin{equation}
    D_\ell = \qll \oplus \varpi_{\ell} \qll.
\end{equation}
Denoting by $\varpi_{\qll}$ the uniformizer of $\qll$, we can further write
\begin{equation}
    D_\ell = \ql \oplus \varpi_{\qll} \ql \oplus \varpi_{\ell} \ql \oplus \varpi_{\qll}\varpi_{\ell}\ql,
\end{equation}
with the condition $\varpi_{\qll}\varpi_{\ell} = \varpi_{\ell}\overline{\varpi_{\qll}}=-\varpi_{\ell}\varpi_{\qll}$.

\subsection{Residually inert orders}

We briefly recall the definition of residually inert (at $\ell$) orders in $D$. The interested reader may consult \cite[Section 24.3]{Voight2021Book},~\cite{HPS1989orders} and~\cite{Pizer80p2} for a more detailed exposition.
\begin{definition}
    Let $N$ be a positive integer prime to $\ell$. We say that an order $R\subset D$ of level $N\ell^2$ is 
    \begin{itemize}
        \item \emph{residually split} at the prime $q\mid N$ if $R_q$ is an Eichler order of level $q^{\textrm{val}_q(N)}$;
        \item \emph{residually inert} at $\ell$ (also known as special or Pizer order) if there exists a ramified quadratic extension $\ql(\varpi_\ell)/\ql$, such that $R_\ell$ is conjugate to 
        \begin{equation*}
            \O_\ell + \left\{x \in D_\ell \mid \nu_\ell(x)\in\ell\Z_\ell\right\}=\O_\ell + \varpi_{D_\ell} \O_{D_\ell},
        \end{equation*}
        where $\O_\ell$ is the ring of integers of $\ql(\varpi_\ell)$.
    \end{itemize}
\end{definition}    
In order to shorten the notation, we call such a global order a \emph{Pizer order of level $N\ell^2$}, however, we remark that these orders are a type of \emph{basic} orders and we point to \cite[Remark 24.5.7]{Voight2021Book} for a complete discussion on the different terminologies for such orders. 
\begin{remark}
    \leavevmode
    \begin{itemize}
        \item If the order has level $N\ell$, then the local order at $\ell$ is no longer residually inert, but it is the unique maximal order in $D_\ell$. In particular, an order $R$ of level $N\ell^2$ is contained in an Eichler order of level $N\ell$.
        \item Adding a subscript $\ell$ at the notation of  Section \ref{Section: Local JL correspondence and test vectors}, we notice that, by equation~\eqref{Eq.: definition U_D^a}, $R_\ell^\times \supseteq U_{D_\ell}^1$ with quotient $\mathbb{F}_\ell^\times$. The quotient in equation~\eqref{Eq.: semidirect product local order} recovers the observations in \cite[proofs of Propositions\ 1.8 and 9.26]{Pizer80p2}.
    \end{itemize}    
\end{remark}
For any prime $q\neq \ell$, we fix a $\Q_q$-linear isomorphism $\iota_q \colon D_q\cong M_2(\Q_q)$; up to changing this isomorphism, we may assume that
\begin{equation}\label{Eq.: iota_q(R) upper triangular}
    \iota_q(R_q) =\left\{ \gamma \in M_2(\Z_q)\ \middle|\ \gamma \equiv \Mat{*}{*}{0}{*}\Mod{ N M_2(\Z_q)}\right\}.
\end{equation}
At the ramified prime $\ell$, we assume that $R_\ell = \O_\ell + \varpi_{D_\ell} \O_{D_\ell}$. For the rest of this section, we fix $R$ to be a Pizer order of level $N\ell^2$, and advise the reader that every time we pick a Pizer order, we are implicitly assuming the above identifications. We also introduce the following notation
\begin{equation}
    U_1(R) =\left\{r=(r_q)\in \widehat{R}^{\times} \ \middle| \ 
    \iota_q(r_q) \equiv \Mat{*}{*}{0}{1}\Mod{NM_2(\Z_q)},\textrm{ for }q\mid N \textnormal{ and } r_\ell\in 1+\varpi_{D_\ell} R_\ell \right\}.
\end{equation}
It is not difficult to notice that it is an open compact subgroup of $\widehat{B}^\times$ (cf. also \cite[Lemma 2.1.3]{DallAva2021Hida}).

\subsection{Lifting characters}\label{Section: Lifting characters}

Let $\chi$ be a Dirichlet character of conductor $C$, for $C\mid N\ell$. 
Every such character can be lifted to a character of $\widehat{R}^\times$; for simplicity we consider the case of $\ell$ odd. A similar construction works for $\ell=2$ and a precise recipe is provided in \cite[Section 7.2]{HPS1989orders} as we focus on the odd case. Let $q$ be a prime and let $\chi_q$ be the $q$-component of the character $\chi$. We define the lift $\widetilde{\chi}_{q}$ of $\chi_q$ to $R_q^\times$ as follows:
\begin{enumerate}[align=left]
    \item If $q\mid N\ell$, but $q\nmid C$, we set $\widetilde{\chi}_{q}(r)=1$ for any $r\in R_q^\times$.
    \item If $q\mid N$ and $q\mid C$, we set $\widetilde{\chi}_{q}(r)=  \chi_q(d)$ for any $r\in R_q^\times$ such that $\iota_q(r)=\mat{a}{b}{c}{d}$;
    \item If $q=\ell>2$ and $\ell\mid C$, we fix, once and for all, an odd character $\varepsilon_\ell$ with $\cond(\varepsilon_\ell)=\ell$. For every even character $\phi_\ell$ of $\cond(\phi_\ell)=\ell$, we fix, once and for all, a character $\gamma_\ell$ with $\cond(\gamma_\ell)=\ell$ and such that $\gamma_\ell^2=\phi_\ell$. As remarked in \cite[Section 7.2]{HPS1989orders}, the particular choice of $\varepsilon_\ell$ and $\gamma_\ell$ is not important, but the fact that a certain choice is fixed once and for all is crucial.
    \begin{enumerate}[align=left, itemsep=2ex]
        \item We first extend $\varepsilon_\ell$ to $R_\ell^\times$ via the composition:
        \begin{equation*}
            \begin{tikzcd}
            	\O_\ell^\times \arrow[d, hook] \arrow[rr]                                                                                         &  & (\O_\ell/\varpi_{D_\ell}\O_\ell)^\times \arrow[rr, "\iso"'] &  & (\Z_\ell/\ell\Z_\ell)^\times \arrow[d, "\varepsilon_\ell"'] \\
            	{\,\,\,\,\,\,\,\,\,\,\,R_\ell^\times=(\O_\ell + \varpi_{D_\ell} \O_{D_\ell})^\times} \arrow[rrrr, "\widetilde{\varepsilon}_\ell"] &  &                                                             &  & \qpbar^\times.                                             
            \end{tikzcd}
        \end{equation*}
        \item If $\chi_\ell$ is even, let $\gamma_\ell$ be the \emph{square root} character associated with it, as above. We define $\widetilde{\chi}_\ell(r)=\gamma_\ell(\nu_\ell(r))$, for any $r\in R_\ell^\times$. Clearly, $\widetilde{\chi}_\ell(r)=\widetilde{\chi}_\ell(r')$ if $r\equiv r'\Mod{\varpi_{D_\ell}\O_{D_\ell}}$.
        \item If $\chi_\ell$ is odd, then $\chi_\ell=\varepsilon_\ell \cdot \chi'_\ell$, with $\chi'_\ell$ even. We then define $\widetilde{\chi}_\ell=\widetilde{\varepsilon}_\ell\widetilde{\chi'_\ell}$.
    \end{enumerate}
\end{enumerate}
\begin{definition}\label{def: quaternionic lift of character}
    We denote the lift of $\chi$ to $\widehat{R}^\times$ by $\widetilde{\chi}$, defined as $\widetilde{\chi}=\prod\limits_{q\mid C}\widetilde{\chi}_q$.
\end{definition}

Recall that we can define the adèlization of $\chi$,
\begin{equation}
    \chi_\A:\Q^\times\backslash\A^\times/\R_{+}(1+N\widehat{\Z})^\times\longrightarrow \C^\times,
\end{equation}
as the unique finite order Hecke character such that $\chi_\A((1,\ldots,1,q,1,\ldots))=\chi(q)^{-1}$; then $\chi_q(q)=\chi(q)^{-1}$ if $q\nmid N$.
\begin{remark}
    \leavevmode
    \begin{enumerate}
        \item The restriction of $\widetilde{\chi}$ to $\Z$ recovers the starting character $\chi$. At any prime $q$ away from $\ell$, the lifting process of each local component $\chi_q$ consists exactly in the adèlization of its inverse, $(\chi_q^{-1})_\A$.
        \item The lifting process can be constructed compatibly with the inclusion of Pizer orders $R\supseteq R'$, with $R'$ Pizer order of level $N'\ell^2$, $N\mid N'$.
        \item By construction, the character $\widetilde{\chi}$ is trivial on $U_1(R)$. More precisely, it is a character of $(\widehat{R}/\delta(R))^\times$, for $\delta(R)$ the different ideal of $\widehat{R}$ (cf. \cite[Sections 7.1-7.2]{HPS1989orders}).
    \end{enumerate}
\end{remark}

\subsection{Forms on definite quaternion algebras}

In this section we recall the various notions of quaternionic forms and their explicit relations. We fix an odd prime $p\neq \ell$ and an absolute closure of $\qp$, which we denote by $\qpbar$.

\subsubsection{Quaternionic modular forms}

Let $A$ be a commutative ring and consider the space of polynomials in two variables, $A[X,Y]$. We endow this module with the action of invertible matrices $GL_2(A)$ given by
\begin{equation}
    \gamma \cdot P(X,Y) = P\left((X,Y)\cdot \gamma\right),
\end{equation}
for any $\gamma \in GL_2(A)$ and $P\in A[X,Y]$. For any $m\in \Z_{\geq 0}$ we define the submodule of homogeneous polynomials of degree $m$ and denote it by $L_m(A)$; the $GL_2(A)$-action descends to an action on $L_m(A)$. In the following, the ring $A$ will be an algebra over the ring of integers of a finite extension of $\qp$, for example, $\qpbar$.

\begin{definition}[Quaternionic modular forms]
    Let $R$ be a Pizer order of level $N\ell^2$, and assume that $p\mid N$. Let $A$ be a $\zp$-algebra and fix an $A$-valued Dirichlet character $\chi$ with conductor $C\mid N\ell$. A {\em quaternionic modular form} for $D^\times$, of weight $k\in \Z_{\geq 2}$, level structure $R$, and character $\chi$, is a continuous function $\varphi: \widehat{D}^\times \longrightarrow L_{k-2}(A)$, such that
    \begin{equation*}
        \varphi(d\hat{d}zr)= (\chi_\A)^{-1}(z)\,\widetilde{\chi}(r)\,z_p^{2-k}\, (r_p^{-1}\cdot \varphi(\hat{d})),
    \end{equation*}
    for $d\in D^\times$, $\hat{d}\in\widehat{D}^\times$, $z\in \A_{f}^\times$, and $r \in \widehat{R}^\times$.
    Here $\widetilde{\chi}$ is the lifting constructed in Definition~\ref{def: quaternionic lift of character}. We denote the space of quaternionic modular forms by $S_k^D(R,\chi,A)$.
\end{definition}

Note that, for $z\in \A_{f}^\times\cap \widehat{R}^\times$, $\widetilde{\chi}(z) = (\chi_\A)^{-1}(z)$, and that $\varphi$ is right-invariant under $U_1(R)^{(p)}$, i.e. away from $p$. The space $S_k^D(R,\chi,A)$ inherits a right-action of $GL_2(\zp)$ given by $\varphi_{|\gamma}(\hat{d})=\gamma^{-1} \cdot \varphi(\hat{d})$.

As defined, quaternionic modular forms are not the \emph{classical} quaternionic modular forms considered in~\cite{HPS1989orders}, which are (non-unitarized) automorphic forms on $D$. In order to clarify the situation, as well as making explicit certain actions on the space of quaternionic modular forms (see Section \ref{Section: Operators on quaternionic modular forms}), we recall the definition of automorphic forms and algebraic quaternionic modular forms, highlighting their relations with quaternionic modular forms.

\subsubsection{Automorphic forms on $D^\times$}

Fix an isomorphism $\iota:\qpbar\cong \C$ and let $\Psi_\infty:D_\infty=D\otimes_\Q \R\hookrightarrow M_2(\C)$ be an embedding such that $\iota(\iota_p(d))=\Psi_\infty(d)$, for any $d\in D^\times\subset D_p^\times\overset{\iota_p}{\cong}GL_2(\qp)$. For any $k\geq 2$, let
\begin{equation}
    \rho_{\infty,k}^{u}:D_\infty^\times \longrightarrow \Aut(L_{k-2}(\C))
\end{equation}
be the unitarized representation defined by
\begin{equation}
   \rho_{\infty,k}^{u}(d_\infty) (P(X,Y)) = |\nu_{D_\infty}(d_\infty)|^{\tfrac{k-2}{2}}\, \Psi_\infty(d_\infty) \cdot P(X,Y).
\end{equation}

\begin{definition}[Automorphic forms on $D^\times$]
    Let $R$ and $\chi$ be as above. We define the space $\Ac^{D}_{k}(R,\chi)$ of {\rm automorphic forms} on $D^\times$, of weight $k$, level $R$, and character $\chi$, as the vector space of continuous functions $\phi \colon D(\A)^\times\longrightarrow L_{k-2}(\C)$, such that
    \begin{equation*}\label{eq: infty action on QMF}
        \phi(d x d_\infty r z) = (\chi_\A)^{-1}(z)\,\widetilde{\chi}(r)\,\rho_{\infty,k}^{u}(d_\infty^{-1})\left( \phi(x)\right),
    \end{equation*}
    for $d\in D^\times$, $d_\infty\in D_\infty^\times$, $z\in \A^\times$, and $r\in \widehat{R}^\times$.
\end{definition}

There is a correspondence between quaternionic modular forms and automorphic forms. Precisely, to each $\varphi \in S^D_{k}(R,\chi)$ we associate the automorphic form $\Phi(\varphi)\in \Ac^D_{k}(R,\chi)$ defined as
\begin{equation}\label{eq: lift from QMFs to AF on D}
    \Phi(\varphi)(d) = \rho_{\infty,k}^{u}(d_\infty^{-1}) \Big(\iota \big( d_p\cdot \varphi(d_f)\big)\Big),
\end{equation}
for any $d\in D(\A)^\times$, with finite part $d_f\in \widehat{D}^\times$, and components $d_p$ and $d_\infty$, respectively at $p$ and infinity.

\subsubsection{Algebraic quaternionic modular forms}\label{Section: Algebraic quaternionic modular forms}

In order to construct the $p$-adic $L$-function in the following  Section \ref{Section: Balanced triple product $p$-adic $L$-function}, we need Hida families of quaternionic modular forms. We define them in Section \ref{Section: Families of quaternionic modular forms}, similarly to \cite[Definition 4.1]{Hsieh2021}, and for this purpose we should introduce algebraic quaternionic modular forms.\\

Consider the Iwasawa algebra $\Lambda=\zp\llbracket 1+p\zp\rrbracket$. For any $z\in 1+p\zp$, we denote by $[z]_\Lambda$ the group-like element of $\Lambda$ determined by $z$. For $r\geq 1$, let $R_{N\ell^2p^r}$ be a sequence of nested Pizer orders on level $N\ell^2p^r$ contained in the nested sequence of Eichler order $R_{Np^r}$. We consider the finite set
\begin{equation}
    X_{r}= D^\times\backslash\widehat{D}^\times / U_1(R_{N\ell^2 p^r}),
\end{equation}
and let $\O[X_r]$ be the finitely generated $\O$-module spanned by divisors on $X_r$. The Iwasawa algebra $\Lambda$ acts $\zp$-linearly on $\O[X_r]$ via
\begin{equation}
    [z]_\Lambda \cdot x = x\, (1,\ldots, 1, \mat{z}{0}{0}{z},1,\ldots),
\end{equation}
for $x\in X_r$, $z\in 1+p\zp$ and $(1,\ldots, 1, \mat{z}{0}{0}{z},1,\ldots)$ the adèle with $p$-th component $\mat{z}{0}{0}{z}$.

\begin{definition}[Algebraic quaternionic modular forms]
    Let $\chi$ be a Dirichlet character of conductor dividing $N\ell p^r$. We define the space of {\em algebraic quaternionic modular forms} of level $R_{N\ell^2p^r}$, weight $k$ and character $\chi$ as the finite $\Lambda$-module
    $$    \mathbb{S}_{k}^D(R_{N\ell^2p^r},\chi)=\left\{f\in \Hom_{\rm cts}(\O[X_r], L_{k-2}(\qpbar)) \, \middle| \begin{array}{c}  f(xzr) = (\chi_\A)^{-1}(z)\widetilde{\chi}(r) f(x) \\ \textrm{ for all }\, z\in\A_{f}^\times, \, r\in \widehat{R}_{N\ell^2p^r}^\times \end{array} \right\}.
    $$
\end{definition}
Recall that any $\varphi\in S^{D}_{k}(R_{N\ell^2p^r},\chi)$ is not $U_1(R_{N\ell^2 p^r})$-invariant but only $U_1(R_{N\ell^2 p^r})^{(p)}$-invariant. The relation between $\mathbb{S}_{k}^D(R_{N\ell^2p^r},\chi)$ and $S_{k}^{D}(R_{N\ell^2p^r},\chi)$ is made clear via the association
\begin{equation}
    S_{k}^{D}(R_{N\ell^2p^r},\chi)\ni \varphi(x) \longmapsto f(x)=x_p\cdot \varphi(x)\in\mathbb{S}_{k}^D(R_{N\ell^2p^r},\chi).
\end{equation}

\begin{remark}
    We will define Hecke operators on $\mathbb{S}_{k}^D(R_{N\ell^2p^r},\chi)$ momentarily (cf.\ Section \ref{Section: Hecke operators}). For now, let us mention that the $\Lambda$-action on the space of quaternionic modular forms $S_{k}^D(R_{N\ell^2p^r},\chi)$ takes the form of
    \begin{equation}\label{Eq: Lambda action}
        [z]_\Lambda\cdot \varphi(x) = z\cdot \varphi(x (1,\ldots, 1, \mat{z}{0}{0}{z},1,\ldots))=(\chi_{\A,p})^{-1}(z)\,\varphi(x),
    \end{equation}
    In other words, we can think of the $\Lambda$-action as a diamond operator at the prime $p$.
\end{remark}

\subsection{Operators on quaternionic modular forms}\label{Section: Operators on quaternionic modular forms} 

\subsubsection{Hecke operators}\label{Section: Hecke operators}
 
As in the usual setting of classical modular forms, one can define Hecke operators acting on the space of quaternionic modular forms. These are easily defined taking into account the adèlic description of these forms; for any prime $q\neq \ell$, let $\varsigma_q\in\A_{f}^\times$ be the element characterized by $\varsigma_{q,q}=q$ and $1$ at other places. Let $A$ be again a $\zp$-algebra and let $\varphi\in S_k^D(R_{N\ell^2},\chi,A)$. For any $\hat{d}\in\widehat{D}^\times$, the Hecke operator $T_q$ is determined by
	\begin{equation}
		T_q \varphi(\hat{d}) =\begin{cases}
			\varphi\left(\hat{d}\mat{1}{0}{0}{\varsigma_q}\right) + \displaystyle\sum_{a\in\Z/q\Z}\varphi\left(\hat{d}\mat{\varsigma_q}{a}{0}{1}\right) & \textrm{ for each prime }q\nmid N\ell^{2},\\
			\\
			\varphi|_{\mat{1}{0}{0}{p}}\left(\hat{d}\mat{1}{0}{0}{\varsigma_p}\right) + \displaystyle\sum_{a\in\Z/q\Z}\varphi|_{\mat{p}{a}{0}{1}}\left(\hat{d}\mat{\varsigma_p}{a}{0}{1}\right) & \textrm{ for }q=p \textrm{ and }p\nmid N,
		\end{cases}
	\end{equation}
	while the Hecke operator $U_q$ is
	\begin{equation}
		U_q \varphi(\hat{d}) =\begin{cases}
			\displaystyle\sum_{a\in\Z/q\Z}\varphi\left(\hat{d}\mat{\varsigma_q}{a}{0}{1}\right) & \text{ for }q\mid N,\\
			\\
			\displaystyle\sum_{a\in\Z/p\Z}\varphi|_{\mat{p}{a}{0}{1}}\left(\hat{d}\mat{\varsigma_p}{a}{0}{1}\right) & \textrm{ for }q=p \textrm{ and } p\mid N.
		\end{cases}
	\end{equation}
 
    \begin{remark}
        \leavevmode
        \begin{enumerate}
            \item In the case of a definite quaternion algebra, one can express the action of quaternionic Hecke operators via Brandt matrices. Via these matrices we can also define Hecke operators for the primes of ramification for the quaternion algebra. As Hecke operators are compatible with the Jacquet--Langlands correspondence (up to a sign at $\ell$), applied to forms which are new at $\ell$, $U_\ell$ will be the zero-operator. We refer the interested reader to~\cite{Eichler73},~\cite{P1977},~\cite{Pizer80p2} and~\cite{HPS1989basis} for their precise definition.
            \item We define an Atkin--Lehner operator in Definition~\ref{def: quaternionic Atkin--Lehner}, following~\cite{Pizer80p2}. One can find quaternionic analogues of the local Atkin--Lehner operators in Section 9 of loc.\ cit.\ and~\cite{P1977}. For what concerns diamond operators, they can be defined as in~\cite[Section 4.4]{Hsieh2021}.
            \item The first-named author wishes to remark that the definition of the $U_\ell$ operator provided in~\cite{DallAva2021Hida} is not correctly stated. As above, one should use Brandt matrices to define it and not the involution obtained by multiplication with an element of norm $\ell$ in $D_\ell^\times$. This definition is correct for Eichler orders, as it coincides (up to a sign) with the Atkin--Lehner operator. The proofs and results contained in loc.\ cit.\ are unfazed by this misstated definition.
        \end{enumerate}
	\end{remark}

\subsubsection{Multiplicity two}

A peculiar property of Pizer orders of level $N\ell^2$ is the failure of multiplicity one for Hecke-eigenspaces of quaternionic modular forms. This phenomenon has been extensively studied in both~\cite{Pizer80p2} and~\cite{HPS1989basis}, and it can be summarized in the following proposition.

\begin{proposition}[Pizer, Hijikata--Pizer--Shemanske]\label{prop: multiplicity-2}
    Let $f\in S_k(\Gamma_1(N\ell^2),\chi)^{\rm new}$ be a Hecke-eigenform, twist-minimal at $\ell$. Suppose that the $\ell$-component of $\chi$ has conductor $\cond(\chi_\ell)\leq \ell$. There is an isomorphism of Hecke-eigenspaces
    \begin{equation*}
        S_k^D(R_{N\ell^2},\chi,\qp(f))[f]=S_k(N\ell^2,\chi)[f]^{\oplus 2}.
    \end{equation*}
    In particular, the Hecke-eigenspace $S_k^D(R_{N\ell^2},\chi,\qp(f))[f]$ is two-dimensional.
\end{proposition}

\begin{remark}
    The Jacquet--Langlands correspondence associates to any automorphic representation $\pi_D$ of $D^\times(\A)$ an automorphic representation $\pi$ of $\GL_2(\A)$. By Strong Multiplicity One for $\GL_2$ and its inner forms, there cannot be two different automorphic  representations $\pi_D$ and $\pi_D'$ of $D^\times(\A)$ that correspond to the same automorphic representation $\pi$ of $\GL_2(\A)$. Proposition~\ref{prop: multiplicity-2} is therefore about automorphic forms, i.e.\ elements of $\pi_D$ and $\pi$. 

    For $\GL_2(\A)$, there is a well-known theory of newforms: there is a compact open subgroup $K_\pi \subseteq \GL_2(\A_f)$ such that $\dim \pi_f^{K_\pi} = 1$. This is no longer the case for $D^\times(\A)$; in fact, underlying Proposition~\ref{prop: multiplicity-2} is the fact that $\dim \pi_{D, f}^{R_{N\ell^2}^\times} = 2$, which was observed locally at the prime $\ell$ in Proposition~\ref{prop:reps_of_cond2}.
\end{remark}

\subsubsection{An extra operator}

	Using the methods of Section \ref{Section: Local JL correspondence and test vectors}, we can define an extra operator acting on the space of $A$-valued quaternionic modular forms $S_k^D(R_{N\ell^2},\chi,A)$; its nature is local, arising from the structure of the order at the prime of ramification $\ell$.
	
	With a slight change of notation from Section \ref{Section: Local JL correspondence and test vectors}, let $\varpi_D$ be the element in $\widehat{D}^\times$ defined as
    \begin{equation}
        \varpi_D=\left(1,\ldots,1,\varpi_{D_\ell},1,\ldots\right),
    \end{equation}
    for $\varpi_{D_\ell}$ a quaternionic uniformizer of $D_\ell/\ql$; note that $N_{\Q_{\ell}(\varpi_{D_\ell})/\Q_{\ell}}(\varpi_{D_\ell})=\varpi_{\ql}$ for a uniformizer of $\Q_\ell$. We will choose $\varpi_{D_\ell}$ so that $\varpi_{D_\ell}^2=\ell$. Equivalently, we are fixing our chosen ramified quadratic extension $\Q_{\ell}(\varpi_{D_\ell})$ to be $\Q_{\ell}(\sqrt{\ell})$. Notice that one can fix another choice of uniformizer and, mutatis mutandis, all the results in this section and in Section~\ref{Section: The JL correspondence in families} are similarly obtained.
	
	\begin{definition}
	    We define the operator $\UL{\ell}$ on the space $S_k^D(R_{N\ell^2},\chi,A)$ as 
	    \begin{equation*}
	        \UL{\ell}\,\varphi(\hat{d}) =\varphi\left(\hat{d}\varpi_D\right).
	    \end{equation*}
	\end{definition}
	\begin{proposition}\label{prop: UL properties}
            \leavevmode
	    \begin{enumerate}
	        \item Iterating $\UL{\ell}$ twice, one has
	        \begin{equation*}
	            \UL{\ell}^2\,\varphi(\hat{d})=\varphi\left(\hat{d}\,(1,\ldots,1,\varpi_{D,\ell}^2,1,\ldots)\right)=\chi_N(\ell)^{-1}\varphi(\hat{d}).
	        \end{equation*}
	        \item The operator $\UL{\ell}$ commutes with the Hecke-operators.
	    \end{enumerate}
	\end{proposition}
	\begin{proof}
	    We notice that $(1,\ldots,1,\ell,1,\ldots)\in \A_{f}^\times$, and
        \begin{equation*}
            \chi_\A^{-1}\left(\ell\right) = 1 
            = \chi_{\ell,\A}^{-1}\left((1,\ldots,1,\ell,1,\ldots)\right)\chi_{N,\A}^{-1}\left((\ell,\ldots,\ell,1,\ell,\ldots)\right),
        \end{equation*}
        hence 
        \begin{equation*}
            \chi_{\ell,\A}^{-1}\left((1,\ldots,1,\ell,1,\ldots)\right)=\chi_{N,\A}\left((\ell,\ldots,\ell,1,\ell,\ldots)\right)=\chi_{N}^{-1}(\ell).
        \end{equation*}
        Both statements follow now from the definitions of $\UL{\ell}$, quaternionic modular forms, and Hecke operators away from $\ell$.
	\end{proof}
    The choice of the notation $\UL{\ell}$ is justified by the above proposition; this extra operator is closer to a diamond operator than a replacement for the Hecke operator at $\ell$. Let us remark that $\chi_N^{-1}$ evaluated at $\ell$ coincides with the evaluation at the same integer of the adèlization of the character $\chi_\ell$; this follows by the definition of the adèlization as the rational elements not prime to the conductor are mapped to 1.
	\begin{proposition}\label{prop: mult-1 with varpi_d}
	    Let $f\in S_k(\Gamma_1(N\ell^2),\chi)^{\rm new}$ be a Hecke-eigenform, twist-minimal at $\ell$. As above, suppose that the $\ell$-component of $\chi$ has conductor $\cond(\chi_\ell)\leq \ell$. The two-dimensional Hecke-eigenspace $S_k^D(R_{N\ell^2},\chi,A)[f]$ decomposes into one-dimensional eigenspaces under the action of $\UL{\ell}$,
	    \begin{equation*}
	        S_k^D(R_{N\ell^2},\chi,A)[f]=S_k^D(R_{N\ell^2},\chi,A)[f]^{\UL{\ell}=+\sqrt{\chi_N^{-1}(\ell)}}\oplus S_k^D(R_{N\ell^2},\chi,A)[f]^{\UL{\ell}=-\sqrt{\chi_N^{-1}(\ell)}}.
	    \end{equation*}
	\end{proposition}
	\begin{proof}
	    Let $\pi_f$ be the automorphic representation associated with $f$ and  consider $\pi_{f,\ell}$ its local component at $\ell$. By definition of automorphic forms, the quaternionic modular forms in $S_k^D(R_{N\ell^2},\chi,A)[f]$ correspond to (semi-)invariant vectors in the automorphic representation $\pi_f^D$, obtained as the Jacquet--Langlands transfer of $\pi_f$. In other words, we can identify $S_k^D(R_{N\ell^2},\chi,A)[f]$ with the space $(\pi_f^D)^{R^\times}$ of (semi-)invariant vectors. By definition, the action of the operator $\UL{\ell}$ is local on the representation $\pi_{f,\ell}^D$ and it acts as the local uniformizer $\varpi_{D_\ell}$. Since the representation $\pi_{f,\ell}^D$ is a twist-minimal smooth representation of $D_\ell^\times$ of conductor $2$, we can apply Proposition~\ref{prop:reps_of_cond2} and conclude the proof. Notice that the central character of $\pi_f$ is (up to the archimedean component) the adèlization of the character $\chi$, thus, the central character of $\pi_{f,\ell}$ at $\ell$ coincides with $\chi_{\A,\ell}\left((1,\ldots,1,\ell,1,\ldots)\right)=\chi_N^{-1}(\ell)$.
	\end{proof}

	\begin{remark}\label{rmk: W in families}
        \leavevmode
	    \begin{enumerate}
	        \item The result of Lemma~\ref{prop: mult-1 with varpi_d} depends only on the local representation at $\ell$, more precisely, on its automorphic type and its \emph{minimal} conductor.
	        \item The local automorphic type, the conductor and the twist-minimality property are rigid in $p$-adic Hida families of modular forms (see \cite[Lemma 2.14]{FouquetOchiai2012} and \cite[Remark 3.1]{Hsieh2021}).
            \item If one considers more general quaternion algebras, the multiplicity will be given by
            \begin{equation*}
                \prod_{\substack{\ell \in {\rm Ram}(D)-\{\infty\}:\\ {\rm cond}_\ell(\pi)=\ell^{2}}} 2.
            \end{equation*}
            Therefore, considering a $\UL{\ell}$-operators for each $\ell\in {\rm Ram}(D)-\{\infty\}$, one can recover multiplicity one for quaternionic modular forms.
            \item In the case $\chi = 1$, one can define a whole set of operators, defining a dihedral group of order $2(\ell+1)$, as in \cite[Section 9]{Pizer80p2}. Among these, by our analysis in Section~\ref{Section: Local JL correspondence and test vectors}, one can recognize the operator $\UL{\ell}$, the Atkin--Lehner operator, and a few mysterious others arising from the local structure of the order. The indefinite setting has been studied in~\cite{deVeraPiquero2013}.
        \end{enumerate}
	\end{remark}

	Our analysis is partially motivated by the explicit work of Pizer~\cite{Pizer80p2}, where the study of Pizer orders is addressed with a more elementary flavor; in particular, the above lemma provides a more explicit take on the conductor $\ell^2$ case of the statements contained in \cite[Section 9]{HPS1989basis}. Moreover, if the central character is trivial, we recover the setting studied by Pizer. Even though the authors of loc.\ cit.\ observed and studied the higher multiplicity phenomenon without addressing how to recover one-dimensional eigenspaces, their explicit work has been a key input for our results. It also helped us to implement the operator $\UL{\ell}$ in Pizer's setting in \magma.  Our implementation can be found at \citeGitHub{1} and we use it to compute a first example below.

	\begin{example}
	    With the above notations, take $\ell=7$ and let $f\in S_2(\Gamma_0(7^2))^{\rm new}$, be the unique newform of level of level $7^2$ and trivial character. In \cite[Example 10.5]{HPS1989basis}, the authors compute that the $f$-eigenspace of quaternionic modular forms of weight 2, level $7^2$ and trivial character, is generated by two forms, namely $\varphi_1$ and $\varphi_2$, whose values on the ideal classes are
	    \begin{align*}
	        \varphi_1 \longleftrightarrow (1,-1,0,0)^t, && \varphi_2 \longleftrightarrow (0,0,1,-1)^t.
	    \end{align*}
        Since the character is trivial, the eigenvalues of the operator $\UL{\ell}$ are $\pm 1$.  Our implementation in \verb|magma| yields the action of the $\UL{\ell}$-operator as
	    \begin{align*}
	        \UL{7}\varphi_1=\varphi_1, && \UL{7}\varphi_2=-\varphi_2.
	    \end{align*}
	    Thus, we decompose
	    \begin{equation*}
	        S^D_2(R_{7^2})^{\rm new}[f]= \left(S^D_2(R_{7^2})^{\rm new}[f]\right)^{\UL{7}=+1}\oplus\left(S^D_2(R_{7^2})^{\rm new}[f]\right)^{\UL{7}=-1},
	    \end{equation*}
	    for $\left(S^D_2(R_{7^2}\right)^{\rm new}[f])^{\UL{7}=+1}=\langle\varphi_1\rangle$ and $\left(S^D_2(R_{7^2})^{\rm new}[f]\right)^{\UL{7}=-1}=\langle\varphi_2\rangle$.
	\end{example}

    For completeness, we conclude this section with an observation on the action of the operator $\UL{\ell}$ on forms which are $\ell$-new but \emph{not} twist-minimal at the same prime.
    \begin{lemma}\label{lem: Action of Uell on non-twist-minimal forms}
        The $\UL{\ell}$-eigenvalue on the space $S_k^D(Np^n\ell^2,\chi)$ is $\pm \sqrt{\chi_{\A,Np^n}(\ell)^{-1}}$.
    \end{lemma}
    \begin{proof}
        Once again, looking at the local description. If $\pi_\ell$ is an admissible irreducible representation of $D_\ell^\times$ of conductor 1, then it is (inflated from) a character $\pi_\ell \colon \langle \varpi_{D_\ell} \rangle \to \C^\times$. In particular, it is determined by $\pi_\ell(\varpi_{D_\ell})$ and once again $\pi_\ell(\varpi_{_\ell}^2) = \pi(\varpi_{\ql})^2 = \chi(\varpi_{\ql})^2$.
     \end{proof}

\subsubsection{$\Lambda$-action and operators}\label{Section: Lambda-action and operators}

Recall the notation introduced in Section~\ref{Section: Algebraic quaternionic modular forms}; equation~\eqref{Eq: Lambda action} shows that the $\Lambda$-action commutes with Hecke operators and also with the $\UL{\ell}$-operators. Considering the $\Lambda$-action on quaternionic modular forms is critical for ensuring the necessary properties of the morphisms between eigenvarieties we construct in~Section \ref{Section: Eigenvarieties}. More precisely, we make use of the relation
\begin{equation}\label{Eq: relation between UL^2 and Lambda-action}
    \UL{\ell}^2\, \varphi = \chi_{\A,\ell}(\ell)\, [\langle\ell\rangle]_\Lambda \cdot \varphi,
\end{equation}
for $\langle\ell\rangle\in 1+p\zp$ the image of $\ell$ under the projection $\langle\cdot\rangle:\zp^\times \longrightarrow 1+p\zp$.

\subsection{Pairings}\label{section: Pairings}

We introduce here the general setting we deal with in Sections \ref{Section: The JL correspondence in families} and~\ref{Section: Balanced triple product $p$-adic $L$-function}. Let $D$ be a definite quaternion algebra over $\Q$; it is ramified at an odd number of primes $\ell\,||\disc(D)$, for $\disc(D)$ the discriminant of $D$. We denote by $R_{N_+,N_-}$ a fixed Pizer order of level $N_+N_-$, with $N_-=N_{-}^{\mathrm{sp}}\,N_{-}^{\mathrm{sc}}$, such that
\begin{align}\label{eq: level hypothesis}
    (N_+,\,\disc(D))=1, && (N_{-}^{\mathrm{sp}},\, N_{-}^{\mathrm{sc}})=1, && N_{-}^{\mathrm{sp}}\, || \disc(D),&& N_{-}^{\mathrm{sc}}\, || \disc(D)^2.
\end{align}
We will occasionally simplify this notation to $R_{N} = R_{N_+, N_-}$ where $N = N_+ N_-$. The notation refers to the local factors of the corresponding representations of $\GL_2(\A)$: they are {\em special} at primes dividing $N_{-}^{\mathrm{sp}}$ and {\em supercuspidal}  at primes dividing $N_{-}^{\mathrm{sc}}$. As discussed above, we need the extra operator $\UL{\ell}$ at primes $\ell | N_{-}^{\mathrm{sc}}$. We remark that $(N_{-}^{\mathrm{sp}})^2 N_{-}^{\mathrm{sc}}=\disc{D}^2$. Moreover, we fix $\chi$ to be a Dirichlet character of conductor $c\mid N_+N_-$, and a weight $k\in\Z_{\geq 2}$.

\subsubsection{The quaternionic Petersson product}\label{Section: The quaternionic Petersson product} 
    
    We recall here the definition of the Petersson product between quaternionic modular forms; our definition extends the one provided in \cite[Section 4.2]{Hsieh2021}.
    Following \cite[Section 2.1]{GS2019}, we consider the \emph{$p$-adic norm form} $\Np:D^\times\backslash\widehat{D}^\times \longrightarrow \zp^\times$ given by
    \begin{equation}
        \Np(-)=\big(|\nu_{\A_f}(-)|_{\A_f} \, \nu_{\A_f}(-)\big)_p,
    \end{equation}
    for $(-)_p:\A_f^\times\longrightarrow\qp^\times$, the projection to the $p$-component. In other words, $\Np$ is defined as the $p$-component of the normalization of the reduced norm by the adèlic absolute value. With the notation of \cite[Section 4.2]{Hsieh2021}, we denote the \emph{$p$-adic cyclotomic character} by $\cyc$; this is the character
    \begin{align}
        \cyc:\Q_+\backslash \A_{f}^\times & \longrightarrow \zp^\times \nonumber \\
        a & \mapsto \big(|a|_{\A_f} \, a\big)_p = |a|_{\A_f} \,a_p. \label{Eq:p-adic cyclotomic character}
    \end{align}
    Therefore, $N_p = \cyc \circ \nu_{\A_{f}}$, as the quaternion algebra $D$ is definite.
    
    We consider the pairing
    \begin{equation}
        (-,-)_{N_+,N_-}:S_k^D(R_{N_+,N_-},\chi) \times S_k^D(R_{N_+,N_-},\chi^{-1})\longrightarrow \qpbar,
    \end{equation}
    given by
    \begin{equation}
        (\varphi_1,\varphi_2)_{N_+,N_-}=\sum\limits_{[x]\in D^\times\backslash\widehat{D}^\times/\widehat{R}_{N_+,N_-}^\times}\frac{\Np(x)^{k-2}}{\#\Gamma_{N_+,N_-}(x)}\langle\varphi_1(x),\varphi_2(x)\rangle_{k-2},
    \end{equation}
    where $\Gamma_{N_+,N_-}(x)=(D^\times \cap x \widehat{R}_{N_+,N_-}^\times x)\,\Q^\times / \Q^\times$ and
    \begin{equation}
        \langle-,-\rangle_{k-2}:L_{k-2}(\qpbar)\times L_{k-2}(\qpbar)\longrightarrow \qpbar
    \end{equation}
    is the paring defined by
    \begin{equation}
        \langle X^iY^{k-2-i},X^jY^{k-2-j}\rangle_{k-2}=\begin{cases}
            (-1)^i\binom{k-2}{i}^{-1} & \textnormal{ if } i+j = k-2,\\
            0 & \textnormal{ if }i+j\neq k-2.
        \end{cases}
    \end{equation}
    For any $g\in \GL_2(\qpbar)$, let $\Bar{g}=\det(g)g^{-1}$ be its image under the canonical involution on $\GL_2(\qpbar)$. Given such $g$ and any $P,\,Q\in L_{k-2}(\qpbar)$, we have:
    \begin{equation}
        \langle g\cdot P,Q\rangle_{k-2}=\langle P,\Bar{g}\cdot Q\rangle_{k-2}.
    \end{equation}
    That implies that the pairing $(-,-)_{N_+,N_-}$ is well defined, as
    \begin{align}  
        \Np(xr)^{k-2}\langle\varphi_1(xr),\varphi_2(xr)\rangle_{k-2} & =\Np(x)^{k-2} \nu_p(r_p)^{k-2}\langle r_p^{-1}\cdot \varphi_1(x),r_p^{-1}\cdot \varphi_2(x)\rangle_{k-2}\,\nonumber \\ & = \nu_p(r_p)^{k-2} \nu_p(r_p^{-1})^{k-2}\Np(x)^{k-2}\langle \varphi_1(x),\varphi_2(x)\rangle_{k-2}\,\\
        & =\Np(x)^{k-2}\langle\varphi_1(x),\varphi_2(x)\rangle_{k-2}, \nonumber
    \end{align}
    for $r\in \widehat{R}_{N_+,N_-}^\times$ and $[x]\in D^\times\backslash\widehat{D}^\times/\widehat{R}_{N_+,N_-}^\times$.
    \begin{definition}\label{def: quaternionic Atkin--Lehner}
        With the notation of Section~\ref{Section: A remark on the structure of quaternion algebras at ramified primes} (cf.\  also \cite[Section 9]{Pizer80p2}), let $\AL{N_+,N_-}\in \widehat{D}^\times$ the element determined by
        \begin{equation*}
            \AL{N_+,N_-, q}= \begin{cases}
                1 & \textnormal{ if }q\nmid N_+N_-,\\
                \iota_q^{-1}\left(\mat{0}{1}{-q^{v_q(N_+N_-)}}{0}\right) & \textnormal{ if }q\mid N_+,\\
                1 & \textnormal{ if }q\mid N_{-}^{\mathrm{sp}},\\
                \varpi_{\qll} & \textnormal{ if }q=\ell\mid N_{-}^{\mathrm{sc}}.
            \end{cases}
        \end{equation*}
        We define the Atkin--Lehner operator
        \begin{equation*}
            [\AL{N_+,N_-}]:S_k^D(R_{N_+,N_-},\chi)\longrightarrow S_k^D(R_{N_+,N_-},\chi^{-1}),
        \end{equation*}
        as
        \begin{equation*}
            [\AL{N_+,N_-}] \,\varphi(x) = \AL{N_+,N_-, p}\cdot\varphi(x \AL{N_+,N_-})\, \chi_\A(\nu(x)).
        \end{equation*}
    \end{definition}
    \begin{lemma}
        The Atkin-Lehner operator is a well-defined involution.
    \end{lemma}
    \begin{proof}
        It is easy to notice that the level of $[\AL{N_+,N_-}] \,\varphi$ is still $R_{N_+,N_-}$, but we need to check that its character is $\chi^{-1}$. Given $r\in \widehat{R}_{N_+,N_-}^\times$, we compute
        \begin{equation*}
            ([\AL{N_+,N_-}] \varphi)(xr)=\AL{N_+,N_-, p}\cdot \varphi(x \,r\, \AL{N_+,N_-})\,\chi_\A(\nu(x))\,\chi_\A(\nu(r)).
        \end{equation*}
        We must only check the primes $\ell\mid N_{-}^{\mathrm{sc}}$, as the local behavior at the other primes is the one of the usual Atkin--Lehner operator. Let $\ell$ be such a prime. With the formalism of Section \ref{Section: A remark on the structure of quaternion algebras at ramified primes}, we note that, for $r=x+\varpi_{\qll}y+\varpi_{D_\ell} z\in R_\ell^\times$, with $x,y\in \Z_\ell$ and $z\in \O_{D_\ell}$,
        \begin{equation*}
            r \varpi_{\qll}=x\varpi_{\qll}+\varpi_{\qll}y\varpi_{\qll}+\varpi_{D_\ell} z \varpi_{\qll} = \varpi_{\qll} ( x - y \varpi_{\qll}) + \varpi_{D_\ell} z \varpi_{\qll}.
        \end{equation*}
        Therefore, $r\varpi_{\qll} \equiv \varpi_{\qll} r\Mod{\varpi_{D_\ell}\O_{D_\ell}}$, hence $\widetilde{\chi}_\ell(r\,\varpi_{\qll})=\widetilde{\chi}_\ell(\varpi_{\qll} r)$. As
\begin{equation*}
	\varpi_{\qll} R_{N_+,N_-,\ell}= R_{N_+,N_-,\ell}\,\varpi_{\qll}
\end{equation*}        
we deduce that the operator is well defined, and inverting the character. We conclude noticing that $[\AL{N_+,N_-}] [\AL{N_+,N_-}] \,\varphi(x) = (\AL{N_+,N_-, p})^2\cdot\varphi\left(x (\AL{N_+,N_-})^2\right)$, and
        \begin{equation*}
            (\AL{N_+,N_-, q})^2= \begin{cases}
            	1 & \textnormal{ if }q\nmid N_+N_-,\\
                -q^{v_q(N_+N_-)} & \textnormal{ if }q\mid N_+,\\
                1 & \textnormal{ if }q\mid N_{-}^{\mathrm{sp}},\\
                \varpi_{\qll}^2 & \textnormal{ if }q=\ell\mid N_{-}^{\mathrm{sc}}.
            \end{cases}
        \end{equation*}
        Therefore, $(\AL{N_+,N_-, q})^2\in\Z_q$, for each $q$.
    \end{proof}
    \begin{proposition}\label{prop: quaternionic Petersson pairing is perfect and equivariant}
        The pairing
        \begin{equation*}
            \langle-,-\rangle_{N_+,N_-}:S_k^D(R_{N_+,N_-},\chi)\times S_k^D(R_{N_+,N_-},\chi) \longrightarrow \qpbar
        \end{equation*}
         defined by $\langle \varphi_1,\varphi_2\rangle_{N_+,N_-}=(\varphi_1,[\AL{N_+,N_-}]\varphi_2)_{N_+,N_-}$ is perfect and equivariant for both the action of the Hecke operators and of $\UL{\ell}$.
    \end{proposition}
    \begin{proof}
        The Hecke-equivariance follows exactly as in \cite[Section 4.2]{Hsieh2021} from \cite[Lemma 3.5]{Hida2006}. Since $\ell^2||\disc(D)$, it is not difficult to compute that $\Gamma_{N_+,N_-}(x)=\Gamma_{N_+,N_-}(x\,\varpi_{D_\ell})$. Writing temporarily $X = D^\times\backslash\widehat{D}^\times/\widehat{R}_{N_+,N_-}^\times$, we have that:
        \begin{align*}
            (\UL{\ell}\,\varphi_1,\varphi_2)_{N_+,N_-} & = \sum\limits_{[x]\in X} \frac{\Np(x)^{k-2}}{\#\Gamma_{N_+,N_-}(x)}\langle \varphi_1(x),\AL{N_+,N_-, p}\cdot \varphi(x\AL{N_+,N_-})\chi_\A(\nu(x))_{N_+,N_-} \rangle_{k-2} \\ 
            & \hspace*{-70pt} = \hspace*{-4pt}\sum\limits_{[x] \in X} \frac{\chi_{\A}^{-1}(\nu(\varpi_{D_\ell}))\,\Np(x)^{k-2}}{\Np(\varpi_{D_\ell})^{k-2}\,\#\Gamma_{N_+,N_-}(x)}\langle \varphi_1(x),\AL{N_+,N_-, p}\cdot \varphi(x \varpi_{D_\ell}^{-1} \AL{N_+,N_-})\chi_\A(\nu(x))_{N_+,N_-} \rangle_{k-2} \\
            & \hspace*{-70pt} = \chi_{\A}^{-1}(\nu(\varpi_{D_\ell})) \sum\limits_{[x]\in X}\frac{\Np(x)^{k-2}}{\#\Gamma_{N_+,N_-}(x)}\langle \varphi_1(x),\AL{N_+,N_-, p}\cdot \varphi(x \AL{N_+,N_-, p} \varpi_{D_\ell} \nu(\varpi_{D_\ell})^{-1})\, \chi_\A(\nu(x))_{N_+,N_-} \rangle_{k-2}\\
            & \hspace*{-70pt} = (\varphi_1,\UL{\ell}\,\varphi_2)_{N_+,N_-}
        \end{align*}
        where the second equality is obtained by the substitution $x\mapsto x\,  \varpi_{D_\ell}^{-1}$, and we have used $\Np(\varpi_{D_\ell})^{k-2} = 1$ and $\varpi_{D_\ell}^{-1}\AL{N_+,N_-, p} = \AL{N_+,N_-, p} \varpi_{D_\ell} (-\varpi_{D_\ell}^{-2}) = \AL{N_+,N_-, p} \varpi_{D_\ell} \nu(\varpi_{D_\ell})^{-1}$.
    \end{proof}

\subsubsection{Automorphic and quaternionic pairings}

    Let $\chi_\A$ be the adèlization of $\chi$ and $R_{N_+,N_-}$ a Pizer order in $D$ of level $N_+N_-$, we denote by $\Ac^{D}_{2}(R_{N_+,N_-},\chi_\A)$ the space of scalar-valued (i.e. weight 2) automorphic forms of level $\widehat{R}_{N_+,N_-}$ and character $\chi_\A$. For any two such forms $f\in \Ac^{D}_{2}(R_{N_+,N_-},\chi_\A)$ and $f'\in \Ac^{D}_{2}(R_{N_+,N_-},\chi_\A^{-1})$, we define the pairing
    \begin{equation}
        \langle f,f' \rangle =\int_{\A^\times D^\times\backslash D(\A)^\times} f(x)\,f'(x)\,dx
    \end{equation}
    for $dx$ the Tamagawa measure on $\A^\times\backslash D(\A)^\times$.
    
    Let $\varphi\in S_k^{D}(R_{N_+,N_-},\chi,\qpbar)$ be a quaternionic modular form of weight $k$ and character $\chi$, and let $u$ be a polynomial in $L_{k-2}(\C)$. We can define the matrix coefficient $\Phi_u(\varphi)\in \Ac^{D}_{2}(R_{N_+,N_-},\chi_\A^{-1})$ as $\Phi_u(\varphi)(x)=\langle\Phi(\phi)(x),u\rangle_{k-2}$, for any $x\in D(\A)^\times$ (recall equation~\eqref{eq: lift from QMFs to AF on D}). The above pairing is related to the one defined in Proposition~\ref{prop: quaternionic Petersson pairing is perfect and equivariant} via the following lemma.
    \begin{lemma}\label{lemma:Hsieh4.5}
         Let $\varphi$ be as above and consider $u$ and $v$ in $L_{k-2}(\C)$. Hence,
        \begin{equation*}
            \langle [\AL{N_+,N_-}]\,\Phi_u(\varphi),\Phi_v(\varphi) \rangle = \frac{\vol(R_{N_+,N_-}^\times)}{\left(N_+N_-^{\mathrm{sc}}\right)^{(k-2)/2}  (k-1)} \, \langle\varphi,\varphi\rangle_{N_+,N_-} \,\langle u ,v\rangle_{k-2}.
        \end{equation*}
    \end{lemma}
    \noindent Before proving the lemma, we remark that if $R$ is an Eichler order of level $N_+$, the Atkin--Lehner operator becomes $\AL{N_+,1}$, and we recover the formula in \cite[equation (4.5)]{Hsieh2021}.
    \begin{proof}[Proof of Lemma~\ref{lemma:Hsieh4.5}]
        This is an application of Schur orthogonality relations for linear pairings once we split the integration over $\A^\times D^\times\backslash D(\A)^\times = \A_f^\times D^\times\backslash \widehat{D}^\times\times \R^\times \backslash D_\infty^\times$. More precisely, one can proceed as in the proof of \cite[Lemma 3.2]{GS2019}, obtaining
        \begin{equation*}
            \langle \chi_\A\,\Phi_u(\varphi),\Phi_v(\varphi) \rangle = \frac{\vol(R_{N_+,N_-}^\times)}{(k-1)} \, \langle\chi_\A\,\varphi,\varphi\rangle_{N_+,N_-} \,\langle u ,v\rangle_{k-2}.
        \end{equation*}
        The factor $\nu\left(\AL{N_+,N_-}\right)^{(k-2)/2}$ arises from the definition of the isomorphism $f$ in \cite[equation (16)]{GS2019}; one must notice that they work with automorphic forms valued in the dual space $(L_{k-2}(\C))^\vee$, so their action at $\infty$ is obtained by inverting $x_\infty$. Expressing everything under our definitions and our choices of normalizations we obtain the stated equality. We conclude noticing that $\nu\left(\AL{N_+,N_-}\right)=N_+N_-^{\mathrm{sc}}$.
    \end{proof}
    
\section{The JL correspondence in families}\label{Section: The JL correspondence in families}

As in Section \ref{section: Pairings}, we take $D$ and $R_{p^nN_+,N_-}$ for $n\in\Z_{\geq 0}$ and $p\geq 5$ a prime with $(p,N_+N_-)=1$, satisfying the conditions in equation~\eqref{eq: level hypothesis}. We also fix a Dirichlet character $\chi$ of conductor $c\mid N_+N_-$ and assume that, for all primes $q\mid N_-$,
\begin{align}\label{eq: conductor hypothesis}
    v_q(c) = v_q(N_{-}^{\mathrm{sp}}) -1 = 0,\,\,\, \textnormal{if}\,\,\, q \mid N_{-}^{\mathrm{sp}}, && \textnormal{and} && v_q(c)\leq v_q(N_{-}^{\mathrm{sc}})-1 = 1 \,\,\, \textnormal{if}\,\,\, q \mid N_{-}^{\mathrm{sc}}.
\end{align}

This section is devoted to the study of quaternionic $p$-adic Hida families, with level structure given by such Pizer orders. Our eventual goal is to study a Jacquet--Langlands transfer in families for Coleman/Hida families on $\GL_2$ of level $N$. We show that for any collection of signs $\epsilon_\ell$ for each $\ell | N_-^{\rm sc}$ corresponding to eigenvalues of $\UL{\ell}$, there is a Coleman/Hida family on $D^\times$ that transfers to a given Coleman/Hida family on $\GL_2$. To that effect, we study eigenvarieties for the Hecke algebra augmented by the extra operators $\UL{\ell}$, and obtain an open immersion of eigenvarieties \'a la Chenevier~\cite{Chenevier2004}. The main tool is therefore Chenevier's interpolation technique, adjusted for dealing with this more general setup. We conclude with a control theorem in the spirit of Hida (see Corollary~\ref{cor: Hida's Control Theorem}). 

The main reference for this section is Bellaiche's book~\cite{Bellaiche2021}, together with~\cite{Buzzard2007},~\cite{Chenevier2004}, and \cite{Ludwig2017}. We try to keep the discussion as brief as possible; we recall only the necessary notions and definitions, and provide references where details can be found.

\subsection{Eigenvarieties}\label{Section: Eigenvarieties}

Let $\W$ denote the weight space, i.e. the rigid analytic space over $\qp$ whose $\qpbar$-points are
\begin{equation}
    \W(\qpbar)=\Hom_{\rm cts}(\zp^\times, \qpbar^\times).
\end{equation}
We freely identify any $\qpbar$-point $\kappa\in\W(\qpbar)$ with the corresponding character of $\zp^\times$. The space $\W$ is endowed with an admissible covering $\mathscr{C}$ by admissible affinoid open subsets, as constructed in \cite[Section 6]{Buzzard2007}; in loc.\ cit.\ the covering is considered on the Fredholm variety, while here we prefer to deal only with the weight space, as in~\cite{Bellaiche2021}, out of convenience of exposition. For any $X= \mathrm{Sp}(A)\in \mathscr{C}$, we consider the space of cuspidal overconvergent $p$-adic modular forms as defined in \cite[Section 6]{Buzzard2007} and denote it by $\mathbb{S}_{N_+N_-,\chi}^\dagger(X)$ (in loc. cit. $\mathbb{S}_{N_+N_-,\chi}^\dagger(X_n)$ is denoted by $M_n$, for $X_n\in\mathscr{C}$). This is a Banach $A$-module and satisfies \emph{Property} (Pr) as in \cite[Section 3.1.6]{Bellaiche2021}. Similarly, we consider the space of cuspidal overconvergent $p$-adic quaternionic modular forms as defined in \cite[Section 9]{Buzzard2007}, and denote it by $\mathbb{S}_{N_+N_-,\chi}^{D,\dagger}(X)$ (in loc. cit. $\mathbb{S}_{N_+N_-,\chi}^{D,\dagger}(X_n)$ is denoted by $\mathbf{S}_\kappa^D(U;r)$, for $\kappa$ and $r$ associated with $X_n\in\mathscr{C}$); once again, this space is a Banach $A$-module and satisfies \emph{Property} (Pr). In order to keep the discussion of this section brief, we opt to not recall the precise definition of overconvergent automorphic forms and instead refer the interested reader to~\cite{Buzzard2007} and~\cite{Chenevier2004}. Let us remark that this decision will not compromise the proofs of the main results as they rely only on properties and operators on classical automorphic forms.\\

We are interested in the eigenvectors of $\UL{\ell}$, for each $\ell\mid N_{-}^{\mathrm{sc}}$; therefore, we fix once and for all a square root
\begin{equation}
    \sqrt{\chi_{\A,N_+N_-/\ell^2}(\ell)^{-1}},
\end{equation}
which we call the \emph{positive} square root of $\chi_{\A,N_+N_-/\ell^2}(\ell)^{-1}$. We consider the field extension
\begin{equation}
    \mathcal{K}=\qp\left(\left\{\sqrt{\chi_{\A,N_+N_-/\ell^2}(\ell)^{-1}}\right\}_{\ell|N_{-}^{\mathrm{sc}}}\right)
\end{equation}
and let $\O$ be its ring of integers. Moreover, let $\mathbf{I'}=\O\llbracket 1+p\zp\rrbracket$ be the corresponding finite extension of the Iwasawa algebra $\Lambda$. As we will make use of equation~\eqref{Eq: relation between UL^2 and Lambda-action}, we need to extend $\mathbf{I'}$, adjoining the square roots of $[\langle\ell\rangle]_\mathbf{I'}$; we set
\begin{equation}
    \mathbf{I} = \mathbf{I'}\left[\left\{\sqrt{[\langle\ell\rangle]_\mathbf{I'}}\right\}_{\ell\mid N_{-}^{\mathrm{sc}}}\right]
\end{equation}
for the corresponding finite flat extension of $\Lambda$. Taking its normal closure in $\Frac(\mathbf{I})$ if needed, we can assume that $\mathbf{I}$ is a normal domain. As above, we fix a compatible choice of square roots
\begin{equation}
    \sqrt{[\langle\ell\rangle]_\mathbf{I'}}, \, \textrm{ for }\, \ell\mid N_{-}^{\mathrm{sc}},
\end{equation}
which we refer to as the \emph{positive} square roots; we shorten the notation and set
\begin{equation}\label{Eq: square root operators}
    \IL{\ell}=\sqrt{[\langle\ell\rangle]_\mathbf{I'}}, \, \textrm{ for }\, \ell\mid N_{-}^{\mathrm{sc}}.
\end{equation}

Let $\T_{\mathbf{I}}$ be the polynomial algebra over $\mathbf{I}$ generated by the Hecke operators $T_q$ for $q\nmid pN_+N_-$, $U_l$ for $l\mid pN_+N_{-}^{\mathrm{sp}}$ and the Diamond operators away form $pN_+N_-$; it is the Hecke algebra away from $\ell\mid N_{-}^{\mathrm{sc}}$, and we remark that $\T_{\mathbf{I}}$ is a commutative algebra over $\mathbf{I}$ with a distinguished element, the $U_p$ operator. We extend the $\Lambda$-action defined in Section \ref{Section: Lambda-action and operators} to an $\mathbf{I}$-action via the group-like elements. Similarly, we consider corresponding $\mathbf{I}$-action on the space of $p$-adic modular forms and automorphic forms (cf.\ \cite[Section 2.3]{Hsieh2021}).

Getting back to Banach modules, for any $X = \mathrm{Sp}(A)$ as above, there exist ring homomorphisms,
\begin{align}
     \psi_X:\T_{\mathbf{I}}\longrightarrow \End_A(\mathbb{S}_{N_+N_-,\chi}^\dagger(X)) && \textrm{ and } &&   \psi_X^D:\T_{\mathbf{I}}\longrightarrow \End_A(\mathbb{S}_{N_+N_-,\chi}^{D,\dagger}(X)),
\end{align}
such that the image of the $U_p$-operator under each map defines a compact operator (see~\cite[Section 6, Lemma 12.2]{Buzzard2007}). Moreover, for any $X'= \mathrm{Sp}(A') \in \mathscr{C}$, the $A'$-modules $\mathbb{S}_{N_+N_-,\chi}^\dagger(X)\widehat{\otimes}_A A'$ and $\mathbb{S}_{N_+N_-,\chi}^\dagger(X')$ (resp. $\mathbb{S}_{N_+N_-,\chi}^{D,\dagger}(X)\widehat{\otimes}_A A'$ and $\mathbb{S}_{N_+N_-,\chi}^{D,\dagger}(X')$) are \emph{linked} (see \cite[Definition 3.5.1]{Bellaiche2021}), as the $\mathbf{I}$-action commutes with Hecke operators (cf. Section \ref{Section: Lambda-action and operators}). 

\subsubsection{Classical eigenvarieties}\label{section: classical eigenvarieties}

    By \cite[Theorem 3.6.3]{Bellaiche2021} (cf. \cite[Theoreme 6.3.6]{Chenevier2004}), there exist two cuspidal eigenvarieties $\E_{\GL_2}$ and $\E_{D}$ associated, respectively, with the eigenvariety data
    \begin{align}
        \D^{\GL_2}_{N_+N_-,\chi}=(\W,\,\T_{\mathbf{I}},\, U_p,\,\mathbb{S}_{N_+N_-,\chi}^\dagger,\, \psi_{\GL_2}) && \textrm{ and } && \D^{D}_{N_+N_-,\chi}=(\W,\,\T_{\mathbf{I}},\, U_p,\,\mathbb{S}_{N_+N_-,\chi}^{D,\dagger},\, \psi_{D}).
    \end{align}
    
    We recall that the eigenvariety $\E_{\GL_2}$ (resp. $\E_{D}$) is a rigid analytic space over $\mathcal{K}$ endowed with
    \begin{itemize}
        \item a locally finite morphism $\omega_{\GL_2}:\E_{\GL_2}\longrightarrow \W$ (resp. $\omega_{D}:\E_{D}\longrightarrow \W$), called the \emph{weight map};
        \item a morphism of rings $\psi_{\GL_2}:\T_{\mathbf{I}}\longrightarrow \O(\E_{\GL_2})$ (resp. $\psi_{D}:\T_{\mathbf{I}}\longrightarrow \O(\E_{D})$), which sends $U_p$ to an invertible function;
        \item the morphism $\omega_{\GL_2}\times  \psi_{\GL_2}(U_p)^{-1}:\E_{\GL_2}\longrightarrow \W\times \A^1_{\rm rig}$ (resp. $\omega_{D}\times  \psi_{D}(U_p)^{-1}:\E_{D}\longrightarrow \W\times \A^1_{\rm rig}$) is finite (\cite[Proposition 3.7.7]{Bellaiche2021}).
    \end{itemize}
    The general definition of eigenvariety satisfies analogous properties and it can be found in \cite[Definition 3.6.2]{Bellaiche2021}.
\begin{remark}\label{rmk: def e properties of eigenvarieties}
    \leavevmode
    \begin{enumerate}
        \item By \cite[Theorem 3.6.3]{Bellaiche2021}, there exists an eigenvariety for each eigenvariety data and it is unique up to unique isomorphism.
        \item The eigenvariety construction of $\E_{D}$ provided in~\cite{Buzzard2007} applies in our case, as the adèlization $\widehat{R}^\times$ is a compact open subset of $\widehat{D}^\times$; cf. \cite[Lemma 2.3]{DallAva2021Hida}.
        \item By \cite[Proposition 3.7.7]{Bellaiche2021}, all the eigenvarieties over $\W$ are separated; moreover, as in~\cite{Ludwig2017}, it is enough for us to work with the reduced rigid analytic space associated with each eigenvariety. 
    \end{enumerate} 
\end{remark}
\begin{hypothesis}
    From now on, we assume all the eigenvarieties we consider throughout this note are reduced and separated.    
\end{hypothesis}

\begin{lemma}[{\cite[Lemma 5.9]{Buzzard2007}}]\label{lem: points on eigenvariety are eigensystems}
    Let $(\E,\psi,\omega)$ be either $(\E_{\GL_2},\psi_{\GL_2},\omega_{\GL_2})$ or $(\E_{D},\psi_{D},\omega_{D})$. For any discretely valued extension $\mathcal{K}'/\mathcal{K}$, the association
    \begin{align*}
        \E(\mathcal{K}') \longrightarrow \Hom_{\rm ring}(\T_{\mathbf{I}},\mathcal{K}')\times \W(\mathcal{K}') && \textrm{ mapping } && z\longmapsto (\,[h\mapsto \psi(h)(z)],\, \omega(z)),
    \end{align*}
    defines a bijection between the $\mathcal{K}'$-valued points of $\E$ and the set of $\mathcal{K}'$-valued systems of eigenvalues with non-zero $U_p$-eigenvalue.
\end{lemma}
\begin{remark}
    Lemma~\ref{lem: points on eigenvariety are eigensystems} holds for any general eigenvariety (\cite[Theorem 3.7.1]{Bellaiche2021}).
\end{remark}

\begin{proposition}[$p$-adic extension of the JL correspondence]\label{prop: p-adic extension of the JL correspondence}
    The Jacquet--Langlands correspondence extends uniquely to a closed immersion
    \begin{equation*}
        \E_{D}\xhookrightarrow{\ph{000}JL_p\ph{000}} \E_{\GL_2}
    \end{equation*}
    compatible with the eigenvariety structure, that is, $\omega_{D}=\omega_{\GL_2}\circ JL_p$ and $\psi_{D}=JL_p^*\circ\psi_{\GL_2}$.
\end{proposition}
Before proving the proposition, we need to recall the notion of arithmetic points on $\W(\qpbar)$.
\begin{definition}
    We say that $\kappa\in\W(\qpbar)$ is {\rm classical} if there exists a point $z\in\E_{\GL_2}(\qpbar)$ such that $\omega_{\GL_2}(z)=\kappa$ and $z$ corresponds to a classical modular form; we denote by $\W^{\rm cl}(\qpbar)$ the set of classical points.
    We say that $\kappa\in\W(\qpbar)$ is \textit{arithmetic} if there exists a point $z\in\E_{\GL_2}(\qpbar)$ such that $\omega_{\GL_2}(z)=\kappa$ and $z$ corresponds to a form of weight-character $\kappa=(k,\varepsilon)\in\W(\qpbar)=\Hom(\zp^\times, \qpbar^\times)$ with $k\in\Z_{\geq2}$ and $\varepsilon$ a $p$-adic character of finite order. We denote by $\W^{\rm arith}(\qpbar)$ the set of arithmetic points. By Coleman classicality's theorem~\cite[Theorem 1.1]{Coleman1997}, we know that arithmetic points of small slope are classical.
\end{definition}
\begin{proof}[Proof of Proposition~\ref{prop: p-adic extension of the JL correspondence}]
    We want to apply \cite[Proposition 2.10]{Ludwig2017}. We define the \emph{very Zariski-dense} (see \cite[Definition 3.8.1]{Bellaiche2021}) subset $Z\subset \E_{D}(\qpbar)$, as the set of arithmetic $\qpbar$-points $z$ such that $\omega_{D}(z)\in\W^{\rm arith}(\qpbar)$. \emph{Coleman's classicality theorem} combined with \cite[Theorem 7.17]{HPS1989basis} (cf. also Proposition~\ref{prop: multiplicity-2}) implies that both $\E_{D}$ and $\E_{\GL_2}$ are endowed with a \emph{classical structure} (see \cite[Definition 3.8.5]{Bellaiche2021}) given by classical forms over $\W^{\rm arith}(\qpbar)$. \cite[Proposition 3.8.6]{Bellaiche2021} tells us that the subset of all classical points $\ZZ\subset\E_{D}$ is \emph{very Zariski-dense} in $\E_{D}$, thus, setting $Z=\ZZ(\qpbar)$ we can assume it to be very Zariski-dense in $\E_{D}(\qpbar)$; $Z$ embeds in $\E_{\GL_2}(\qpbar)$ by the classical \emph{Jacquet--Langlands correspondence} and \emph{Strong Multiplicity One} for automorphic representations. More precisely, Lemma~\ref{lem: points on eigenvariety are eigensystems} defines an injective association
    \begin{equation*}
        JL_p:\,\E_{D}^{\rm cl}(\qpbar)\ni z \longleftrightarrow \pi_{z}^{D}\longmapsto \pi_z \longleftrightarrow JL_p(z)\in \E_{\GL_2}^{\rm cl}(\qpbar),
    \end{equation*}
    where $\pi_z^D$ is the automorphic representation over $D^\times(\A)$ associated with the classical point $z$ and $\pi_z=JL(\pi_z^D)$. This inclusion is compatible with the structure of the two eigenvarieties by definition of the map. The uniqueness is ensured by the reducedness of the eigenvarieties together with \cite[Lemma 7.2.7]{BellaicheChenevier2009} and Lemma~\ref{lem: points on eigenvariety are eigensystems}.
\end{proof}

\subsubsection{The extended eigenvariety}

    The Hecke algebra $\T_{\mathbf{I}}$ acts on quaternionic modular forms with $R_{p^nN_+,N_-}$-level structure and, for each $\ell\mid N_{-}^{\mathrm{sc}}$, we can extend $\T_{\mathbf{I}}$ by adjoining the operator $\UL{\ell}$.
    
    \begin{definition}
        We define the $\mathcal{K}$-algebra
        \begin{equation*}
            \Ttilde_{\mathbf{I}}=\T_{\mathbf{I}}\Big[\UL{\ell}\,\Big|\, \ell\mid N_{-}^{\mathrm{sc}}\Big]
        \end{equation*}
        endowed with the canonical injection $\T_{\mathbf{I}}\hookrightarrow \Ttilde_{\mathbf{I}}$. We refer to it as the \emph{extended Hecke algebra}. By Proposition~\ref{prop: UL properties} and equation~\eqref{Eq: relation between UL^2 and Lambda-action}, $\Ttilde_{\mathbf{I}}$ is commutative.
    \end{definition} 
    While the algebra $\T_{\mathbf{I}}$ acts on both $\mathbb{S}_{N_+N_-,\chi}$ and $\mathbb{S}_{N_+N_-,\chi}^D$, $\Ttilde_{\mathbf{I}}$ acts only on $\mathbb{S}_{N_+N_-,\chi}^D$.
    We can consider the tuple
    \begin{equation}
        \widetilde{\D}^{D}_{N_+N_-,\chi}=(\W,\,\Ttilde_{\mathbf{I}},\, U_p,\,\mathbb{S}_{N_+N_-,\chi}^{D,\dagger},\, \psi_{D}),
    \end{equation}
    which defines an eigenvariety datum, as one only needs to prove the following lemma.
    \begin{lemma}\label{lem: extended links}
        For any pair of affinoid subdomains, $X'=\Sp(A')\subseteq X=\Sp(A)\subset \W$ with $X,X'\in\mathscr{C}$, the $A'$-modules $\mathbb{S}_{N_+N_-,\chi}^{D,\dagger}(X)\widehat{\otimes}_{A}A'$ and $\mathbb{S}_{N_+N_-,\chi}^{D,\dagger}(X')$ are linked (see \cite[Definition 3.5.1]{Bellaiche2021}) under $\Ttilde_{\mathbf{I}}$.
    \end{lemma}
    \begin{proof}
        We apply \cite[Lemma 3.5.2]{Bellaiche2021} to  \cite[Lemma 12.2]{Buzzard2007}; we must prove that the two morphisms are $\Ttilde_{\mathbf{I}}$-equivariant. The first inclusion morphism is $\T_{\mathbf{I}}$-equivariant (recall that we are changing the structure only at the ramified places). Thus, the properties of $\UL{\ell}$ in Proposition~\ref{prop: UL properties} together with the relation in Section~\ref{Section: Lambda-action and operators} imply the $\Ttilde_{\mathbf{I}}$-equivariance of both morphisms.
    \end{proof}
    We deduce the following proposition.
    \begin{proposition}\label{prop: existence of extended eigenvariety}
        There exists a reduced eigenvariety associated with the datum $\widetilde{\D}^{D}_{N_+N_-,\chi}$, which we denote by $(\Etilde_{D},\widetilde{\omega}_{D},\widetilde{\psi}_{D})$. We will refer to it as the \emph{extended eigenvariety for $D$}.
    \end{proposition}
    This variety plays the role of an auxiliary eigenvariety, whose points parameterize system of Hecke eigenvalues with finite $U_p$-slope together with a collection of additional eigenvalues associated with the $\UL{\ell}$-operators. In particular, applying Lemma~\ref{lem: points on eigenvariety are eigensystems} to the inclusion $\T_{\mathbf{I}}\hookrightarrow\Ttilde_{\mathbf{I}}$, we obtain a morphism on the $\cp$-points of the eigenvarieties. This canonical map extends to a morphism of eigenvarieties; in fact, more is true.

    \begin{proposition}\label{prop: finite map from E_D tilde to E_D}
        The canonical inclusion $\T_{\mathbf{I}}\hookrightarrow \Ttilde_{\mathbf{I}}$ defines a unique finite morphism,
        \begin{equation*}
            \widetilde{\pi}:\Etilde_{D}\xrightarrow{\ph{000000}} \E_{D},
        \end{equation*}
        compatible with the eigenvariety structure. Moreover, it sends classical points to classical points.
    \end{proposition}
    \begin{proof}
        We follow the proof of \cite[Proposition 3.15]{Ludwig2017} and construct the map on the local pieces of the eigenvariety. More precisely, we take $\Etilde_{D,X}=\widetilde{\omega}_{D}^{-1}(X)\subset  \Etilde_{D}$, for any $X=\Sp(A)$ in the admissible covering $\mathscr{C}$. We consider the admissibile covering $\{\Etilde_{D,X}(U)\}_{U}$ indexed over connected open affinoids $U\subseteq X$; by construction, each $\Etilde_{D,X}(U)$ is an affinoid $\Sp(\Ttilde_{\mathbf{I}}(U))$, for
        \begin{equation*}
            \Ttilde_{\mathbf{I}}(U) = \Image\left(\Ttilde_{\mathbf{I}}\otimes_{\mathcal{K}}\O(U)\longrightarrow \End_{\O(U)}\left(\left(\mathbb{S}_{N_+N_-,\chi}^{D,\dagger}(X)\widehat{\otimes}_{\O(X)}\O(U)\right)^{\textnormal{finite slope}}\right)\right)^{\rm red}    
        \end{equation*}
        the reduced image of $\Ttilde_{\mathbf{I}}\otimes_{\mathcal{K}}\O(U)$. For any $U$ as above, $\T_{\mathbf{I}}\hookrightarrow \Ttilde_{\mathbf{I}}$ induces the inclusion $\T_{\mathbf{I}}(U)\hookrightarrow \Ttilde_{\mathbf{I}}(U)$; therefore, we obtain
        \begin{equation*}
            \iota^*_{U}: \Etilde_{D,X}(U)\longrightarrow \E_{D,X}(U).
        \end{equation*}
        Note that the $\E_{D,X}(U)$ form an admissible covering of $\E_{D,X}$. The composition of $\iota^*_{U}$ with the natural inclusion $ \E_{D,X}(U)\hookrightarrow  \E_{D,X}$ determines the morphism
        \begin{equation*}
            \widetilde{\pi}_{U}:\Etilde_{D,X}(U)\longrightarrow \E_{D,X}.
        \end{equation*}
        Let $X'$ be an admissible open in $\mathscr{C}$ and take $U\subset X$ and $U'\subset X'$, connected open affinoids. Up to refining the covering, we can assume that $X\subseteq X'$. Invoking \cite[Lemma 5.2]{Buzzard2007} as in \cite[Proposition 3.15]{Ludwig2017}, we deduce that $\widetilde{\pi}_{U}|_{\Etilde_{D,X}(U\cap U')}=\widetilde{\pi}_{U\cap U'}$, therefore the morphisms $\widetilde{\pi}_{U}$ glue compatibly. Again, the uniqueness is ensured by the reducedness of the eigenvarieties together with \cite[Lemma 7.2.7]{BellaicheChenevier2009} and Lemma~\ref{lem: points on eigenvariety are eigensystems}. Let now $z$ be a classical point in $\Etilde_{D}(\qpbar)$. Such $z$ corresponds to a classical eigenform $\varphi_z$ identified by a system of Hecke and $\UL{\ell}$-eigenvalues for $\ell\mid N_{-}^{\mathrm{sc}}$, say $\widetilde{\lambda}_z$. The morphism $\widetilde{\pi}$, by construction, maps $\widetilde{\lambda}_z$ to the eigensystem $\lambda_z$ obtained by forgetting the $\UL{\ell}$-eigevalues. Lemma~\ref{lem: points on eigenvariety are eigensystems} shows that $\pi(z)$ is classical. Finiteness follows from the relation
        \begin{equation*}
            \UL{\ell}^2\,-\chi_{N_+N_-/\ell^2}(\ell)^{-1}\cdot [\langle\ell\rangle]_\mathbf{I} = 0
        \end{equation*}
        obtained from equation~\eqref{Eq: relation between UL^2 and Lambda-action}.
    \end{proof}

\subsubsection{Idempotents}
    For any $\ell\mid N_{-}^{\mathrm{sc}}$, recall the square root operators $\IL{\ell}$ defined in equation~\eqref{Eq: square root operators}. They are invertible elements of the algebra $\mathbf{I}$, as the $[\langle\ell\rangle]_\mathbf{I}$ are invertible. Therefore, they determine invertible operators. We can hence define idempotents in $\mathbf{I}$ and on $\mathbb{S}_{N_+N_-,\chi}^D$.
    \begin{definition}\label{def: idempotents}
        For each prime $\ell\mid N_{-}^{\mathrm{sc}}$, we set
        \begin{equation*}
            e_{\pm}^{D_\ell} = \frac{1}{2}\left(\mathds{1}\mp \frac{1}{\IL{\ell}\,\sqrt{\chi_{\A,N_+N_-/\ell^2}(\ell)^{-1}}}\,\UL{\ell}\,\right).
        \end{equation*}
        Moreover, for each tuple of signs $\epsilon=(\epsilon_\ell)_{\ell\mid N_{-}^{\mathrm{sc}}}\in \{\pm 1\}^{\#\{\ell \mid N_{-}^{\mathrm{sc}}\}}$, we define a canonical idempotent
        \begin{equation*}
            e_{\epsilon}^{D}=\prod\limits_{\ell\mid N_{-}^{\mathrm{sc}}}e_{\epsilon_\ell}^{D_\ell}.
        \end{equation*}
    \end{definition}
\begin{remark}
    \leavevmode
    \begin{enumerate}
        \item It is immediate from the definition that the sum $ e_{+}^{D_\ell}+ e_{-}^{D_\ell}$ is the identity element $\mathds{1}$ in $\Ttilde_{\mathbf{I}}$. Therefore, the sum of all these idempotents is
        \begin{equation*}
            \sum\limits_{\epsilon\in \{\pm 1\}^{\#\{\ell \mid N_{-}^{\mathrm{sc}}\}}}e_{\epsilon}^{D} = \mathds{1}.
        \end{equation*}
        \item The operators $\IL{\ell}$ take care of the $p$-component of the square root as explained by equation~\eqref{Eq: Lambda action}.
    \end{enumerate}
\end{remark}
\begin{proposition}\label{prop: idempotents dec}
    \leavevmode
    \begin{enumerate}
        \item Let $\kappa=(k,\kappa_p)\in \W^{\rm cl}(\qpbar)$ be any weight-character and let $n=\max\{1,v_p(\cond(\kappa_p))\}$. For any $\ell\mid N_{-}^{\mathrm{sc}}$, let  $S_k^{D,\pm}(p^nN_+N_-,\chi\kappa_p)=e_{\pm}^{D_\ell}S_k^Dp^nN_+N_-,\chi\kappa_p)$. There is a canonical decomposition
        \begin{equation*}
            S_k^D(p^nN_+N_-,\chi\kappa_p)=S_k^{D,+}(p^nN_+N_-,\chi\kappa_p)\oplus S_k^{D,-}(p^nN_+N_-,\chi\kappa_p).
        \end{equation*}
        \item The projectors $e_{\pm}^{D_\ell}$ are orthogonal with respect to the quaternionic Petersson inner product defined in Section \ref{Section: The quaternionic Petersson product}.
    \end{enumerate}
\end{proposition}
\begin{proof}
    It is clear that the two idempotents $e_{\pm}^{D_\ell}$  give rise to projectors on $S_k^D(p^nN\ell^2,\chi\kappa_p)$, satisfying the identity $\mathds{1} = e_{+}^{D_\ell} + e_{-}^{D_\ell}$, i.e. $e_{-}^{D_\ell} = \mathds{1} - e_{+}^{D_\ell}$. Therefore, for any $\varphi\in S_k^D(p^nN\ell^2,\chi\kappa_p)$, we compute $e_{\pm}^{D_\ell}e_{\mp}^{D_\ell}\varphi=0$, in fact, $e_{\pm}^{D_\ell}e_{\mp}^D\varphi= e_{\pm}^{D_\ell}\varphi - e_{\pm}^{D_\ell}e_{\pm}^{D_\ell}\varphi= e_{\pm}^{D_\ell}\varphi-e_{\pm}^{D_\ell}\varphi=0$. This proves $(1)$. Part $(2)$ follows from Proposition~\ref{prop: quaternionic Petersson pairing is perfect and equivariant} and equation~\eqref{Eq: Lambda action}.
\end{proof}

 \subsubsection{The extended eigenvarieties of idempotent type}   
    We fix one of the idempotents in Definition~\ref{def: idempotents}, say $e_{\epsilon}^{D}$, and consider the tuple
    \begin{equation}
        e_{\epsilon}^{D}\widetilde{\D}^{D}_{N_+N_-,\chi}=(\W,\,\Ttilde_{\mathbf{I}},\, U_p,\,e_{\epsilon}^{D}\mathbb{S}_{N_+N_-,\chi}^{D,\dagger},\, \psi_{D}).
    \end{equation}

    \begin{proposition}\label{prop: existence of idempotent type extended eigenvariety}
        There exists a unique reduced eigenvariety associated with the datum $e_{\epsilon}^{D}\widetilde{\D}^{D}_{N_+N_-,\chi}$, which we denote by $(\Etilde_{D}^{\epsilon},\widetilde{\omega}_{D}^{\epsilon}=\widetilde{\omega}_{D|_{\Etilde_{D}^{\epsilon}}},\widetilde{\psi}_{D})$ and we will refer to it as the \emph{extended eigenvariety for $D$}.
    \end{proposition}
    \begin{proof}
        The above Proposition~\ref{prop: idempotents dec} guarantees that the projection by $e_{\epsilon}^{D}$ (which is linear and continuous) produces Banach modules satisfying \emph{Property $(Pr)$} (see \cite[Exercise 3.1.23]{Bellaiche2021}). Since $e_{\epsilon}^{D}$ commutes with all Hecke operators, Lemma~\ref{lem: extended links} shows that the modules are still linked.
    \end{proof}

\begin{proposition}\label{prop: closed embedding extended eigenvarieties}
    The natural map 
    \begin{equation*}
        \widetilde{\iota}^{\,\epsilon}:\Etilde_{D}^{\epsilon}\xhookrightarrow{\ph{000000}} \Etilde_{D},
    \end{equation*}
    is the unique closed immersion compatible with the eigenvariety structure. Moreover, $\Etilde_{D}$ is the disjoint union of the images of $\Etilde_{D}^{\epsilon}$, for all the $\epsilon\in \{\pm 1\}^{\#\{\ell \mid N_{-}^{\mathrm{sc}}\}}$.
\end{proposition}
\begin{proof}
    This is just \cite[Exercise 3.7.2 and Exercise 3.6.4]{Bellaiche2021}, together with the direct sum decomposition of Proposition~\ref{prop: idempotents dec}. Alternatively, one can also consider an approach similar to the proof of \cite[Proposition 2.10]{Ludwig2017}.
\end{proof}
\begin{proposition}\label{prop: JL_p^epsilon is closed immersion}
    The morphism $\widetilde{JL}_p^{\epsilon}$ of eigenvarieties, obtained by composition,
    \begin{equation*}
        \widetilde{JL}_p^{\epsilon} = JL_p \circ \widetilde{\pi} \circ \widetilde{\iota}^{\,\epsilon} : \Etilde_{D}^{\epsilon}\xhookrightarrow{\ph{000000}} \E_{\GL_2}
    \end{equation*}
    is a closed immersion.
\end{proposition}
\begin{proof}
    The morphism preserves the eigenvariety structures and it is finite, by composition of finite morphisms of eigenvarieties; recall that closed immersions are finite by definition, see e.g. \cite[Definition 4.5.7 and Definition 4.10.1]{FresnelVanderPut2004}. As in the proof of \cite[Lemma 2.9]{Ludwig2017}, we consider an affinoid $V\subset \W\times \A^1_{\rm rig}$, its preimage $X=(\widetilde{\omega}_{D}^{\epsilon}\times\widetilde{\psi}_{D}(U_p)^{-1})^{-1}(V)$ in $\Etilde_{D}^{\epsilon}$, and the natural \emph{surjective} map
    \begin{equation*}
        \widetilde{\psi}_{D}\otimes(\widetilde{\omega}_{D}^{\epsilon}\times\widetilde{\psi}_{D}(U_p)^{-1})^*\colon\Ttilde_{\mathbf{I}} \otimes_K \O(V) \longrightarrow \O(X).
    \end{equation*}
    To prove that $\widetilde{JL}_p^{\epsilon}$ is a closed immersion, it is enough to show that the  restriction 
    \begin{equation*}
        \widetilde{\psi}_{D}\otimes(\widetilde{\omega}_{D}^{\epsilon}\times\widetilde{\psi}_{D}(U_p)^{-1})^*:\T_{\mathbf{I}} \otimes_K \O(V) \longrightarrow \O(X).
    \end{equation*}
    is still surjective. This follows by the finiteness of $\O(X)$ as a $\O(V)$-module (see Section \ref{section: classical eigenvarieties}) and from the relation
    \begin{equation*}
        \widetilde{\psi}_{D}(\UL{\ell})= \epsilon_{\ell}\, \sqrt{\chi_{\A,\ell}(\ell)}\,\,\widetilde{\psi}_{D}(\IL{\ell}),\,\textrm{ for all }\, \ell \mid N_{-}^{\mathrm{sc}}.
    \end{equation*}
    We can hence follow through the proof of \cite[Proposition 2.10]{Ludwig2017}, considering finite subsets $I_V\subset \T_{\mathbf{I}}$ containing the operators $\IL{\ell}$ for $\ell \mid N_{-}^{\mathrm{sc}}$.
\end{proof}
We can summarize the above eigenvarieties and their morphisms in the following commutative diagrams:
\begin{equation} 
    \begin{tikzcd}
	\widetilde{\mathscr{E}}_{D}^{\,\epsilon} \arrow[rrrr, "\widetilde{\iota}^{\, \epsilon}", hook] \arrow[rrrrrrrrdd, "\widetilde{JL}_p^{\epsilon}", hook, bend left, shift left=2] \arrow[ddd, "\widetilde{\omega}^{\epsilon}_{D}\times  \widetilde{\psi}^{\epsilon}_{D}(U_p)^{-1}"'] &  &  &  & \widetilde{\mathscr{E}}_{D} \arrow[dd, "\widetilde{\pi}"] \arrow[llllddd, "\widetilde{\omega}_{D}\times  \widetilde{\psi}_{D}(U_p)^{-1}"'] &  &  &  &                                                                                                                   &  & \widetilde{\mathbb{T}}_{\mathbf{I}}      \\
	&  &  &  &                                                                                                                                            &  &  &  &                                                                                                                   &  &                             \\
	&  &  &  & \mathscr{E}_{D} \arrow[rrrr, "JL_p", hook] \arrow[lllld, "\omega_{D}\times  \psi_{D}(U_p)^{-1}"]                                           &  &  &  & \mathscr{E}_{\GL_2} \arrow[lllllllld, "\omega_{\GL_2}\times  \psi_{\GL_2}(U_p)^{-1}"', bend left, shift right] &  & {\,\mathbb{T}_{\mathbf{I}},} \arrow[uu, hook] \\
	\W\times \mathbb{A}^1_{\rm rig}                                                                                                                                                                                                                                                                                                        &  &  &  &                                                                                                                                            &  &  &  &                                                                                                                   &  &                            
    \end{tikzcd}
\end{equation}

\begin{equation}
    \begin{tikzcd}
	\O(\Etilde^{\,\epsilon}_{D})                                                                               &  &  &  & \O(\Etilde_{D}) \arrow[llll, "(\widetilde{\iota}^{\, \epsilon})^*"]      &  &  &  &                                                                                                                                       \\
	\widetilde{\mathbb{T}}_{\mathbf{I}} \arrow[rrrru, "\widetilde{\psi}_{D}"] \arrow[u, "\widetilde{\psi}^{\epsilon}_{D}"] &  &  &  &                                                                                      &  &  &  &                                                                                                                                       \\
	&  &  &  & \O(\E_{D}) \arrow[uu, "(\widetilde{\pi})^*"]                                        &  &  &  & {\O(\E_{\GL_2}).} \arrow[llll, "(JL_p)^*"'] \arrow[lllllllluu, "(\widetilde{JL}_p^{\epsilon})^*"', bend right, shift right] \\
	&  &  &  & \mathbb{T}_{\mathbf{I}} \arrow[lllluu, hook] \arrow[rrrru, "\psi_{\GL_2}"'] \arrow[u, "\psi_{D}"] &  &  &  &                                                                                                                                      
    \end{tikzcd}
\end{equation}

\subsubsection{The image of $\widetilde{JL}_p^{\epsilon}$}

We conclude this section with a couple of interesting observations which strengthen the comparison between our immersion and Chenevier's one. As explained in \cite[Section 2.1.1, Section 4.6, and Proposition 4.7.(b)]{Chenevier2005}, one can construct a reduced closed rigid analytic subvariety $\E_{\GL_2}^{\mathrm{tr}=0}\subseteq \E_{\GL_2}$ whose classical points correspond to forms which are new at the primes dividing the discriminant of $D$.

\begin{proposition}\label{prop: JL_p isomorphism onto tr = 0}
    The morphism $JL_p$ is an isomorphism onto $\E_{\GL_2}^{\mathrm{tr}=0}$.
\end{proposition}
\begin{proof}
    By Proposition~\ref{prop: p-adic extension of the JL correspondence} and the above discussion, we know that $JL_p$ is a closed immersion. Note that classical points are dense both in $\E_{\GL_2}^{\mathrm{tr}=0}$ and in $\E_D$, and hence it is enough to show that
    $$JL_p(\E_D^{\rm cl}) \supseteq \E_{\GL_2}^{\mathrm{tr} = 0, {\rm cl}}$$ in order to prove that the morphism is surjective. Each point in $\E_{\GL_2}^{\mathrm{tr} = 0, {\rm cl}}$ corresponds to a Hecke eigenform which, under our assumptions, can be transferred to the quaternion algebra.
\end{proof}

\begin{proposition}\label{prop: pi restrict is iso onto its image}
    The morphism $\widetilde{\pi}_{|\widetilde{\mathscr{E}}_{D}^{\,\epsilon}}$ is an isomorphism onto its image.
\end{proposition}
\begin{proof}
    We would like to apply \cite[Proposition 7.2.8]{BellaicheChenevier2009}, but $\Image(\widetilde{\pi}_{|\widetilde{\mathscr{E}}_{D}^{\,\epsilon}})$ is not necessarily an eigenvariety. However, modifying the sets $I_V$ as in the proof of Proposition~\ref{prop: JL_p^epsilon is closed immersion}, the proof of~\cite[Proposition 7.2.8]{BellaicheChenevier2009} applies verbatim to our setup.
\end{proof}

\begin{corollary}
    The composition $\widetilde{\pi}\circ \widetilde{\iota}^{\,\epsilon}$ is an open and closed immersion. Therefore, $\widetilde{JL}_p^{\epsilon}$ is an open and closed immersion into $\E_{\GL_2}^{\mathrm{tr}=0}$.
\end{corollary}
\begin{proof}
    The first assertion follows by combining Proposition~\ref{prop: pi restrict is iso onto its image} with Proposition~\ref{prop: closed embedding extended eigenvarieties}. The second one follows from Proposition~\ref{prop: JL_p isomorphism onto tr = 0}.
\end{proof}

\begin{remark}
    For any two distinct choices of signs, $\epsilon$ and $\epsilon'$, the images of the eigenvarieties $\Etilde_{D}^{\epsilon}$ and $\Etilde_{D}^{\epsilon'}$ are not disjoint in $\E_{\GL_2}$, since their intersection contains the twist-minimal forms.
\end{remark}

\subsection{Families of quaternionic modular forms}\label{Section: Families of quaternionic modular forms}

This section and the next one contain the main results we employ in the construction of the triple product $p$-adic $L$-function, namely the existence of families of quaternionic modular forms of finite slope. We follow the definition given in \cite[Section 6.2.6]{Chenevier2004}, and restrict our attention to the two eigenvarieties
\begin{align}
    \E_{\GL_2}\sim (\W,\,\T_{\mathbf{I}},\, U_p,\,\mathbb{S}_{N_+N_-,\chi}^\dagger,\, \psi_{\GL_2}) && \textrm{ and } && \Etilde^{\,\epsilon}_{D}\sim (\W,\,\Ttilde_{\mathbf{I}},\, U_p,\,e_{\epsilon}^{D}\mathbb{S}_{N_+N_-,\chi}^{D,\dagger},\,\widetilde{\psi}^{\epsilon}_{D}).    
\end{align}
To $\mathbb{S}_{N_+N_-,\chi}^\dagger$ (resp. $e_{\epsilon}^{D}\mathbb{S}_{N_+N_-,\chi}^{D,\dagger}$), one can associate a \emph{sheaf of Banach modules} on $\W$, which we denote by $\mathscr{S}_{N_+N_-,\chi}^\dagger$ (resp. $\mathscr{S}_{N_+N_-,\chi}^{D,\epsilon,\dagger}$); this geometric object is roughly obtained glueing the Banach modules of overconvergent forms compatibly. As we are only interested in its sheaf nature, we do not recall the precise definition here, but refer the reader to \cite[Section 3]{Chenevier2004} and \cite[Section 1]{Chenevier2005} for a thorough discussion. Let $\mathscr{T}_{\mathbf I}^{\GL_2}$ be the closure of $\T_{\mathbf{I}}$ in
\begin{equation}
    \mathcal{E}_{\mathbf{I}}^{\GL_2}=\left\{\, h \in \End(\mathscr{S}_{N_+N_-,\chi}^\dagger) \,\middle | \, h \textrm{ is integral and rational}\right\},
\end{equation}
where the topology is the coarsest topology such that, for any open affinoid $X\in\W$, the restriction map
\begin{equation}
    \mathcal{E}_{\mathbf{I}}^{\GL_2} \longrightarrow \End^{\rm cts}_{\O_X(X)}(\mathbb{S}_{N_+N_-,\chi}^{\dagger}(X))
\end{equation}
is continuous; here $\End^{\rm cts}_{\O_X(X)}(\mathbb{S}_{N_+N_-,\chi}^{\dagger}(X))$ is endowed with the topology induced by the supremum norm. Analogously (cf.\ \cite[Proposition 4.5.4]{Chenevier2004} and equation~\eqref{Eq: relation between UL^2 and Lambda-action}), we define $\widetilde{\mathscr{T}}_{\mathbf I}^{\,D,\epsilon}$ as the closure of $\Ttilde_{\mathbf{I}}$ in
\begin{equation}
    \mathcal{E}_{\mathbf{I}}^{D,\epsilon}=\left\{\, h \in \End(\mathscr{S}_{N_+N_-,\chi}^{D,\epsilon,\dagger}) \,\middle | \, h \textrm{ is integral and rational}\right\}.
\end{equation}

Let $(\E,\, \W,\,\mathbb{S}^\dagger,\,\psi,\,\omega,\,\mathscr{T})$ be either the tuple
\begin{align}
    (\E_{\GL_2},\, \W, \,\mathbb{S}_{N_+N_-,\chi}^\dagger, \,\psi_{\GL_2},\,\omega_{\GL_2},\mathscr{T}_{\mathbf I}^{\GL_2}) && \textrm{ or } && (\Etilde^{\,\epsilon}_{D}, \,\W,\, e_{\epsilon}^{D}\mathbb{S}_{N_+N_-,\chi}^{D,\dagger}, \,\widetilde{\psi}^{\epsilon}_{D},\,\widetilde{\omega}_{D}^{\epsilon},\,\widetilde{\mathscr{T}}_{\mathbf I}^{\,D,\epsilon}).
\end{align}
For any open affinoid $\mathcal{U}$ in $\W$, we denote the submodule of \emph{power-bounded elements} by
\begin{equation}
    \O(\mathcal{U})^0=\left\{s\in \O(\mathcal{U}) \,\big| \, |s(u)|\leq 1 \text{ for all } u\in \mathcal{U} \right\}.
\end{equation}
If $\mathcal{U}$ is reduced, which is always the case if $\mathcal{U}$ is small enough, $\O(\mathcal{U})^0$ is compact in $\O(\mathcal{U})$ (see \cite[Lemma 7.2.11]{BellaicheChenevier2009}). The same definition applies for any affinoid $\mathcal{X}\subset\E$.
\begin{definition}[Families of quaternionic cusp forms]\label{def: Families of (quaternionic) modular forms}
    Let $\kappa\in \W(\cp)$ and let $\varphi\in \omega^{-1}(\kappa)$ be a $p$-adic overconvergent cuspidal form. We define a \emph{family of quaternionic cuspidal modular forms passing through~$\varphi$} as the collection of
    \begin{itemize}
        \item an affinoid open $\mathcal{U}\subseteq \W$, $\kappa\in \mathcal{U}(\cp)$;
        \item an affinoid $\mathcal{X}\subset \E$, endowed with a finite morphism $\boldsymbol{\varphi}\colon\mathcal{X}\longrightarrow \mathcal{U}$, surjective when restricted to any irreducible component of $\mathcal{X}$;
        \item a $\cp$-point $x_0\in \mathcal{X}(\cp)$ such that $\boldsymbol{\varphi}(x_0)=\kappa$;
        \item a continuous ring homomorphism $\lambda\colon \mathscr{T}\longrightarrow \O_{\mathcal{X}}(\mathcal{X})^{0}$;
    \end{itemize}
    satisfying:
    \begin{itemize}
        \item for all $x\in \mathcal{X}(\cp)$, there exists a form $\varphi_x\in \mathcal{S}^{\dagger}(\mathcal{X})\cap \omega^{-1}(\omega(x))$, such that, for all $h\in \mathscr{T}$, 
        \begin{equation*}
            h(\varphi_x)=\lambda(h)(x)\, \varphi_x;
        \end{equation*}
        \item the form $\varphi$ is such that one can take $\varphi_{x_0}=\varphi$.
    \end{itemize}
    We say that the family is parameterized by $\mathcal{X}$, and that it has \emph{slope} $\alpha$ (resp.\ \emph{finite slope}) if every form $\varphi_x$ in the family has slope $v_p(\lambda(U_p)(x))=\alpha$ (resp.\ $v_p(\lambda(U_p)(x))$ is finite). We call a family of slope $0$ a {\em Hida family}, while we refer to families of finite slope bigger equal than $0$ as {\em Coleman families}.
\end{definition}

\begin{remark}
    \leavevmode
    \begin{enumerate}
        \item By \cite[Proposition 6.2.7]{Bellaiche2021}, if there exists a classical point in a family, then the classical points are dense in the family.
        \item All the eigenvarieties we consider are equidimensional of dimension one (\cite[Proposition 3.7.5]{Bellaiche2021}).
    \end{enumerate}
\end{remark}

\begin{theorem}\label{thm: multiplicity-1 for lifts of Coleman families}
    For any Coleman family $\mathbf{f}$ on $\E_{\GL_2}$, there exist at most a \emph{unique} Coleman family on $\Etilde_{D}^{\epsilon}$, of the same slope, lifting $\mathbf{f}$. Moreover, there exist a family $\boldsymbol{\varphi}$ on $\Etilde_{D}^{\epsilon}$, for each choice of signs $\epsilon$, corresponding to $\mathbf{f}$, if $\mathbf{f}$ passes through a classical point whose corresponding form is supercuspidal at each prime $\ell \mid N_{-}^{\mathrm{sc}}$.
\end{theorem}
\begin{proof}
    Let $(\mathcal{U},\, \kappa,\,\mathcal{X},\,x_0,\,\mathbf{f}\colon\mathcal{X}\rightarrow \mathcal{U})$ be a Coleman family on $\E_{\GL_2}$ and $(\widetilde{JL_p}^{\epsilon})^{-1}(\{x_0\})$ be the preimage of $x_0$. It is either empty or it contains a unique point $x_{0}^{D,\epsilon}\in \Etilde_{D}^{\epsilon}$.
    If $(\widetilde{JL_p}^{\epsilon})^{-1}(\{x_0\})=\emptyset$, up to shrinking $\mathcal{X}$, and hence $\mathcal{U}$, we can assume that $\mathcal{X}\cap \Image(\widetilde{JL_p}^{\epsilon})=\emptyset$ (by Proposition~\ref{prop: JL_p^epsilon is closed immersion}). Therefore, there is no family on $\Etilde_{D}^{\epsilon}$ lifting $(\mathcal{U},\, \kappa,\,\mathcal{X},\,x_0,\,\mathbf{f}\colon\mathcal{X}\rightarrow \mathcal{U})$. Suppose now that $x_{0}^{D,\epsilon}=(\widetilde{JL_p}^{\epsilon})^{-1}(x_0)$. Since $\mathcal{X}$ can be taken to be an open affinoid and closed immersions are finite, $\mathcal{X}_D^{\epsilon}=(\widetilde{JL_p}^{\epsilon})^{-1}(\mathcal{X})$ is an open affinoid (see~\cite[discussion after Definition 4.5.7]{FresnelVanderPut2004}). 
    We hence define $\boldsymbol{\varphi}^{\epsilon}=\mathbf{f}\circ \widetilde{JL_p}^{\epsilon}$. As above, up to shrinking $\mathcal{X}$, and hence $\mathcal{U}$, we can assume that $\mathcal{X}$ is contained in $\Image(\widetilde{JL_p}^{\epsilon})$, therefore $\boldsymbol{\varphi}^{\epsilon}$ is a finite morphism, surjective when restricted to any irreducible component. As $\widetilde{JL_p}^{\epsilon}$ is an eigenvariety morphism, we obtain $\boldsymbol{\varphi}^{\epsilon}(x_{0}^{D,\epsilon})=\kappa$. It remains to lift the ring homomorphism $\lambda$. Consider now the quotient map
    \begin{equation*}
         \Ttilde_{\mathbf{I}}\cong  \frac{\T_{\mathbf{I}}[X_\ell]_{\ell \mid N_{-}^{\mathrm{sc}}}}{(X_\ell^2-\chi_{\A,\ell}(\ell)[\ell]_\mathbf{I})_{\ell \mid N_{-}^{\mathrm{sc}}}}\longrightarrow \T_{\mathbf{I}},
    \end{equation*}
    determined by the choice of signs $\epsilon$, obtained by sending
    \begin{equation*}
        X_\ell\longmapsto \epsilon_\ell\,\sqrt{\chi_{\A,\ell}(\ell)}\,\IL{\ell},\, \textrm{ for each }\, \ell \mid N_{-}^{\mathrm{sc}}.
    \end{equation*}
    This map defines, for any such affinoid open $\mathcal{V}\subseteq \W$, a diagram
    \begin{equation*}
        \begin{tikzcd}
            \widetilde{\mathbb{T}}_{\mathbf{I}} \arrow[dd,two heads]\arrow[rr]      &  & {\Image(\widetilde{\mathbb{T}}_{\mathbf{I}}\rightarrow\End^{\rm cts}_{\O_\mathcal{V}(\mathcal{V}))}(\mathbb{S}_{N_+N_-,\chi}^{D,\epsilon,\dagger}(\mathcal{V})))} \arrow[rrr, "\widetilde{\lambda}^{D}_{\mathcal{V}}", dashed]\arrow[dd] &  &  & {\O_{\mathcal{X}_D^{\epsilon}}(\mathcal{X}_D^{\epsilon})^0}       \\
            &  &                                                                                                                                                                                                                            &  &  &                                                                                             \\
            \mathbb{T}_{\mathbf{I}} \arrow[rr] &  & {\Image(\mathbb{T}_{\mathbf{I}}\rightarrow\End^{\rm cts}_{\O_\mathcal{V}(\mathcal{V})}(\mathbb{S}_{N_+N_-,\chi}^{\dagger}(\mathcal{V}))} \arrow[rrr, "\lambda_{\mathcal{V}}"]                                                  &  &  & {\O_{\mathcal{X}}(\mathcal{X})^0,} \arrow[uu, "(\widetilde{JL_p}^{\epsilon})^*",two heads]
        \end{tikzcd}
    \end{equation*}
    where the first square is a commutative square of continuous algebra homomorphisms, with the second vertical map obtained from the Jacquet--Langlands correspondence (cf.\ Section \ref{Section: Hecke operators}) and the chosen quotient map. We can define $\widetilde{\lambda}^{D}_{\mathcal{V}}$ as the composition of morphisms in the right hand side square, where $\lambda_{\mathcal{V}}$ is the restriction of $\lambda$. Let us remark that the function $\widetilde{\lambda}^{D}_{\mathcal{V}}(\UL{\ell})\in\O_{\mathcal{V}}(\mathcal{V})^0$ obtained by composition is the constant function $\widetilde{\lambda}^{D}(\UL{\ell})(x)=\epsilon_\ell\,\sqrt{\chi_{\A,\ell}(\ell)}\,\IL{\ell}$, for all $x\in\mathcal{X}$ (cf. Lemma~\ref{lem: Action of Uell on non-twist-minimal forms}). The morphism $\widetilde{\lambda}^{D}_{\mathcal{V}}$ is a continuous ring homomorphism.
    It remains to extend $\widetilde{\lambda}^{D}_{\mathcal{V}}$ to $\widetilde{\mathscr{T}}_{\mathbf I}^{\,D,\epsilon}$. We start noticing that the defined morphisms $\widetilde{\lambda}^{D}_{\mathcal{V}}$ glue compatibly, as the $\lambda_{\mathcal{V}}$ do. Moreover, $\widetilde{\mathscr{T}}_{\mathbf I}^{\,D,\epsilon}$ is a commutative $\mathbf{I}$-algebra, hence the translates of $\Image(\widetilde{\mathbb{T}}_{\mathbf{I}}\rightarrow\End^{\rm cts}_{\O_\mathcal{V}(\mathcal{V}))}(\mathbb{S}_{N_+N_-,\chi}^{D,\epsilon,\dagger}(\mathcal{V})))$ define an open covering. The uniqueness of the lift, as well as the invariance of the slope, follow now from the construction of $\boldsymbol{\varphi}^{\epsilon}$.
\end{proof}
\begin{remark}\label{rmk: convergence of the ordinary projector in E^D}
\leavevmode
    \begin{enumerate}
        \item In the above proof, we note that fixing the eigenvalue of each $\UL{\ell}$, shows that
            \begin{equation*}
                \Image(\widetilde{\mathbb{T}}_{\mathbf{I}}\rightarrow\End^{\rm cts}_{\O_\mathcal{V}(\mathcal{V})}(\mathbb{S}_{N_+N_-,\chi}^{D,\epsilon,\dagger}(\mathcal{V})))=\Image(\mathbb{T}_{\mathbf{I}}\rightarrow\End^{\rm cts}_{\O_\mathcal{V}(\mathcal{V})}(\mathbb{S}_{N_+N_-,\chi}^{D,\epsilon,\dagger}(\mathcal{V}))).
            \end{equation*}
        \item We must stress the fact that the sequence $\{U_p^{n!}\}_n$ still converges in $\widetilde{\mathscr{T}}_{\mathbf I}^{\,D,\epsilon}$, as it does so in each closure $\overline{\Image(\widetilde{\mathbb{T}}_{\mathbf{I}}\rightarrow\End^{\rm cts}_{\O_\mathcal{V}(\mathcal{V})}(\mathbb{S}_{N_+N_-,\chi}^{D,\epsilon,\dagger}(\mathcal{V})))}$.
        
    \end{enumerate}
\end{remark}

\subsection{$\Lambda$-adic quaternionic forms and Hida families}\label{Section: Lambda-adic quaternionic forms and Hida families}

Even though Theorem~\ref{thm: multiplicity-1 for lifts of Coleman families} guarantees the existence of Coleman families, we need a more explicit way to describe them. In particular we restrict our attention to Hida families. 

\subsubsection{$\Lambda$-adic quaternionic forms}

We keep the notation introduced in Section \ref{Section: Algebraic quaternionic modular forms} and fix a chain of inclusions of orders
\begin{equation}
    R_{N_+,N_-}\supset R_{pN_+,N_-}\supset R_{p^2N_+,N_-}\supset\cdots\supset R_{p^nN_+,N_-}\supset\cdots\supset R_{p^\infty N_+,N_-},
\end{equation}
for
\begin{equation}
    R_{p^\infty N_+,N_-}=\left\{r\in \widehat{R}_{N_+,N_-}\,\middle|\,\iota_p(r_p)=\mat{a}{b}{0}{d}\,a,d\in\zp^\times,\, b\in\zp\right\}.
\end{equation}
Considering the set
\begin{equation}
    X_\infty=D^\times\backslash \widehat{D}^\times/U_1(R_{p^\infty N_+,N_-}),
\end{equation}
for 
\begin{equation}
    U_1(R_{p^\infty N_+,N_-})=\left\{r\in U_1(R_{N_+,N_-})\,\middle|\,\iota_p(r_p)=\Mat{a}{b}{0}{1}\,a\in\zp^\times,\, b\in\zp\right\},
\end{equation}
we have the natural quotient maps
\begin{equation}
    X_\infty\longrightarrow X_m\longrightarrow X_n, 
\end{equation}
for any $m>n$. We also take $P_n=\left((1+T)^{p^n}-1\right)$, a height one
prime ideal in $\Lambda=\zp\llbracket 1+p\zp\rrbracket\cong\zp\llbracket T\rrbracket$, for $T=\langle 1 + p\rangle_{\Lambda} -1$.
Recall that we define the diamond operators as in \cite[Section 4.4]{Hsieh2021}. We extend the notion of $\Lambda$-adic forms provided in loc.\ cit.\ as follows.
\begin{definition}
    Let $\mathbf{S}^{D}(R_{N_+,N_-},\Lambda)$ be the space of functions $\mathbf{f} \colon X_\infty\longrightarrow \Lambda$, such that, for any $z\in 1+p\zp$,
    \begin{equation*}
        \mathbf{f}(xz)=\mathbf{f}(x)\langle z\rangle^2\langle z\rangle_\Lambda^{-1},
    \end{equation*}
    and, for any $n$ sufficiently large,
    \begin{equation*}
        \mathbf{f}\Mod{P_n} \colon X_\infty\longrightarrow \Lambda/P_n
    \end{equation*}
    factors through $X_n$. We call it the space of {\em $\Lambda$-adic quaternionic modular forms} of level $R_{N_+,N_-}$.
\end{definition}
By construction,
\begin{equation}\label{eq: Lambda-adic forms as profinite limit}
    \mathbf{S}^{D}(R_{N_+,N_-},\Lambda) = \varprojlim\limits_{n} \Hom_{\Lambda}(\zp[X_n],\Lambda/P_n)\otimes_{\Lambda,\iota_2}\Lambda,
\end{equation}
for $\iota_2\colon\Lambda\longrightarrow\Lambda$ the $\zp$-algebra morphism (twisting the action at $p$) defined by
\begin{equation}
    \iota_2\colon T\longmapsto (1+T)^{-2}(1+p)^2-1.
\end{equation}
Therefore, the $\Lambda$-module $\mathbf{S}^{D}(R_{N_+,N_-},\Lambda)$ is compact and endowed with the Hecke action defined by
\begin{equation}
    t \cdot \mathbf{f}(x) = \mathbf{f}(t\cdot x),\, \textrm{ for any }\,t\in\Ttilde_{\mathbf{I}}\, \textrm{ and }\, x\in X_\infty.
\end{equation}

Recall the notation of Section \ref{Section: Lambda-action and operators}. For any Dirichlet character $\chi$ modulo $N_+N_- p$, valued in $\zp$, we define
\begin{equation}
    \mathbf{S}^{D}(R_{N_+,N_-},\chi,\Lambda)=
    \left\{\mathbf{f}\in \mathbf{S}^{D}(R_{N_+,N_-},\Lambda)\}\, \middle|  \, \begin{array}{c}
       \mathbf{f}(xzr) = (\chi_\A)^{-1}(z)\widetilde{\chi}(r) f(x)\,\langle \cyc(z) \rangle^2\,[\langle \cyc(z) \rangle]_\Lambda   \\
         \textrm{for all }\, z\in\A_{f}^\times, \, r\in \widehat{R}_{p^\infty N_+,N_-}^\times 
    \end{array}
     \right\},
\end{equation}

and, for any finite flat extension $\Lambda'/\Lambda$ and any $\Lambda'$-valued Dirichlet character $\chi$, we set
\begin{equation}
    \mathbf{S}^{D}(R_{N_+,N_-},\Lambda')=\mathbf{S}^{D}(R_{N_+,N_-},\Lambda)\otimes_\Lambda\Lambda'\supseteq \mathbf{S}^{D}(R_{N_+,N_-},\chi,\Lambda').
\end{equation}

From equation~\eqref{eq: Lambda-adic forms as profinite limit}, we deduce the compactness of $\mathbf{S}^{D}(R_{N_+,N_-},\Lambda)$, hence the ordinary projector $e^{\rm ord}=\lim\limits_{n \to \infty}U_p^{n!}$ converges in $\End_\Lambda(\mathbf{S}^{D}(R_{N_+,N_-},\Lambda))$, as it converges in each $\End_\Lambda(\zp[X_n])$. We define
\begin{equation}
    e^{\rm ord}\mathbf{S}^{D}(R_{N_+,N_-},\chi,\Lambda')
\end{equation}
as the space of ordinary $\Lambda$-adic quaternionic forms.

\subsubsection{Hida families}

The ordinary projector $e^{\rm ord}$ determines an idempotent in endomorphism ring of quaternionic modular forms and this implies a decomposition into the ordinary and non-ordinary components of the eigenvarieties, which we will denote by the corresponding superscript. We obtain the following diagram of eigenvarieties:
\begin{equation}
    \begin{tikzcd}
	{\widetilde{\mathscr{E}}_{D}^{\,\epsilon,\rm ord}} \arrow[rr, "\widetilde{\iota}^{\,\epsilon}"', hook] \arrow[rrrrd, "\widetilde{JL}_p^{\epsilon}", hook, bend left, shift left=2] &  & \widetilde{\mathscr{E}}_{D}^{\rm ord} \arrow[d, "\widetilde{\pi}"] &  &                           & \widetilde{\mathbb{T}}_{\mathbf{I}}              \\
	&  & \mathscr{E}_{D}^{\rm ord} \arrow[rr, "JL_p", hook]                 &  & \mathscr{E}_{\GL_2}^{\rm ord} & {\ph{,}\mathbb{T}_{\mathbf{I}}.} \arrow[u, hook]
    \end{tikzcd}
\end{equation}
All the results obtained in Section \ref{Section: Eigenvarieties} descend to the ordinary eigenvarieties.

\begin{proposition}\label{prop: identification chenevier families and hsieh families}
    Let $(\mathcal{U},\mathcal{X},\boldsymbol{\varphi},\lambda)$ be a Hida family on $\Etilde^{\,\epsilon}_{D}$ (that is, a family in $\Etilde^{\,\epsilon,\rm ord}_{D}$), and suppose that it contains a classical point. Up to shrinking it, we can identify the family $(\mathcal{U},\mathcal{X},\boldsymbol{\varphi},\lambda)$ with the $\O(\mathcal{U})^0$-adic module 
    \begin{equation*}
        (e^{\rm ord}\mathbf{S}^{D}(R_{N_+,N_-}, \O(\mathcal{U})^0))[\boldsymbol{\varphi}]=\left\{\mathbf{f}\in \mathbf{S}^{D}(R_{N_+,N_-},\Lambda)\otimes_{\Lambda} \O(\mathcal{U})^0 \,\, \middle|\,\, t\cdot \mathbf{f} = \lambda_{\boldsymbol{\varphi}}(t) \,\mathbf{f},\,\textrm{ for } \,t\in \Ttilde_{\mathbf{I}}\,\right\},
    \end{equation*}
    where $\lambda_{\boldsymbol{\varphi}}=\omega_{\GL_2}^*\circ\lambda\colon\widetilde{\mathscr{T}}_{\mathbf I}^{\,D,\epsilon}\longrightarrow \O(\mathcal{U})^0$.
\end{proposition}

\begin{proof}
    The weight map $\omega_{\GL_2}$ is étale at cuspidal ordinary classical points (\cite[Theorem 7.6.4 and Remark 7.6.6]{Bellaiche2021}, where the weight is shifted by $2$) and, by \cite[Proposition 6.2.7]{Chenevier2004}, these points are dense in the family. We deduce that $\O(\mathcal{X})\cong \O(\widetilde{JL}_p^{\epsilon}(\mathcal{X}))$ is a finite flat algebra over $\O(\mathcal{U})$, hence finite flat over $\mathbf{I}$. Up to shrinking $\mathcal{U}$ to a subaffinoid in the admissible covering $\mathscr{C}$, and intersecting its preimage with $\mathcal{X}$, we can identify $\widetilde{JL}_p^{\epsilon}(\mathcal{X})=\mathcal{U}$ by \cite[Lemma 8.1.3]{FresnelVanderPut2004} (restrict $\omega_{\GL_2}$ to a suitable wide open affinoid neighborhood in $\mathcal{X}$, hence shrink $\mathcal{U}$ to a be contained in the wide open neighborhood of the target affinoid). Therefore, $\omega_{\GL_2}$ becomes an isomorphism. Let now $\mathbf{f}$ be an element in $(e^{\rm ord}\mathbf{S}^{D}(R_{N_+,N_-},\Lambda)\otimes_{\Lambda} \O(\mathcal{U})^0)[\boldsymbol{\varphi}]$; by equation~\eqref{eq: Lambda-adic forms as profinite limit}, it is uniquely characterized by the sequence
    \begin{equation*}
        \left\{\mathbf{f}\Mod{P_{n}}\colon X_n\longrightarrow \O(\mathcal{U})^0\otimes_\Lambda\Lambda/P_{n} \right\}_{n\gg 1}.
    \end{equation*}
    For $n$ big enough we consider the specialization map $\lambda_{\boldsymbol{\varphi},n}$ defined as
    \begin{equation*}
        \lambda_{\boldsymbol{\varphi},n}\colon\Ttilde_{\mathbf{I}}\longrightarrow\widetilde{\mathscr{T}}_{\mathbf I}^{\,D,\epsilon}\overset{\lambda_{\boldsymbol{\varphi}}}{\longrightarrow}\O(\mathcal{U})^0\longrightarrow \O(\mathcal{U})^0\otimes_\Lambda\Lambda/P_{n}.
    \end{equation*} 
    Up to $\O(\mathcal{U})^0$-constants, we can then associate each $\mathbf{f}\Mod{P_{n}}$ to
    \begin{equation*}
        x_n=(\lambda_{\boldsymbol{\varphi},n}(t))_{t\in\Ttilde_{\mathbf{I}}}\,\in\, \mathcal{X}(\O(\mathcal{U})^0\otimes_\Lambda\Lambda/P_{n}),
    \end{equation*} where $x_n$ corresponds to the system of eigenvalues associated with the quaternionic modular form $\mathbf{f}\Mod{P_{n}}$. Note that here we are using the fact that, up to further shrinking $\mathcal{U}$, $\O(\mathcal{U})^0$ is a finitely generated compact $\Lambda$-module. We have then proved the sought-for identification.
\end{proof}

Let now $(\mathcal{U},\mathcal{X},\boldsymbol{\varphi},\lambda) $ be a Hida family on $\Etilde^{\,\epsilon}_{D}$ and let $\boldsymbol{\varphi}^{\GL_2}$ be the ordinary family corresponding to it. We consider the space of $\Lambda$-adic classical Hida families and its eigenspace
\begin{equation}
    (e^{\rm ord}\mathbf{S}^{\GL_2}(\Gamma_1(N_+N_-),\Lambda)\otimes_{\Lambda} \O(\mathcal{U})^0)[\boldsymbol{\varphi}^{\GL_2}]=\left\{\mathbf{f}\in \O(\mathcal{U})^0\llbracket T\rrbracket \, \middle|\, t\cdot \mathbf{f} = \lambda_{\boldsymbol{\varphi}^{\GL_2}}(t) \mathbf{f},\,\textrm{ for } \,t\in \T_{\mathbf{I}}\,\right\}.
\end{equation}

\begin{theorem}\label{thm: rank-1 hida families}
    Suppose that $\boldsymbol{\varphi}^{\GL_2}$ contains a classical point. Up to shrinking $\mathcal{U}$, there exists an isomorphism of rank one $\O(\mathcal{U})^0$-modules
    \begin{equation*}
        (e^{\rm ord}\mathbf{S}^{D}(R_{N_+,N_-},\Lambda)\otimes_{\Lambda} \O(\mathcal{U})^0)[\boldsymbol{\varphi}] \cong (e^{\rm ord}\mathbf{S}^{\GL_2}(\Gamma_1(N_+N_-),\Lambda)\otimes_{\Lambda} \O(\mathcal{U})^0)[\boldsymbol{\varphi}^{\GL_2}].
    \end{equation*}
\end{theorem}
\begin{proof}
    Proceeding as in the proof of Proposition~\ref{prop: identification chenevier families and hsieh families}, up to shrink the family, we can assume it to be étale over the weight space and identify it with the corresponding neighborhood $\mathcal{U}$. Moreover, we can assume that all the points are minimal \cite[Lemma 7.4.8]{Bellaiche2021}, therefore the corresponding eigespaces are one-dimensional. The isomorphism in the statement is then obtained keeping track of the constants in $\O(\mathcal{U})^0$. By Strong Multiplicity One on $\GL_2$ and finite flatness of $\O(\mathcal{U})^0$ over $\mathbf{I}$, we deduce that
    \begin{equation*}
        (e^{\rm ord}\mathbf{S}^{D}(R_{N_+,N_-},\Lambda)\otimes_{\Lambda} \O(\mathcal{U})^0)[\boldsymbol{\varphi}]\otimes_{\O(\mathcal{U})^0} \Frac(\O(\mathcal{U})^0)\cong \Frac(\O(\mathcal{U})^0).
    \end{equation*}
    The rank-$1$ statement follows now from the $\GL_2$-case. An alternative approach can be obtained combining \cite[Theorem 7.6.4, Theorem 8.1.5 and Lemma 8.1.1]{Bellaiche2021}, keeping in mind that we are considering cuspidal eigenvarieties and the closed immersions of eigenvarieties of \cite[Theorem 7.2.3]{Bellaiche2021} hold true.
\end{proof}
We are now ready to state the control theorem generalizing \cite[Theorem 4.10]{DallAva2021Hida}. Notice that the specialization morphism can be made explicit as in \cite[Theorem 4.2]{Hsieh2021}.
\begin{corollary}[Hida's Control Theorem]\label{cor: Hida's Control Theorem}
    For any arithmetic weight $(k,\varepsilon_n)\in \W(\cp)\cap \mathcal{U}$,
    \begin{equation*}
        (e^{\rm ord}\mathbf{S}^{D}(R_{N_+,N_-},\O(\mathcal{U})^0))[\boldsymbol{\varphi}]\otimes_{\Lambda}\Lambda/P_{(k,\varepsilon_n)} \cong (e^{\rm ord}S_k^{D}(R_{p^nN_+,N_-}))[\varphi_{(k,\varepsilon_n)}].
    \end{equation*}
\end{corollary}
\begin{proof}
    The above Theorem~\ref{thm: rank-1 hida families} together with Strong Multiplicity One and Proposition~\ref{prop: mult-1 with varpi_d} imply the isomorphism between the rank $1$ modules in the statement.
\end{proof}
\begin{remark}
    In the introduction (Section \ref{section: introduction}) we already pointed out that the condition for having weight one classical specializations in a family $\boldsymbol{\varphi}^{\GL_2}$ are rather strict. In the setting of this section, this condition can be read easily from the level $N_-$: a family will not contain classical weight one specializations unless $N_{-}^{\mathrm{sp}}=1$, i.e.\ $N_-=N_{-}^{\mathrm{sc}}$. Therefore, throughout the rest of the paper, we will work under this assumption.
\end{remark}

\section{Balanced triple product $p$-adic $L$-function}\label{Section: Balanced triple product $p$-adic $L$-function}

In this section, we prove our main theorem about the existence of balanced $p$-adic $L$-functions and their interpolation property.

\subsection{Definition of the $p$-adic $L$-function}\label{section: Definition of the $p$-adic $L$-function}

 Let $\mathbf F = (\mathbf f, \mathbf g, \mathbf h)$ be the triple product of primitive Hida families of tame conductors $(N_1, N_2, N_3) \in \N^3$ and characters $(\chi_1, \chi_2, \chi_3)$. Let $N = \lcm(N_1, N_2, N_3)$. For any classical weight $(k_1, k_2, k_3)$, we write $(f_{k_1}, g_{k_2}, h_{k_3}) = \mathbf F(k_1, k_2, k_3)$; moreover, we let:
$$\Sigma^- = \{ \ell \text{ finite} \ | \ \epsilon_\ell ( f_{k_1} \times g_{k_2} \times h_{k_3}) = -1 \},$$
which is independent of the choice of $(k_1, k_2, k_3)$ by the rigidity of automorphic types. 

We make the following assumptions: 
\begin{itemize}
    \item $|\Sigma^-|$ is odd; then there exists a definite quaternion algebra $D$ over $\Q$ ramified exactly at the places in $\Sigma^-$ (and infinity),
    \item if $\ell \in \Sigma^-$, then $v_\ell(N) \leq 2$ (this is an improvement on Hsieh's~\cite{Hsieh2021} assumption that $v_\ell(N) = 1$).
\end{itemize}

We now choose test vectors on the quaternionic group $D^\times(\A)$. First, we write
\begin{align}
    \Sigma_{\ast}^{-, \mathrm{sc}} & = \{ \ell \in \Sigma^- \ | \ \ast \text{ is supercuspidal at }\ell\} & \text{for } \ast \in \{f, g, h \}, \\
    \Sigma^{-, \mathrm{sc}} & = \Sigma^{-, \mathrm{sc}}_f \cup \Sigma^{-, \mathrm{sc}}_g \cup \Sigma^{-, \mathrm{sc}}_h \subseteq \Sigma^{-}, \\
    \Sigma^{-, n} & = \{\ell \in \Sigma^- \ | \ n\text{ of $\pi_{f, \ell}$, $\pi_{g, \ell}$, $\pi_{h, \ell}$ are supercuspidal} \} & \text{for $n \in \{0,1,2,3\}$}.\label{eqn:Sigma-n}
\end{align}
Once again, these sets only depend on $\mathbf f$, $\mathbf g$, $\mathbf h$.

Recall from Proposition~\ref{prop:epsilon=-1} that $\Sigma^- = \Sigma^{-,0} \sqcup \Sigma^{-, 2} \sqcup \Sigma^{-,3}$. The computations of the local integrals at $\ell \in \Sigma^-$ are in Propositions~\ref{prop:local_int_spspsp}, \ref{prop:local_int_scscsp}, ~\ref{prop:local_integral_scscsc}, respectively.

Next, we choose the signs which determine the Hida families on $D^\times$:
\begin{align*}
\epsilon_\ast & \in \{\pm1\}^{\Sigma^{-, \mathrm{sc}}_{\ast}} & \text{for }\ast \in \{f,g,h\}, \\
\epsilon & = (\epsilon_f, \epsilon_g, \epsilon_h).
\end{align*}
To shorten the notation from Section~\ref{Section: Lambda-adic quaternionic forms and Hida families} we will write $eS^{D, \epsilon}(N, \psi, \mathcal U)$ for the space of $\Lambda$-adic quaternionic forms $e^{\mathrm{ord}}\mathbf{S}^{D}(R_{N_+,N_-},\psi_{\A}^{-1},\mathcal{O}(\mathcal{U})^{0})$ where the extra Hecke operators act accordingly to the choice of signs~$\epsilon$.
Then by Theorem~\ref{thm: multiplicity-1 for lifts of Coleman families} and Proposition~\ref{prop: identification chenevier families and hsieh families}, for any classical point of the weight space there exist an open admissible affinoid neighborhood $\mathcal U = \mathcal U_1 \times \mathcal U_2 \times \mathcal U_3$ of this point, and elements:
    $$(\mathbf f^{D, \epsilon_f}, \mathbf g^{D, \epsilon_g}, \mathbf h^{D, \epsilon_h}) \in eS^{D, \epsilon_f}(N_1, \psi_1, \mathcal U_1)[\mathbf f] \times eS^{D, \epsilon_g}(N_2, \psi_2, \mathcal U_2)[\mathbf g] \times eS^{D, \epsilon_h}(N_2, \psi_2, \mathcal U_3)[\mathbf h].$$
Note that these choices are only well-defined up to elements in $\O(\mathcal U)^\times = \O(\mathcal U_1)^\times \hat\otimes \O(\mathcal U_2)^\times \hat\otimes \O(\mathcal U_3)^\times$.
    
Finally, we bring these forms to the common level $N$, following~\cite[Definition 4.8]{Hsieh2021}.

\begin{definition}\label{def: addjustment levels and test families}
    \leavevmode
    \begin{enumerate}
        \item Define the adjustments of levels $\mathbf d_f$, $\mathbf d_g$, $\mathbf d_h$ as in \cite[Section 3.4]{Hsieh2021}. At $\ell \in \Sigma^{-, \mathrm{sc}}$, we make no additional adjustment.
        \item  Consider the sets $\Sigma_{\ast, 0}^{\rm IIb}$ as in \cite[Section 3.4]{Hsieh2021} and define:
    $$(\mathbf f^{D\star, \epsilon_f}, \mathbf g^{D\star, \epsilon_g}, \mathbf h^{D\star, \epsilon_h}) \in eS^{D, \epsilon_f}(N_1, \psi_1, \mathcal U_1)[\mathbf f] \times eS^{D, \epsilon_g}(N_2, \psi_2, \mathcal U_2)[\mathbf g] \times eS^{D, \epsilon_h}(N_3, \psi_3, \mathcal U_2)[\mathbf h]$$
    by
    \begin{align*}
        \mathbf f^{D\star, \epsilon_f} & = \sum_{I \subseteq \Sigma_{f, 0}^{\rm IIb}} (-1)^{|I|} \beta_I(f)^{-1} V_{\mathbf d_f/n_f}  \mathbf f^{D, \epsilon_f}, \\
        \mathbf g^{D\star, \epsilon_f} & = \sum_{I \subseteq \Sigma_{g, 0}^{\rm IIb}} (-1)^{|I|} \beta_I(g)^{-1} V_{\mathbf d_g/n_g}  \mathbf g^{D, \epsilon_g}, \\
        \mathbf h^{D\star, \epsilon_h} & = \sum_{I \subseteq \Sigma_{h, 0}^{\rm IIb}} (-1)^{|I|} \beta_I(h)^{-1} V_{\mathbf d_h/n_h}  \mathbf f^{D, \epsilon_h}.
    \end{align*}
    \end{enumerate}

\end{definition}

We now define an unnormalized version of the triple product $p$-adic $L$-function.

\begin{definition}
    \leavevmode
    Let $\mathcal R = \O(\mathcal U) = \O(\mathcal U_1) \hat \otimes_{\O} \O(\mathcal U_2) \hat \otimes_{\O} \O(\mathcal U_3)$.
    \begin{enumerate}
        \item Define the {\em triple product} $\mathbf F^{D\star, \epsilon} \colon (D^\times \backslash \widehat D^\times)^3 \to \mathcal R$ by $\mathbf F^{D\star, \epsilon} = \mathbf f^{D\star, \epsilon_f} \boxtimes \mathbf g^{D\star, \epsilon_g} \boxtimes \mathbf h^{D\star, \epsilon_h}$.
        \item The associated {\em theta element} is:
    $$\Theta_{\mathbf F^{D\star, \epsilon}} = (\mathbf F^{D\star, \epsilon})^\ast(\Delta_\infty^\dagger) \in \mathcal R,$$
    where $\Delta_\infty^\dagger$ is the {\em regularized diagonal cycle} from \cite[Definition\ 4.6]{Hsieh2021}.
    \end{enumerate}
\end{definition}

Since $\mathbf F^{D\star, \epsilon}$ is only well-defined up to scalars, so is $\Theta_{\mathbf F^{D \star, \epsilon}}$. To define the genuine $p$-adic $L$-function, we will divide $\Theta_{\mathbf F^{D \star, \epsilon}}$ by the Petersson norm of $\mathbf F^{D \star, \epsilon}$.

Recall the quaternionic Petersson product introduced in Section \ref{Section: The quaternionic Petersson product}. Considering such pairing, for $\mathbf f$ and $\mathbf f'\in \mathbf S^D(N_+N_-, \chi, \mathbf I)$ we define
\begin{equation}
    \mathbf B_{N_+,N_-,\alpha}(\mathbf f, \mathbf f') = \sum\limits_{[x]\in D^\times\backslash\widehat{D}^\times/\widehat{R}_{N_+p^\alpha,N_-}^\times}\frac{\chi_\A(\nu_{\A_f}(x))\,\langle\Np(x)\rangle\, [\langle\Np(x)\rangle]_{\mathbf I}}{\#\Gamma_{N_+p^\alpha,N_-}(x)}\,\mathbf{f}\left(x\, \AL{N_+p^\alpha,N_-}\right) \mathbf{f}'\left(x\right).
\end{equation}
As in~\cite[Definition 4.3]{Hsieh2021}, we obtain a Hecke-equivariant $\mathbf I$-bilinear pairing
\begin{equation}
    \mathbf B_N \colon e \mathbf S^D(N, \chi, \mathbf I) \times e \mathbf S^D(N, \chi, \mathbf I) \to \mathbf I,
\end{equation}
\begin{equation}
    \mathbf B_N(\mathbf f, \mathbf f') = \varprojlim_\alpha \mathbf B_{N, \alpha}(\mathbf f, \mathbf f') \in \varprojlim_\alpha \mathbf I/P_\alpha = \mathbf I
\end{equation}
such that for every arithmetic point $\kappa = (k, \epsilon)$ and integer $\alpha \geq \max\{1,  v_p(\cond(\epsilon))\}$, we have (cf.~\cite[Proof of Lemma 4.4]{Hsieh2021}) 
\begin{align}
    \mathbf B_N(\mathbf f, \mathbf f')(\kappa) = (-1)^k \langle \mathbf U_p^{-\alpha} \mathbf f_{\kappa}^{}, \mathbf f_{\kappa}' \rangle_{N_+p^\alpha,N_-}.
\end{align}

\begin{definition}
    For $\mathbf f^D \in e\mathbf S^D(N, \chi, \O(\mathcal U))$, the {\em Petersson inner product} of $\mathbf f^D$ is
    $$\eta_{\mathbf f^D} = \mathbf B_N(\mathbf f^D, \mathbf f^D) \in \O(\mathcal U)$$
    and the {\em Petersson norm} of $\mathbf f^D$ is:
    $$\| \mathbf f^{D} \| =  \eta_{\mathbf f^D}^{1/2} \in \Frac \O(\mathcal U).$$
    Similarly, for a triple product $\mathbf F^D =  \mathbf f^D \boxtimes \mathbf g^D \boxtimes \mathbf h^D$, its {\em Petersson norm} is:
    $$ \| \mathbf F^{D} \| = \| \mathbf f^D \| \cdot \| \mathbf g^D \| \cdot \| \mathbf h^D \| \in \Frac \mathcal O(\mathcal U).$$
\end{definition}

Before defining the $p$-adic $L$-function we recall two more technical points from~\cite{Hsieh2021}, related to the choice of test vectors at primes dividing $N^+$:
\begin{itemize}
    \item There is a twist $\mathbf F' = (\mathbf f \otimes \chi_1, \mathbf g \otimes \chi_2, \mathbf h \otimes \chi_3)$ by Dirichlet characters $\chi_1, \chi_2, \chi_3$ modulo $M$ with $M^2 | N^+$ such that $\chi_1 \chi_2 \chi_3 = 1$ and $\mathbf F'$ satisfies Hypothesis 6.1 of loc. cit. (see also Remark~6.2). 
    \item There is a {\em fudge factor} $\mathfrak f_{\mathbf F'} = \prod\limits_{q | N^+}  f_{\mathbf F', q}  \in \mathcal R^\times$ defined in Proposition 6.12 of loc.\ cit., and, enlarging $\mathcal O$ if necessary, we have that $\sqrt{\mathfrak f_{\mathbf F'}} \in \mathcal R^\times$.
\end{itemize}

By~\cite[Lemma 6.11]{Hsieh2021}, there exists $\epsilon^{\Sigma^-}(\mathbf F) \in \mathcal R^\times$ such that 
$$\epsilon^{\Sigma^-}(\mathbf F)(\kappa) = \epsilon^{\Sigma^-}(f_{\kappa_1}) \epsilon^{\Sigma^-}(g_{\kappa_2}) \epsilon^{\Sigma^-}(h_{\kappa_3}) $$ is the product of the away-from-$\Sigma^-$ parts of the root numbers. 

\begin{definition}\label{def:p-adic_L-fun}
    The {\em (genuine) square root balanced triple product $p$-adic $L$-function} associated with $\mathbf F$ and~$\epsilon$ is:
        $$\mathcal L_{\mathbf F, \epsilon}^{\rm bal}(\kappa) = \frac{\Theta_{\mathbf F'^{D\star, \epsilon}}}{\| \mathbf F^{D \star, \epsilon} \|} \cdot 2^{- \frac{|\Sigma^-| + 1 -k_1-k_2-k_3}{2}} \cdot (N^{-})^{-1/2} \cdot \epsilon^{\Sigma^-}(\mathbf F)^{-1/2} \cdot \sqrt{\mathfrak f_{\mathbf F'}}^{-1} \prod_{\ell \in \Sigma^{-, \mathrm{sc}}} \frac{\ell^{k_1 + k_2 + k_3 - 13/2}}{\sqrt{\zeta_\ell(2)}}\in \Frac \mathcal R.$$
\end{definition}

\subsection{The interpolation property}

Consider the subset of {\em arithmetic points} in $\mathcal U$:
\begin{equation}
    \mathcal U^{\rm arith} = \left\{\kappa = (k_1, k_2, k_3, \chi_1, \chi_2, \chi_3) \in \mathcal{U}_1^{\rm arith} \times \mathcal{U}_2^{\rm arith} \times \mathcal{U}_3^{\rm arith}\ \middle|\ k_1+k_2+k_3 \equiv 0 \pmod 2 \right\},
\end{equation}
and let $\mathcal{U}^{\rm bal}$ be the subset of {\em balanced arithmetic points}:
\begin{equation}
    \mathcal{U}^{\rm bal}=\left\{\kappa = (k_1, k_2, k_3, \chi_1, \chi_2, \chi_3) \in \mathcal{U}^{\rm arith}\ \middle|\ k_1+k_2+k_3 > 2k_i\  \text{for all $i=1,2,3$}\right\}.
\end{equation}

\begin{theorem}\label{thm:interpolation}
    For an arithmetic point $\kappa = (k_1, k_2, k_3, \chi_1, \chi_2, \chi_3) \in \mathcal U^{\bal}$ in the balanced range we have:
    $$ (\mathcal L_{\mathbf F, \epsilon}^{\bal}(\kappa))^2 = \frac{ \Gamma_{\mathbf V_\kappa^\dagger}(0) \cdot L(\mathbf V_\kappa^\dagger, 0)}{(\sqrt{-1})^{k_1 + k_2 + k_3-1} \langle \mathbf F_\kappa, \mathbf F_\kappa \rangle} \cdot \mathcal E_p^{\bal}(\mathbf V_\kappa) \cdot \prod_{q \in \Sigma_{\mathrm{exc}}} (1 - q^{-1})^2 \cdot \prod_{\ell \in \Sigma^{-, 2}} \frac{1 + \epsilon_{\ell, 1} \epsilon_{\ell, 2}\sqrt{\omega_3(\ell)}}{2} \prod_{\ell \in \Sigma^{-, 3}} \frac{1 + \epsilon_{\ell, 1} \epsilon_{\ell, 2} \epsilon_{\ell, 3}}{4},$$
    where $\mathcal E_p(V_\kappa)$ is an Euler factor defined by:
    $$\mathcal E_p^{\bal}(\mathbf V_\kappa) = \frac{\mathcal E_p(\mathrm{Fil}^+_{\bal} \mathbf V_\kappa)}{\mathcal E_p(\mathbf F_\kappa, \Ad) }$$
    and $\mathcal E_p(\mathbf F_\kappa, \Ad)$ is a modified adjoint Euler factor~\eqref{eqn:modified_adjoint_Euler_factor}.
    
    In particular, $\mathcal L_{\mathbf F, \epsilon}^{\bal}\neq 0$ only if $\epsilon$ satisfies:
    \begin{enumerate}
        \item if $\ell \in \Sigma^{-, 2}$, then $\epsilon_\ell = (\epsilon_{\ell, 1}, \epsilon_{\ell, 2})$: if the special representation is a twist of Steinberg by $\omega_3$ such that $\omega_3(\ell) = 1$, then $\epsilon_{\ell, 1} \cdot \epsilon_{\ell, 2} = 1$,
        \item if $\ell \in \Sigma^{-, 3}$, then $\epsilon_{\ell, 1} \cdot \epsilon_{\ell, 2} \cdot \epsilon_{\ell, 3} = 1$.
    \end{enumerate}
\end{theorem}

\begin{remark}\label{rmk:difference_to_Hsieh}
    In this extended remark, we compare our $p$-adic $L$-function to the one defined by Hsieh~\cite[Theorem B]{Hsieh2021}. The main difference is that his $p$-adic $L$-function is defined without normalizing by the Petersson norm $\| \mathbf F^{D \ast, \epsilon} \|$ on the quaternion algebra. As a result, note that Gross periods (\cite[Definition 4.12]{Hsieh2021}) do not feature in the interpolation formula for our $p$-adic $L$-function.  
    
    We instead use the above definition for two reasons:
    \begin{enumerate}
        \item It seems that this is the correct $p$-adic $L$-function to state an Elliptic Stark Conjecture~\ref{conj:ES} for. For example, we will soon see in Theorem~\ref{thm:CM} that it is this $p$-adic $L$-function that admits a natural factorization in the CM case.
        \item We did not prove the analogue of Theorem 4.5 in loc. cit.: that the $\mathbf I$-modules $e S^{D, \pm}(N, \mathbf I)[\lambda_{\mathbf f}^D]$ are free of rank one over the whole algebra $\mathbf I$. Instead, we satisfy ourselves with the local statement in Theorems~\ref{thm: multiplicity-1 for lifts of Coleman families} and~\ref{thm: rank-1 hida families}, because it is enough for our arithmetic applications. However, this means that our choices of vectors $\mathbf F^{D\star, \epsilon}$ are only well-defined up to scalars, and hence only the quotient by the Petersson norm is well-defined.
    \end{enumerate}
    The terminology {\em genuine} $p$-adic $L$-function is inspired by the discussion after \cite[Theorem A]{Hsieh2021}. The advantage is that it is independent of choices, but the disadvantage is that it is only an element of $\Frac \mathcal R$ and not $\mathcal R$. 

    The denominators of $\mathcal L_{\mathbf F, \epsilon}^{\bal}$ should be captured by the congruence module of $\mathbf F^{D\star, \epsilon}$, which in turn should be related to the congruence module for $\mathbf F$ (see \cite[Remark 7.8]{Hsieh2021}). Indeed, if one could choose vectors $\mathbf f^{D, \pm} \in eS^{D, \pm}(N, \mathbf I)[\lambda_{\mathbf f}^D]$ as in \cite[Theorem 4.5]{Hsieh2021}, then one could also define a $p$-adic $L$-function in $\mathbf I \hat \otimes \mathbf I \hat \otimes \mathbf I$, generalizing the one constructed in loc.\ cit. It would then differ from our genuine $p$-adic $L$-function by $\sqrt{\eta_{\mathbf F'^{D\star, \epsilon}}}$ and one could presumably show that $\eta_{\mathbf F'^{D\star, \epsilon}}$ is a generator for the congruence module as in \cite[Section 7.2]{Hsieh2021}. We decided to defer these questions to future work.
\end{remark}

\begin{remark}
    Note that the Euler factor at $p$: 
    $$\mathcal E_p(\mathbf V_\kappa) = \frac{\mathcal E_p(\mathrm{Fil}^+_{\bal} \mathbf V_\kappa)}{\mathcal E_p(\mathbf F_\kappa, \Ad) }$$
    is analogous to the Euler factor for the unbalanced $p$-adic $L$-function constructed by Darmon--Rotger~\cite[Theorem 1.3]{Darmon_Rotger}:
    $$\frac{\mathcal E(f,g,h)}{\mathcal E_0(f) \mathcal E_1(f)} = \frac{(1 - \beta_f \alpha_g \alpha_h p^{-c})(1 - \beta_f \alpha_g \beta_h p^{-c})(1 - \beta_f \beta_g \alpha_hp^{-c}) (1 - \beta_f \beta_g \beta_h p^{-c})}{(1 - \beta_f^2 \chi_f^{-1}(p) p^{1-k})(1 - \beta_f^2 \chi_f^{-1}(p) p^{-k})}.$$
\end{remark}

The proof of Theorem~\ref{thm:interpolation} will occupy the rest of this section and amounts to generalizing the results of~\cite[Section 4]{Hsieh2021}. We split it into two parts:
\begin{itemize}
    \item an intermediate interpolation property obtained from Ichino's formula~\cite{Ichino2008}, with factors coming from certain normalized local zeta integrals,
    \item the evaluation of the local zeta integrals.
\end{itemize}

\subsection{An intermediate interpolation property}

We start by proving the analogue of \cite[Corollary 4.13]{Hsieh2021}, which gives an interpolation property up to certain local factors.

\begin{proposition}\label{prop:interpolation_theta}
    Let:
    \begin{align}\label{eqn:script_I}
        \mathscr I_{\Pi_{\kappa, q}}^\star & = \begin{cases} \displaystyle
            I_q(\phi_q^\star \otimes\widetilde \phi_q^\star) \cdot B_{\Pi_q} \cdot \frac{\zeta_q(1)^2}{|N|_q^2 \zeta_q(2)^2}  \cdot \omega_{F, q}^{-1}(\mathbf d_F) |\mathbf d_F^{\underline{k} - 2}|_q & \text{for }q|N^+, \\[0.3cm]
            \displaystyle I_q(\phi_q \otimes\widetilde \phi_q) \cdot B_{\Pi_q} \cdot q \frac{\zeta_q(1)^5}{\zeta_q(2)^3} \cdot \omega_{F, q}^{-1}(\mathbf d_F) |\mathbf d_F^{\underline{k}-2}|_q & \text{for }q \in \Sigma^{-, 2}, \\[0.3cm] 
            \displaystyle I_q(\phi_q \otimes\widetilde \phi_q) \cdot B_{\Pi_q} \cdot q \frac{\zeta_q(1)^4}{\zeta_q(2)^2} \cdot \omega_{F, q}^{-1}(\mathbf d_F) |\mathbf d_F^{\underline{k}-2}|_q & \text{for }q \in \Sigma^{-, 3}
        \end{cases}
    \end{align}
    be the normalized local zeta integral at $q|(N/d^-)$, with $I_q(\phi_q^\star \otimes\widetilde \phi_q^\star)$ defined in equation~\eqref{eqn:Ichino_local_int}, and let:
    \begin{align}
        \mathscr I_{\Pi_{\kappa, q}}^{\bal} & = I_p^{\mathrm{ord}}(\phi_p, \breve{\mathbf t}_n) \cdot B_{\Pi_p^{\rm ord}}^{[n]} \frac{\omega_{F,p}^{1/2}(-p^{2n}) |p|_p^{-n(k_1+k_2+k_3)}}{\alpha_p(F)^{2n} \zeta_p(2)^2}.
    \end{align}
    be the normalized $p$-adic zeta integral defined in \cite[(4.21)]{Hsieh2021} (see loc.\ cit.\ for details).
    
    Then for $\kappa = (\kappa_1, \kappa_2, \kappa_3) \in \mathcal{U}^{\bal}$ in the balanced range, we have the interpolation formula:
    $$\frac{\Theta_{\mathbf F^{D \star}}(\kappa)^2}{\langle F^D, F^D \rangle} = 2^{\# \Sigma^{-,0} +1 -  k_1 - k_2 - k_3} \cdot N^- \cdot \frac{L(1/2, \Pi_\kappa) \cdot \epsilon^{\Sigma^-}(f_{k_1}) \epsilon^{\Sigma^-}(g_{k_2}) \epsilon^{\Sigma^-}(h_{k_3})}{\langle f_{k_1}, f_{k_1} \rangle \langle g_{k_2}, g_{k_2} \rangle \langle h_{k_3}, h_{k_3} \rangle} \cdot \frac{\mathscr I_{\Pi_{\kappa, p}}^{\bal}}{\mathcal E_p(\mathbf F_\kappa, \Ad)} \cdot \prod_{q | (N/d^-)}\mathscr I_{\Pi_{\kappa, q}}^\star.$$
    where $f_{k_1}$, $g_{k_2}$, $h_{k_3}$ are the newforms associated with the specializations of $\mathbf F$ at $\kappa$, 
    $$\epsilon^{\Sigma^-}(F) = \prod_{\ell | N/d^-} \epsilon(1/2, \pi_{F, \ell}) |N_F|_\ell^{(2 - k_i)/2} \in \widehat \Z_{(p)}^\times$$
    is the away-from-$\Sigma^-$ part of the root number of $F \in \{f_{k_1}, g_{k_2}, h_{k_3}\}$, and
    \begin{equation}\label{eqn:modified_adjoint_Euler_factor}
        \mathcal E_p(\mathbf F_\kappa, \Ad) = \mathcal E_p(f_{k_1} \Ad) \mathcal E_p(g_{k_2}, \Ad) \mathcal E_p(h_{k_3}, \Ad).
    \end{equation}
    is the modified Euler factor in \cite[(3.10)]{Hsieh2021}.
\end{proposition}
\begin{proof}
    Let $\kappa = (\kappa_1, \kappa_2, \kappa_3)$ for $\kappa_i = (k_i, \chi_i)$. We write
    \begin{equation}\label{eqn:Ichino_local_int}
        I_\ell(\phi_\ell \otimes\widetilde{\phi_\ell}) = \frac{L(1, \Pi_v, \Ad)}{\zeta_\ell(2)^2  L(1/2, \Pi_v)}  \cdot \frac{I_{\ell}'(\phi_\ell \otimes\widetilde{\phi_\ell})}{\langle \phi_\ell,\widetilde{\phi_\ell} \rangle}
    \end{equation}
    for the normalized local integrals. Then Ichino's Formula~\cite[Theorem 1.1, Remark 1.3]{Ichino2008} gives the following expression (by combining \cite[p. 473, proof of Proposition 4.10]{Hsieh2021} with Lemma~\ref{lemma:Hsieh4.5}):
    \begin{align}
        \frac{I(\varrho(\breve{\mathbf t}_n) \phi_F^{D\star})^2}{\langle F^D, F^D \rangle} & = \frac{\mathrm{vol}(\widehat \O_D^\times)}{8} \cdot \frac{\zeta_\Q(2)^2 L(1/2, \Pi)}{L(1, \Pi, \Ad)} \label{eqn:Ichinos_formula}  \\
        & \hspace{1cm} \cdot I_p^{\mathrm{ord}}(\phi_p \otimes \widetilde{\phi_p}, \breve{\mathbf t}_n)
        \prod_{q \in \Sigma^{-} \cup \{\infty \}} I_q(\phi_q \otimes \widetilde \phi_q) 
        \prod_{q \not\in p \cup \Sigma^-} I_q(\phi_q^\star \otimes \widetilde \phi_q^\star) \nonumber \\
        & \hspace{1cm} \cdot \omega_F^{-1/2}(\hat N_1^+) \omega_{F, p}^{-1}(p^n) \alpha_p(F)^{2n} \prod_{i=1}^3 \frac{\mathrm{vol}(\widehat R^\times_{N_i p^{2n}})}{(N_i^+ N_i^{-, {\rm sc}} p^{2n})^{(k_i - 2)/2}(k_i - 1)}, \nonumber
    \end{align}
    where 
    $$\langle F^D, F^D \rangle = \langle U_p^{-n} f^D, f^D \rangle_{N_1 p^n} \langle U_p^{-n} g^D, g^D \rangle_{N_2 p^n} \langle U_p^{-n} h^D, h^D \rangle_{N_3 p^n}.$$
    Next, we use the volume formula from \cite[Theorem 3.4]{Pizer80p2}:
    \begin{align*}
        \mathrm{vol}(\widehat R_N^\times) & = \frac{48}{N} \prod_{q || N^-} \zeta_q(1) \prod_{q^2 || N^-} \zeta_q(2) \prod_{q | N^+} \frac{\zeta_q(2)}{\zeta_q(1)} \\
        & = \frac{48}{[\SL_2(\Z) : \Gamma_0(N)]} \prod_{q || N^-} \frac{1 + q^{-1}}{1 - q^{-1}} \prod_{q^2 || N^-} \frac{1 + q^{-1}}{1 - q^{-2}} \\
        & = \frac{48}{[\SL_2(\Z) : \Gamma_0(N)]} \prod_{q || N^-} \frac{\zeta_q(1)^2}{\zeta_q(2)} \prod_{q^2 || N^-} \zeta_q(1). 
    \end{align*}
    Using this, the final factor in equation~\eqref{eqn:Ichinos_formula} becomes:
    \begin{equation}\label{eqn:volume_term}
        \prod_{i=1}^{3} \frac{48}{(N_i^+ N_i^{-, {\rm sc}} p^{2n})^{(k_i-2)/2} [\SL_2(\Z) : \Gamma_0(N_i p^{2n})] (k_i - 1)} \prod_{q || N_i^-} \frac{\zeta_q(1)^2}{\zeta_q(2)} \prod_{q^2 || N_i^-} \zeta_q(1).
    \end{equation}
    As we will need it shortly, we also record that:
    \begin{equation}\label{eqn:volume_by_volume_squared}
        \frac{\vol(\widehat \O_D^\times)}{\vol(\widehat R_N^\times)^2} = \frac{1}{48} \prod_{q || N^-} q \frac{1}{\zeta_q(1)} \prod_{q^2 || N^-} q^3 \frac{\zeta_q(1)}{\zeta_q(2)^2} \prod_{q | N^+} \frac{\zeta_q(1)^2}{|N|_q^2\zeta_q(2)^2}.
    \end{equation}

    Next, we recall the relationship between the Petersson norm of a newform $f$ and the adjoint $L$-value~\cite[(2.18)]{Hsieh2021}:
    \begin{equation}\label{eqn:Pet_and_adjoint}
        \langle f, f \rangle_{\Gamma_0(N)} = \frac{[\SL_2(\Z) : \Gamma_0(N)]}{2^k w(f)} \cdot L(1, \pi, \Ad) \cdot \prod_{q | N} B_{\pi_q},
    \end{equation}
    where $B_{\pi_q}$ are given by~\cite[(2.18)]{Hsieh2021}. In particular:
    \begin{equation}\label{eqn:Pet_and_adjoint_triple}
        L(1, \Pi, \Ad) = \prod_{i=1}^3 \langle f_i^\circ, f_i^\circ \rangle  \frac{2^{k_i} w(f_i^\circ)}{[\SL_2(\Z) : \Gamma_0(N_i p^{c_i})]  \prod\limits_{q|N} B_{(\pi_i)_q}},
    \end{equation}
    where $f_i^\circ$ is the newform associated with $f_i$ and its level is $Np^{c_i}$. We will write $B_{\Pi_{F, q}} = B_{\pi_{1, q}} B_{\pi_{2, q}} B_{\pi_{3, q}}$, as in loc.\ cit..

    Finally, Hsieh~\cite[Proposition 4.9]{Hsieh2021} gives a relationship between $I(\varrho(\breve{\mathbf t}_n) \phi_F^{D\star})$ and $\Theta_{\mathbf F^{D\star}}(\underline Q)$:
    \begin{equation}\label{eqn:Theta_and_I}
        \Theta_{\mathbf F^{D\star}}(\kappa) = \frac{1}{\mathrm{vol}(\widehat R_N^\times)} I(\varrho(\breve{\mathbf t}_n) \phi_F^{D\star}) \cdot \frac{\omega_{F, p}^{1/2}(p^n) |p|^{-\frac{k_1 + k_2 + k_3}{2}}}{\alpha_p(F)^n \zeta_p(2)} \frac{1}{\omega_{F}^{1/2}(\hat{\mathbf{d}_f}) \mathbf d_{F}^{(\underline k - 2)/2}}.
    \end{equation}

    Next, we put all of these facts together to get:
    \begin{align*}
        \frac{\Theta_{\mathbf F^{D \star}}(\kappa)^2}{\langle F^D, F^D \rangle} &  = \frac{1}{\mathrm{vol}(\widehat R_N^\times)^2} \frac{I(\varrho(\breve{\mathbf t}_n) \phi_F^{D\star})^2}{\langle F^D, F^D \rangle} \cdot \frac{\omega_{F, p}(p^n) |p^n|^{-(k_1 + k_2 + k_3)}}{\alpha_p(F)^{2n} \zeta_p(2)^2} \frac{1}{\omega_{F}(\hat{\mathbf{d}_f}) \mathbf d_{F}^{\underline k - 2}} & \text{\eqref{eqn:Theta_and_I}} \\
        & = \frac{\vol(\hat\O_D^\times)}{8\mathrm{vol}(\widehat R_N^\times)^2}
        \cdot \frac{\zeta_\Q(2)^2 L(1/2, \Pi)}{L(1, \Pi, \Ad)} & \text{\eqref{eqn:Ichinos_formula}} \\
        & \hspace{1cm} \cdot I_p^{\mathrm{ord}}(\phi_p \otimes \widetilde{\phi_p}, \breve{\mathbf t}_n)
        \prod_{q \in \Sigma^{-} \cup \{\infty \}} I_q(\phi_q \otimes \widetilde \phi_q) 
        \prod_{q \not\in p \cup \Sigma^-} I_q(\phi_q^\star \otimes \widetilde \phi_q^\star)  \\
        & \hspace{1cm} \cdot \omega_F^{-1/2}(\hat N_1^+) \omega_{F, p}^{-1}(p^n) \alpha_p(F)^{2n} \prod_{i=1}^3 \frac{\mathrm{vol}(\widehat R^\times_{N_i p^{2n}})}{(N_i^+ N_i^{-, {\rm sc}} p^{2n})^{k_i - 2/2}(k_i - 1)} \\
        & \hspace{1cm} \cdot \frac{\omega_{F, p}(p^n) |p^n|^{-(k_1 + k_2 + k_3)}}{\alpha_p(F)^{2n} \zeta_p(2)^2} \frac{1}{\omega_{F}(\hat{\mathbf{d}_f}) \mathbf d_{F}^{\underline k - 2}}  \\
        & = \frac{\vol(\hat\O_D^\times)}{8\mathrm{vol}(\widehat R_N^\times)^2}
        \cdot \frac{\zeta_\Q(2)^2 L(1/2, \Pi)}{L(1, \Pi, \Ad)} & \text{\eqref{eqn:volume_term}} \\
        & \hspace{1cm} \cdot I_p^{\mathrm{ord}}(\phi_p \otimes \widetilde{\phi_p}, \breve{\mathbf t}_n)
        \prod_{q \in \Sigma^{-} \cup \{\infty \}} I_q(\phi_q \otimes \widetilde \phi_q) 
        \prod_{q \not\in p \cup \Sigma^-} I_q(\phi_q^\star \otimes \widetilde \phi_q^\star)  \\
        & \hspace{1cm} \cdot  \left(\prod_{i=1}^{3} \frac{48}{(N_i^+ N_i^{-, {\rm sc}} p^{2n})^{(k_i-2)/2} [\SL_2(\Z) : \Gamma_0(N_i p^{2n})] (k_i - 1)} \prod_{q || N_i^-} \frac{\zeta_q(1)^2}{\zeta_q(2)} \prod_{q^2 || N_i^-} \zeta_q(1) \right)  \\
        & \hspace{1cm} \cdot \frac{p^{6n}}{\zeta_p(2)^2} \frac{\omega_F^{-1/2}(\hat N_1^+)}{\omega_{F}(\hat{\mathbf{d}_f}) \mathbf d_{F}^{\underline k - 2}} \\
        & = \frac{48^3 \vol(\hat\O_D^\times)}{8\mathrm{vol}(\widehat R_N^\times)^2}
        \cdot \frac{\zeta_\Q(2)^2 L(1/2, \Pi)}{\langle f^\circ, f^\circ \rangle \langle g^\circ, g^\circ \rangle \langle h^\circ, h^\circ \rangle} \cdot \left( \prod\limits_{q | N} B_{\Pi_{F,q}} \right) \cdot 2^{-k_1-k_2-k_3} w(f^\circ)w(g^\circ)w(h^\circ)  & \text{\eqref{eqn:Pet_and_adjoint_triple}} \\
        & \hspace{1cm} \cdot I_p^{\mathrm{ord}}(\phi_p \otimes \widetilde{\phi_p}, \breve{\mathbf t}_n)
        \prod_{q \in \Sigma^{-} \cup \{\infty \}} I_q(\phi_q \otimes \widetilde \phi_q) 
        \prod_{q \not\in p \cup \Sigma^-} I_q(\phi_q^\star \otimes \widetilde \phi_q^\star)  \\
        & \hspace{1cm} \cdot \left(\prod_{i=1}^{3} \frac{1}{(N_i^+ N_i^{-, {\rm sc}})^{\frac{k_i-2}{2}} (k_i - 1)} \right) \left( \prod_{q || N^-} \frac{\zeta_q(1)^6}{\zeta_q(2)^3} \right) \left( \prod_{q \in \Sigma^{-, 2}} \frac{\zeta_q(1)^4}{\zeta_q(2)} \right) \left( \prod_{q \in \Sigma^{-, 3}} \zeta_q(1)^3 \right)  \\
        & \hspace{1cm} \cdot \left( \prod_{i=1}^3 [\SL_2(\Z):\Gamma_0(p^{c_i})] p^{-2n} (1 + p^{-1})^{-1} \right)  \frac{p^{6n}}{\zeta_p(2)^2} \frac{\omega_F^{-1/2}(\hat N_1^+)}{\omega_{F}(\hat{\mathbf{d}_f}) \mathbf d_{F}^{\underline k - 2}} \\
        & = 2^5 3^2
        \cdot \frac{\zeta_\Q(2)^2 L(1/2, \Pi)}{\langle f^\circ, f^\circ \rangle \langle g^\circ, g^\circ \rangle \langle h^\circ, h^\circ \rangle} \cdot \left( \prod\limits_{q | N} B_{\Pi_{F, q}} \right) \cdot 2^{-k_1-k_2-k_3} w(f^\circ)w(g^\circ)w(h^\circ)  & \text{\eqref{eqn:volume_by_volume_squared}} \\
        & \hspace{1cm} \cdot I_p^{\mathrm{ord}}(\phi_p \otimes \widetilde{\phi_p}, \breve{\mathbf t}_n)
        \prod_{q \in \Sigma^{-} \cup \{\infty \}} I_q(\phi_q \otimes \widetilde \phi_q) 
        \prod_{q \not\in p \cup \Sigma^-} I_q(\phi_q^\star \otimes \widetilde \phi_q^\star)  \\
        & \hspace{1cm} \cdot \left(\prod_{i=1}^{3} \frac{1}{(N_i^+ N_i^{-, {\rm sc}} )^{\frac{k_i-2}{2}} (k_i - 1)} \right) \left( \prod_{q || N^-} q \frac{\zeta_q(1)^5}{\zeta_q(2)^3} \right) \\
        & \hspace{1cm} \cdot \left( \prod_{q \in \Sigma^{-, 2} } q^3 \frac{\zeta_q(1)^5}{\zeta_q(2)^3} \right) \left( \prod_{q \in \Sigma^{-, 3}} q^3 \frac{\zeta_q(1)^4}{\zeta_q(2)^2} \right) \left( \prod_{q | N^+} \frac{\zeta_q(1)^2}{|N|_q^2 \zeta_q(2)^2} \right )  \\
        & \hspace{1cm} \cdot \left( \prod_{i=1}^3 [\SL_2(\Z):\Gamma_0(p^{c_i})] p^{-2n} (1 + p^{-1})^{-1} \right)  \frac{p^{6n}}{\zeta_p(2)^2} \frac{\omega_F^{-1/2}(\hat N_1^+)}{\omega_{F}(\hat{\mathbf{d}_f}) \mathbf d_{F}^{\underline k - 2}}.
    \end{align*}
    Next, we recall the computation of some of factors from~\cite[p.\ 474, p.\ 478]{Hsieh2021}:
    \begin{align}
        I_\infty(\phi_\infty \otimes \widetilde \phi_\infty) & = (4 \pi^2)^{-1} (k_1 - 1)(k_2 - 1)(k_3 - 1) \\
        I_q(\phi_q \otimes \widetilde \phi_q) & = 2 \zeta_q(1)^{-2} & \text{for $q||N^-$}, \\
        B_{\Pi_q} & = (-1) \frac{\zeta_q(2)^3}{\zeta_q(1)^3} & \text{for $q||N^-$}.
    \end{align}
    Plugging these into the above equations and recalling that $\zeta_\Q(2) = \frac{\pi}{6}$ and $\omega_{F,q}^{1/2}(N_f^+) B_{\Pi_q} = B_{\Pi_{F,q}}$ for $q \neq p$ by definition (\cite[p.\ 477]{Hsieh2021}), we can simplify the final expression in the chain of equalities to:
    \begin{align*}
        \frac{\Theta_{\mathbf F^{D \star}}(\kappa)^2}{\langle F^D, F^D \rangle} & = (-2)^{\# \Sigma^{-,0}} 2^{1-k_1-k_2-k_3} \cdot N^-
        \cdot \frac{L(1/2, \Pi)}{ \langle f^\circ, f^\circ \rangle \langle g^\circ, g^\circ \rangle \langle h^\circ, h^\circ \rangle} \cdot \left( \prod\limits_{q | N/d^-} B_{\Pi_{q}} \right) \cdot w(f^\circ)w(g^\circ)w(h^\circ)  \\
        & \hspace{1cm} \cdot I_p^{\mathrm{ord}}(\phi_p \otimes \widetilde{\phi_p}, \breve{\mathbf t}_n)
        \prod_{q^2 || N^-} I_q(\phi_q \otimes \widetilde \phi_q) 
        \prod_{q | N^+} I_q(\phi_q^\star \otimes \widetilde \phi_q^\star)  \\
        & \hspace{1cm} \cdot \left(\prod_{i=1}^{3} \frac{1}{(N_i^+ N_i^{-, {\rm sc}} )^{\frac{k_i-2}{2}}} \right) \left( \prod_{q \in \Sigma^{-, 2} } q^3 \frac{\zeta_q(1)^5}{\zeta_q(2)^3} \right) \left( \prod_{q \in \Sigma^{-, 3}} q^3 \frac{\zeta_q(1)^4}{\zeta_q(2)^2} \right) \left( \prod_{q | N^+} \frac{\zeta_q(1)^2}{|N|_q^2 \zeta_q(2)^2} \right )  \\
        & \hspace{1cm} \cdot \left( \frac{[\SL_2(\Z) : \Gamma_0(p^{c_i})]}{1 + p^{-1}} \right) \zeta_p(2)^{-2} \frac{\omega_F^{-1/2}(\hat N^-)}{\omega_{F}(\hat{\mathbf{d}_f}) \mathbf d_{F}^{\underline k - 2}}.
    \end{align*}
    Next, we recall that for $F \in \{f, g, h\}$:
    \begin{align*}
        w(F) & = \prod_{\ell < \infty} \epsilon(1/2, \pi_{F, \ell}) \\
        & = \epsilon(1/2, \pi_{F, p}) \prod_{q || N^-} \epsilon(1/2, \pi_{F, q}) \prod_{q| N/d^-} \epsilon(1/2, \pi_{F,q})
    \end{align*}
    and for $q || N^-$:
    $$-1 = \epsilon(1/2, \Pi_q) = \omega_{F, q}^{-1/2}(q) \epsilon(1/2, \pi_{f_1, q})\epsilon(1/2, \pi_{f_2, q})\epsilon(1/2, \pi_{f_3, q}).$$
    We define the away-from-$\Sigma^-$ part of the conductor to be:
    $$\epsilon^{\Sigma^-}(F) = \prod_{q | N/d^-} \epsilon(1/2, \pi_{F, q}) |N_F|_q^{(2 - \underline k)/2} \in \widehat \Z_{(p)}^\times.$$
    This gives:
    \begin{align*}
        \frac{\Theta_{\mathbf F^{D \star}}(\kappa)^2}{\langle F^D, F^D \rangle} & = 2^{\# \Sigma^{-,0} + 1-k_1-k_2-k_3} \cdot N^-
        \cdot \frac{L(1/2, \Pi)}{ \langle f^\circ, f^\circ \rangle \langle g^\circ, g^\circ \rangle \langle h^\circ, h^\circ \rangle} \cdot \epsilon^{\Sigma^-}(f^\circ) \cdot \epsilon^{\Sigma^-}(g^\circ) \cdot \epsilon^{\Sigma^-}(h^\circ)  \\
        & \hspace{1cm} \cdot I_p^{\mathrm{ord}}(\phi_p \otimes \widetilde{\phi_p}, \breve{\mathbf t}_n)  B_{\Pi_{F, p}} \left( \prod_{i=1}^3 \frac{[\SL_2(\Z) : \Gamma_0(p^{c_i})]  \epsilon(1/2, \pi_{f_i, p})}{1 + p^{-1}} \right) \cdot \zeta_p(2)^{-2} \\
        & \hspace{1cm} \cdot \prod_{q \in \Sigma^{-, 2}}  I_q(\phi_q \otimes \widetilde \phi_q) B_{\Pi_q} q \frac{\zeta_q(1)^5}{\zeta_q(2)^3} \omega_{F, q}^{-1}(\hat{\mathbf{d}_f})   |\mathbf d_F^{\underline k - 2}|_q  \\
        & \hspace{1cm} \cdot \prod_{q \in \Sigma^{-, 3}}  I_q(\phi_q \otimes \widetilde \phi_q) B_{\Pi_q} q \frac{\zeta_q(1)^4}{\zeta_q(2)^2} \omega_{F, q}^{-1}(\hat{\mathbf{d}_f})   |\mathbf d_F^{\underline k - 2}|_q  \\
        & \hspace{1cm} \cdot \prod_{q | N^+} I_q(\phi_q^\star \otimes \widetilde \phi_q^\star) B_{\Pi_{q}} \frac{\zeta_q(1)^2}{|N|_q^2\zeta_q(2)^2} \omega_{F, q}^{-1}(\hat{\mathbf{d}_f}) |\mathbf d_{F}^{\underline k - 2}|_{q}.
    \end{align*}
    
    To finish the proof, we recall from \cite[p.\ 477]{Hsieh2021} that:
    \begin{equation}
        \frac{B_{\Pi_p^{\rm ord}}^{[n]}}{B_{\Pi_{F,p}}} = \omega_{f,p}^{1/2}(-p^{-2n}) \prod_{i=1}^3 \frac{\alpha_{f_i, p} |\cdot|_p^{1/2}(p^{2n})}{\epsilon(1/2, \pi_{f_i, p})} \cdot \frac{[\SL_2(\Z): \Gamma_0(p^{c_i})]}{1 + p^{-1}} \cdot \mathcal E_p(f_i, \Ad). \qedhere
    \end{equation}
\end{proof}

\subsection{Computation of local factors}

By Proposition~\ref{prop:interpolation_theta}, proving Theorem~\ref{thm:interpolation} amounts to computing the local factors. For $p$ and $q | N^+$, they were already computed by Hsieh.

\begin{proposition}[Hsieh]\label{prop:local_Hsieh}
    \leavevmode
    \begin{enumerate}
        \item We have that:
        $$\mathscr I_{\Pi_{\kappa, p}}^{\bal} = \mathcal E_{\bal}(\Pi_{\kappa, p}) \cdot \frac{1}{L(1/2, \Pi_{\kappa, p})}.$$
        
        \item For $q | N^+$: $$\mathscr I_{\Pi_{\kappa, q}}^\star = \mathfrak f_{\mathbf F, q}(\kappa) \cdot \begin{cases}
            (1 + q^{-1})^2 & \text{if }q \in \Sigma_{\mathrm{exc}}, \\
            1 & \text{otherwise}
        \end{cases} $$ 
    \end{enumerate}
\end{proposition}
\begin{proof}
    Part (1) is \cite[Proposition\ 5.6]{Hsieh2021} and part (2) is \cite[Proposition\ 6.12]{Hsieh2021}.
\end{proof}

Therefore, it remains to compute $\mathscr I_{\Pi_{\kappa, \ell}}$ for $\ell \in \Sigma^-$ such that $\ell^2 || N$, i.e.\ for $\ell \in \Sigma^{-, 2}$ and $\ell \in \Sigma^{-, 3}$. 

\begin{proposition}\label{prop:local_two_sc}
    Suppose $\ell \in \Sigma^{-,2}$ and $\omega_{F, \ell} = 1$. Then:
    $$\mathscr I_{\Pi_{\kappa, \ell}}^\star = - \frac{1 + \epsilon_1 \epsilon_2 \sqrt{\omega_3(\ell)}}{2} \frac{\ell^{-2(k_1 + k_2 + k_2) + 13}}{\zeta_\ell(2)}.$$
\end{proposition}
\begin{proof}
    We use:
    \begin{itemize}
        \item Proposition~\ref{prop:local_int_scscsp}:
    $$\frac{I_{\ell}'(\phi_\ell^{\epsilon_\ell} \otimes\widetilde{\phi_\ell^{\epsilon_\ell}})}{\langle \phi_\ell^{\epsilon_\ell},\widetilde{\phi_\ell^{\epsilon_\ell}} \rangle} = \frac{1 + \epsilon_1 \epsilon_2(\sqrt{\omega_3(\ell)})}{2},$$
        \item Propositions~\ref{prop:local_L-factors_triple},~\ref{prop:local_L-factors_adjoint}:
        $$L(s, \Pi_{\ell}, \Ad) = \zeta_\ell(s+1) (\zeta_\ell(2s)/\zeta_\ell(s))^2, \qquad L(s, \Pi_{\ell}) = \zeta_\ell(s + 1/2)^2.$$
    \end{itemize}
    Therefore:
    $$ I_\ell(\phi_\ell^{\epsilon_\ell} \otimes\widetilde{\phi_\ell^{\epsilon_\ell}})  = \frac{\zeta_\ell(2) \zeta_\ell(2)^2}{\zeta_\ell(2)^2 \zeta_\ell(1)^2 \zeta_\ell(1)^2} \frac{1 + \epsilon_1 \epsilon_2 \sqrt{\omega_3(\ell)}}{2} = \frac{\zeta_\ell(2)}{\zeta_\ell(1)^4} \frac{1 + \epsilon_1 \epsilon_2 \sqrt{\omega_3(\ell)}}{2} .$$
    Note that $B_{\pi_\ell} = 1$ if $\pi_\ell$ is supercuspidal and $\pi_\ell \otimes \xi \iso \pi_\ell$ for the unramified quadratic character $\xi$ of $\Q_\ell^\times$ (cf.\ \cite[Section 4.2]{collins2018numerical}). 
    Therefore, $B_{\Pi_\ell} = B_{\pi_1} B_{\pi_2} B_{\pi_3} = (-1) \frac{\zeta_\ell(2)}{\zeta_\ell(1)}$ (similarly to \cite[p. 478]{Hsieh2021}).
    Altoghether, we have that:
    \begin{align*}
        \mathscr I_{\Pi_{\kappa, \ell}}^\star & = - \frac{\zeta_\ell(2)}{\zeta_\ell(1)^4} \frac{1 + \epsilon_1 \epsilon_2 \sqrt{\omega_3(\ell)}}{2} \frac{\zeta_\ell(2)}{\zeta_\ell(1)} \ell \frac{\zeta_\ell(1)^5}{\zeta_\ell(2)^3} |\mathbf d_F^{\underline k-2}|_\ell \\
        & = - \frac{1 + \epsilon_1 \epsilon_2 \sqrt{\omega_3(\ell)}}{2} \frac{  \ell |\ell|_\ell^{2(k_1 + k_2 + k_2 - 6)}}{\zeta_\ell(2)},
    \end{align*}
    as claimed.
\end{proof}

\begin{proposition}\label{prop:local_three_sc}
    Suppose $\ell \in \Sigma^{-,3}$ and $\omega_{F, \ell} = 1$. Then:
    $$\mathscr I_{\Pi_{\kappa, \ell}}^\star = \frac{1 + \epsilon_1 \epsilon_2 \epsilon_3}{4} \frac{\ell^{-2(k_1 + k_2 + k_3) + 13}}{\zeta_\ell(2)}.$$
\end{proposition}
\begin{proof}
    We use:
    \begin{itemize}
        \item Proposition~\ref{prop:local_integral_scscsc}:
    $$\frac{I_{\ell}'(\phi_\ell^{\epsilon_\ell} \otimes\widetilde{\phi_\ell^{\epsilon_\ell}})}{\langle \phi_\ell^{\epsilon_\ell},\widetilde{\phi_\ell^{\epsilon_\ell}} \rangle} = \frac{1 + \epsilon_1 \epsilon_2 \epsilon_3}{4},$$
        \item Propositions~\ref{prop:local_L-factors_triple},~\ref{prop:local_L-factors_adjoint}:
        $$L(s, \Pi_{\ell}, \Ad) = (\zeta_\ell(2s)/\zeta_\ell(s))^3, \qquad L(s, \Pi_{\ell}) = \zeta_\ell(2s).$$
    \end{itemize}
    Therefore:
    $$I_\ell( \phi_\ell^{\epsilon_\ell} \otimes\widetilde{\phi_\ell^{\epsilon_\ell}}) = \frac{\zeta_\ell(2)^3 }{\zeta_\ell(2)^2 \zeta_\ell(1)^3 \zeta_\ell(1)} \frac{1 + \epsilon_1 \epsilon_2 \epsilon_3}{4} = \frac{\zeta_\ell(2)}{\zeta_\ell(1)^4} \frac{1 + \epsilon_1 \epsilon_2 \epsilon_3}{4}.$$
    Moreover, $B_{\Pi_\ell} = 1$. Altogether, we have that:
    \begin{align*}
        \mathscr I_{\Pi_{\kappa, \ell}}^\star = & \frac{\zeta_\ell(2)}{\zeta_\ell(1)^4} \frac{1 + \epsilon_1 \epsilon_2 \epsilon_3}{4} \cdot \ell \frac{\zeta_\ell(1)^4}{\zeta_\ell(2)^2} \omega_{F, \ell}^{-1}(\mathbf d_f^{\underline k-2}) |\mathbf d_F^{\underline{k} - 2}|_\ell \\
        & = \frac{1 + \epsilon_1 \epsilon_2 \epsilon_3}{4} \frac{\ell |\ell|_\ell^{2(k_1 + k_2 + k_3-6)}}{\zeta_\ell(2)},
    \end{align*}
    as claimed.
\end{proof}

\subsection{Finishing the proof of Theorem~\ref{thm:interpolation}}

Putting everything together gives the interpolation property.

\begin{proof}[Proof of Theorem~\ref{thm:interpolation}]
    We compute using the definition of $\mathcal L_{\mathbf F, \epsilon}^{\bal}$:
    \begin{align*}
        (\mathcal L_{\mathbf F, \epsilon}^{\rm bal}(\kappa))^2 & = \frac{\Theta_{\mathbf F'^{D\star, \epsilon}}^2}{\| \mathbf F^{D \star, \epsilon} \|^2} \cdot 2^{- (|\Sigma^-| + 1 -k_1-k_2-k_3)} \cdot \epsilon^{\Sigma^-}(\mathbf F)(\kappa)^{-1} \cdot \mathfrak f_{\mathbf F'}(\kappa)^{-1} \prod_{\ell \in \Sigma^{-, \mathrm{sc}}} \frac{\ell^{2(k_1 + k_2 + k_3) - 13}}{\zeta_\ell(2)} \\
        & = \frac{L(1/2, \Pi_{\kappa})}{\langle \mathbf F_{\kappa}, \mathbf F_{\kappa} \rangle}  \cdot \frac{\mathscr I_{\Pi_{\kappa, p}}^{\bal}}{\mathcal E_p(\mathbf F_\kappa, \Ad)} \cdot \prod_{\ell^2 || N^-} \mathscr I_{\Pi_{\kappa, \ell}}^\star \frac{\ell^{2(k_1 + k_2 + k_3) -13}}{\zeta_\ell(2)} \cdot \prod_{q | N^+} I_{\Pi_{\kappa, q}}^\star \cdot f_{\mathbf F'}(\kappa)^{-1} \\
        & = \frac{L(1/2, \Pi_{\kappa})}{\langle \mathbf F_{\kappa}, \mathbf F_{\kappa} \rangle}  \cdot \frac{\mathcal E_{\bal}(\Pi_{\kappa, p})}{L(1/2, \Pi_{\kappa, p}) \mathcal E_p(\mathbf F_\kappa, \Ad)} \cdot \prod_{\ell \in \Sigma^{-,2}} \frac{1 + \epsilon_1 \epsilon_2 \sqrt{\omega_3(\ell)}}{2} \prod_{\ell \in \Sigma^{-,3}} \frac{1 + \epsilon_1 \epsilon_2 \epsilon_3}{4} \cdot \prod_{q \in \Sigma_{\rm exc}} \frac{1}{(1 + q^{-1})^2} 
    \end{align*}
    using Proposition~\ref{prop:interpolation_theta} for the second equality and Propositions \ref{prop:local_Hsieh}, \ref{prop:local_two_sc}, \ref{prop:local_three_sc} for the final one. This proves the theorem.
\end{proof}

\section{Application: Elliptic Stark Conjecture in rank one}\label{section: Application: Elliptic Stark Conjecture in rank one}

We are now ready to discuss the analogue of the Elliptic Stark Conjecture of Darmon--Lauder--Rotger~\cite{Darmon_Lauder_Rotger_Stark_points} in this setting. We will consider the restriction of $\mathcal L_p^{\bal}$ to the weights $(k_1, k_2, k_3) = (2, \ell, m)$ with the goal of studying the point $(2,1,1)$ which is not balanced and hence lies outside of the interpolation range.\footnote{We hope that using $\ell$ for the weight instead of the prime of supersingular type will cause no confusion to the reader.} We state a conjecture expressing this value in terms of arithmetic data and prove it in some special cases.

\subsection{Statement of the conjecture}

Let $E$ be an elliptic curve over $\Q$ and let $V_p(E)$ be the associated $p$-adic Galois representation. Let $\varrho \colon G_\Q \to \GL(V_\varrho)$ be an Artin representation, i.e. a complex continuous Galois representation, which factors through a finite quotient $\Gal(H/\Q)$ of $G_\Q$, with coefficients in a finite extension $L$ of $\Q$. We consider the Hasse--Weil--Artin $L$-series
$L(E, \varrho, s)$ associated with the $p$-adic Galois repersentation $V_p(E) \otimes V_\varrho$ of $G_\Q$. The equivariant version of the Birch--Swinnerton-Dyer Conjecture asserts that the analytic rank of $L(E, \varrho, s)$ at $s = 1$ is equal to the rank of the $V_\varrho$-isotypic component of $E(H) \otimes L$:
$$\mathrm{ord}_{s=1} L(E, \varrho, s) = \dim_L \Hom_{G_\Q}(V_\varrho, E(H) \otimes L).$$

Note that the value $L(E, \varrho, s)$ at $s = 1$ is outside of the convergence region of the Euler product defining the $L$-function. We put ourselves in a situation where the analytic continuation and functional equation for the $L$-function $L(E, \varrho, s)$ is known:
\begin{itemize}
    \item let $f$ be the modular form of level $N_1$ associated with the elliptic curve $E$ of conductor $N_1$,
    \item suppose that $\varrho = \varrho_{gh} = \varrho_g \otimes \varrho_h$ for two odd irreducible two-dimensional Artin representations $\varrho_g$ and $\varrho_h$; let $g$ and $h$ be the modular forms of conductors $N_2$ and $N_3$ corresponding to $\varrho_g$ and $\varrho_h$, respectively.
\end{itemize}
Then
$$L(E, \varrho_{gh}, s) = L(f \times g \times h, s)$$
is the triple product $L$-function studied by Garret. We assume that:
$$\det \varrho_g  \det \varrho_h = 1,$$
i.e. if $\chi$ is the character of $g$, then $\chi^{-1}$ is the character of $h$. Then the representation $V_E \otimes V_g \otimes V_h$ is self-dual, and hence there is a functional equation with the root number $\epsilon(E, \varrho_{gh}) = \pm 1$. Moreover, the global root number is a product of finite local root numbers:
$$\epsilon(E, \varrho_{gh}) = \prod_{v | \lcm(N_1, N_2, N_3)} \epsilon_v(E, \varrho_{gh}),$$
because $\epsilon_\infty(E, \varrho_{gh}) = 1$.

In the seminal work~\cite{Darmon_Lauder_Rotger_Stark_points}, Darmon--Lauder--Rotger studied the $L$-value $L(E, \varrho_{gh}, s)$ when the analytic rank is even and at least two, using $p$-adic analytic method. Therefore, they assume that $\epsilon(E, \varrho_{gh}) = 1$, and in fact that $\epsilon_v(E, \varrho_{gh}) = 1$ for all $v$.

In rank two, they proposed the Elliptic Stark Conjecture: a formula relating the value of a triple product $p$-adic $L$-function associated with $f$, $g$, $h$ to a regulator of $p$-adic logarithms of points on $E(H) \otimes L$ and a $p$-adic logarithm of a Stark unit.

The main motivation for our work is to develop a rank one version of this conjecture. Our first hypothesis is therefore.

\begin{hypothesisalph}\label{hyp:epsilon-1}
    The global root number $\epsilon(E, \varrho_{gh})$ is $-1$. Therefore, there is an odd number of finite places $v$ such that 
    $$\epsilon_v(E, \varrho_{gh}) = -1.$$
\end{hypothesisalph}

Next, we want to construct a balanced $p$-adic $L$-function associated with the $L$-function $L(E, \varrho_{gh}, s) = L(f \times g \times h, s)$. Under Hypothesis~\ref{hyp:epsilon-1}, this $L$-function vanishes at $s = 1$. However, if we consider Hida families associated with $g$ and $h$, and the $L$-value $L(f \times g_\ell \times h_\ell, s)$ for $\ell \geq 2$, then
$$\epsilon_\infty(f \times g_\ell \times h_\ell) = -1, \qquad \epsilon_v(f \times g_\ell \times h_\ell) = \epsilon_v(E, \varrho_{gh}) \text{ for $v$ finite},$$
$$\epsilon(f \times g_\ell \times h_\ell) = +1.$$
In particular, we expect that the central $L$-value for $L(f \times g_\ell \times h_\ell, s)$ is generically non-vanishing and hence there should be a non-vanishing $p$-adic $L$-function interpolating these values.

To define the Hida families, assume that $g$ and $h$ are ordinary at $p$ and consider $p$-stabilizations of $g$ and $h$. For $F \in \{g, h\}$, $F \in S_1(N, \chi_F)$, suppose the Hecke polynomial is:
$$x^2 - a_p(F)x + \chi_F(p) = (x - \alpha_p(F))(x - \beta_p(F)) $$
with roots $\alpha_p(F)$ and $\beta_p(F)$. We then have $p$-stabilizations $F_\alpha, F_\beta \in S_1(Np, \chi_F)$ such that:
$$U_p F_\alpha = \alpha F_\alpha, \quad U_p F_\beta = \beta F_\beta.$$
We will assume that $F$ is {\em regular}, i.e.\ $\alpha_p(F) \neq \beta_p(F)$, and hence it has two distinct $p$-stabilizations. 

We then make the following classicality hypothesis.\medskip

\begin{hypothesisalph}~\label{hyp:classicality}
    Assume that $f$ is ordinary at $p$. For $F \in \{g, h\}$, $F$ is ordinary and regular at $p$ (i.e. $\varrho_g$ is irreducible and $\varrho_g(\sigma_p)$ has two distinct eigenvalues for the Frobenius $\sigma_p$), and it is not the theta series of a character of a real quadratic field in which $p$ splits. 
    
\end{hypothesisalph}

Under this hypothesis, Darmon--Lauder--Rotger, based on results of Cho and Vatsal and of Bella\"iche and Dimitrov, proved that there are no non-classical $p$-adic modular forms in the generalized eigenspace of $F_\alpha$. 

\begin{proposition}[{\cite[Proposition 1.1]{Darmon_Lauder_Rotger_Stark_points}}]\label{prop:classicality}
    Under Hypothesis~\ref{hyp:classicality}, the natural inclusion:
    $$S_1(Np, \chi_F)_{\C_p}[F_\alpha] \hookrightarrow S_1^{\mathrm{oc}, \mathrm{ord}}(N, \chi_F)\llbracket F_\alpha\rrbracket$$
    is an isomorphism of $\C_p$-vector spaces. 
\end{proposition}

By Proposition~\ref{prop:classicality}, there exist Hida families $\mathbf g_\alpha$ and $\mathbf h_\alpha$ whose specializations at weight one are $g_\alpha$ and $h_\alpha$, respectively.

We want to consider a balanced $p$-adic $L$-function $\mathcal L_p^{\bal}$ associated with $f$ and the Hida families $\mathbf g_\alpha$ and $\mathbf h_\alpha$. Therefore, we need to make another hypothesis which will guarantee that Hypothesis~\ref{hyp:cond_at_most_2} needed to construct the balanced $p$-adic $L$-function is satisfied.

\begin{hypothesisalph}\label{hyp:Ccond_at_most_2}
    For each prime $q$ such that $\epsilon_q(E, \varrho_{gh}) = -1$, $v_q(N_i) \leq 2$ for $i = 1, 2, 3$.
\end{hypothesisalph}

Under this hypothesis, we have constructed a balanced triple product $p$-adic $L$-function associated with the triple $f$, $g$, $h$, and the Elliptic Stark Conjecture concerns its value at $(2,1,1)$, which is outside of the range of interpolation. 

\begin{definition}
    Let $f$, $g$, $h$ be three modular forms of weights $2, 1, 1$ satisfying Hypotheses~\ref{hyp:epsilon-1}, \ref{hyp:classicality}, \ref{hyp:Ccond_at_most_2}. Let $f_\alpha$, $g_\alpha$, $h_\beta$ be ordinary $p$-stabilizations of $f$, $g$, and $h$, and consider Hida families $\mathbf f$, $\mathbf g$, $\mathbf h$ specializing to $p$-stabilizations $f_\alpha$, $g_\alpha$, $h_\alpha$ of $f$, $g$, $h$, respectively. Define the {\em 2-variable triple product $p$-adic $L$-function} by:
    $$\mathcal L_p^{\bal, f} = 2^{-\#\Sigma^{-,3}} \cdot \prod_{q \in \Sigma_{\mathrm{exc}}}  (1 - q^{-1})^{-2}  \cdot \mathcal E_p(f_\alpha, \Ad) \cdot \mathcal L_{\mathbf f \times \mathbf g \times \mathbf h, \underline{+1}}^{\bal}$$
    using Definition~\ref{def:p-adic_L-fun} for $\epsilon_q = \begin{cases}
        (+1, +1) & q \in \Sigma^{-, 2}, \\
        (+1, +1, +1) & q \in \Sigma^{-, 2}. \\
    \end{cases}$
\end{definition}

\begin{remark}
    One could presumably weaken or even remove the assumption that $f$ is ordinary at~$p$ by defining the 2-variable $p$-adic $L$-function directly, instead of considering a Hida family through $f_\alpha$ and referring to the 3-variable $p$-adic $L$-function. However, we decided not to pursue this point here.
\end{remark}

As a corollary to Theorem~\ref{thm:interpolation}, we get the following interpolation property.

\begin{corollary}\label{cor:interpolation}
    For $\ell, m \geq 2$ such that $2 - \ell \leq m \leq 2 + \ell$ and $\ell + m \equiv 0 \pmod 2$:
    $$\mathcal L_p^{\bal, f}(\ell, m)^2 = \frac{\Lambda(f \times g_\ell \times h_m, (\ell + m)/2)}{\langle f, f \rangle \langle g_\ell, g_\ell \rangle \langle h_m, h_m \rangle} \cdot \mathcal E_p^{\bal}(\mathbf V_\kappa) \cdot \mathcal E_p(f_\alpha, \Ad),$$
    where $\Lambda(f \times g_\ell \times h_m, s)$ is the completed triple product $L$-function.
\end{corollary}

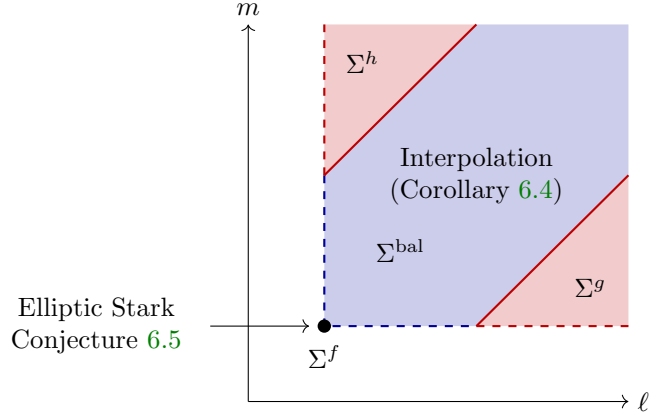
\begin{figure}
	\begin{tikzpicture}
 	\draw[->] (0,0) -- (5,0) node[right] {$\ell$};
	\draw[->] (0,0) -- (0,5) node[above] {$m$};

	\fill[color=darkblue!20]    (1,1) -- (1,3) -- (3,5) -- (5,5) -- (5,3)  -- (3,1);
        \draw (2, 2) node {$\Sigma^{\bal}$};
        \draw (3, 3) node {\begin{tabular}{c} Interpolation \\ (Corollary~\ref{cor:interpolation}) \end{tabular}};
        \draw[dashed, thick, color=darkblue] (1,3) -- (1,1) -- (3,1);
        
        \fill[color=darkred!20]    (3,1) -- (5,1) -- (5,3);
        \draw[color=darkred, thick, dashed]  (3,1) -- (5,1);
        \draw[color=darkred, thick]  (3,1) -- (5,3);
        \draw (4.5, 1.5) node {$\Sigma^g$};

        \fill[color=darkred!20]    (1,3) -- (1,5) -- (3, 5);
        \draw[color=darkred, thick, dashed]  (1,3) -- (1, 5);
        \draw[color=darkred, thick]  (1, 3) -- (3, 5);
        \draw (1.5, 4.5) node {$\Sigma^h$};
        
        \fill (1,1) circle(2.5pt);
        \draw (1, 0.6) node {$\Sigma^f$};
        \draw (-2, 1) node {\begin{tabular}{c} Elliptic Stark \\ Conjecture~\ref{conj:ES} \end{tabular}};
        \draw[->](-0.5, 1) -- (0.8, 1);

	\end{tikzpicture}
    
    \caption{We fix $k = 2$, and consider two weights $\ell, m$ with $\ell + m \equiv 0 \pmod 2$. We indicate the four regions for the weights $\ell$, $m$; $\Sigma^F$ is the region where $F$ is dominant, and $\Sigma^{\bal}$ is the region where the weights are balanced and $\ell, m \geq 2$. We also indicate the point $(\ell, m) = (1,1)$ where our Elliptic Stark Conjecture~~\ref{conj:ES} applies.}
\end{figure}

The value $\mathcal L_p^{\bal, f}(1,1) = \mathcal L_{\mathbf f \times \mathbf g \times \mathbf h, \epsilon}^{\bal}(2,1,1)$ will be expressed in terms of arithmetic data associated with the triple $(E, \varrho_g, \varrho_h)$. We introduce this next.

Let $\sigma_p \in \Gal(H/\Q)$ be the Frobenius at $p$ associated with an embedding $H \to \Q_p^{\mathrm{ur}}$. Under Hypothesis~\ref{hyp:classicality}, for $F \in \{g, h\}$, we have that
$$\varrho_F(\sigma_p) = \begin{pmatrix}
    \alpha_p(F) & 0 \\
    0 & \beta_p(F)
\end{pmatrix}$$
and we may hence consider the one-dimensional eigenspace $V_F^\alpha \subseteq V_F$ for $\sigma_p$ associated with the eigenvalue $\alpha(F)$. This determines a one-dimensional $L$-subspace:
$$V_{\alpha \alpha} = V_g^\alpha \otimes V_h^\alpha \subseteq V_g \otimes V_h$$
and we fix an element $v_{\alpha \alpha} \in V_{\alpha \alpha}$. On the other hand, under Hypothesis~\ref{hyp:epsilon-1}, we expect that
$r(E, \varrho_{gh}) = \dim_L \Hom_{G_\Q}(V_{gh}, E(H) \otimes L) \geq 1$. 
If $\dim_L \Hom_{G_\Q}(V_{gh}, E(H) \otimes L) = 1$, we choose its basis
$$\Phi \colon V_{gh} \to E(H) \otimes L,$$
and let
$$E(H)_L^{V_{gh}} = \Phi(V_{gh}) \subseteq E(H) \otimes L.$$
We will consider the point:
$$\Phi(v_{\alpha \alpha}) \in E(H)_L^{V_{gh}} \subseteq E(H) \otimes L.$$

Next, associated with $F \in \{g, h\}$, we have an adjoint representation $\Ad_F = \Hom^0(V_F, V_F)$, with Frobenius eigenvalues $1$, $\frac{\alpha_p(F)}{\beta_p(F)}$,  $\frac{\beta_p(F)}{\alpha_p(F)}$. By~\cite[Proposition 1.5]{Darmon_Lauder_Rotger_Stark_points}, we have that
$$\Hom_{G_\Q}(\Ad_F, \O_{H_F}^\times \otimes L)$$
is one-dimensional. Let $\varphi$ be its basis and consider
$$(\O_{H_F}^\times)_L^{\Ad_F} = \varphi(\Ad_F) \subseteq \O_{H_F}^\times \otimes L.$$
Under the extra assumption that $\alpha_p(F) \neq -\beta_p(F)$ or $\varrho_F$ is induced from a character of an imaginary quadratic field in Hypothesis~\ref{hyp:classicality}, the subspace
$$U_{F_\alpha} = \left\{ u \in (\O_{H_F}^\times)_L^{\Ad_F} \ \middle| \ \sigma_p(u) = \frac{\alpha_p(F)}{\beta_p(F)} \right\}$$
is one-dimensional~\cite[Lemma 1.6]{Darmon_Lauder_Rotger_Stark_points}. In any case, we let $u_{F_\alpha} \in U_{F_\alpha}$ be a non-torsion element.

\begin{conjecture}[Elliptic Stark Conjecture, rank one case]\label{conj:ES}
Assume Hypotheses~\ref{hyp:epsilon-1}, \ref{hyp:classicality}, \ref{hyp:Ccond_at_most_2} and that $r(E, \varrho_{gh}) \geq 1$. If $r(E, \varrho_{gh}) > 1$, then $\mathcal L_p^{\bal, f}(1,1) = 0$. If $r(E, \varrho_{gh}) = 1$, then:
$$\mathcal L_p^{\bal, f}(1,1) \sim_{\sqrt{L^\times}} \frac{\log_{E, p}(\Phi(v_{\alpha \alpha}))}{\log_p(u_{g_\alpha})^{1/2}\log_p(u_{h_\alpha})^{1/2}},$$
where
\begin{align*}
    \log_p \colon (\O_{H}^\times)_L & \to H_p \otimes L & \text{$p$-adic logarithm}, \\
    \log_{E, p} \colon  E(H)_L & \to H_p \otimes L & \text{$p$-adic formal group logarithm for $E$}.
\end{align*}
\end{conjecture}

\begin{remark}\label{rmk:sanity_check}
    We would like to thank Alan Lauder for suggesting the following sanity check. The left hand side in the conjecture is independent of the number field $H$, whereas the right hand side seems to depend on the Frobenius $\sigma_p \in \Gal(H/\Q)$ and its eigenvalues for the Artin representations $\varrho_g$ and $\varrho_h$. However, the Frobenius $\sigma_p$ acts on the right hand side as:
    \begin{align*}
        \frac{(\alpha_p(g) \alpha_p(h))^2}{\frac{\alpha_p(g)}{\beta_p(g)} \frac{\alpha_p(h)}{\beta_p(h)}} & = \alpha_p(g) \alpha_p(h) \beta_p(g) \beta_p(h) \\
        & = \chi_g(p) \chi_h(p) \\
        & = 1.
    \end{align*}
\end{remark}

\begin{remark}
    Taking the formal group law exponential $\exp_{E, p}$, we get an interesting formula for a $p$-adic point on the elliptic curve $E$ which is conjecturally in $E(H) \otimes L$:
    $$\exp_p( \mathcal L_p^{f, \bal}(1,1) \log_p(u_{g_\alpha})^{1/2} \log_p(u_{h_\alpha})^{1/2}) \overset?\in E(H)_L^{V_{gh}} \subseteq E(H) \otimes L.$$
\end{remark}

\subsection{Proof in the CM case}

Let $K/\Q$ be an imaginary quadratic field of discriminant $-D_K$ in which $p$ splits. The next goal of the paper is to prove the Elliptic Stark Conjecture~\ref{conj:ES} in the case when $\varrho_g$ and $\varrho_h$ are representations induced from Dirichlet characters of $K$.

\subsubsection{Statement of the theorem}\label{section: Statement of the theorem}

Let $\psi \colon G_K \to \C^\times$ be a finite order character of conductor $\mathfrak c \subseteq \O_K$. We can then consider the associated Artin representation representation 
$$V_\psi = \Ind_{G_K}^{G_\Q} \psi \colon G_\Q \to \GL_2(\C)$$ 
with determinant $\chi = \det V_\psi = \psi \circ \mathrm{Tr}$, where $\mathrm{Tr} \colon G_\Q^{\rm ab} \to G_K^{\rm ab}$ is the transfer map. The weight one modular form associated with this Artin representation is explicitly constructed as a theta series for the character of $\A_K^\times$ corresponding to $\psi$:
$$\theta_\psi \in M_1(D_K \cdot N_{K/\Q}(\mathfrak c), \chi).$$
Moreover, $\theta_\psi$ is a cusp form if and only if $\psi^\sigma \neq \psi$ where $\sigma$ is a generator of $\Gal(K/\Q)$. We will only be interested in this case.

We consider two finite order characters $\psi_g$ and $\psi_h$ such that $\chi_h = \chi_g^{-1}$, where $\chi_g = \psi_g \circ \mathrm{Tr}$ and $\chi_h = \psi_h \circ \mathrm{Tr}$, and $\psi_g^\sigma \neq \psi_g$, $\psi_h^\sigma \neq \psi_h$. Then we have associated modular forms
$$g = \theta_{\psi_g} \in S_1(N_g, \chi) \qquad h = \theta_{\psi_g} \in S_1(N_h, \overline{\chi}).$$
Write $\mathfrak c_\star$ for the conductor of $\psi_\star$ for $\star \in \{g, h\}$.

The goal is to study the Elliptic Stark Conjecture~\ref{conj:ES} in this example. We have a factorization
$$V_{gh} = V_{\psi_g} \otimes V_{\psi_h} \iso V_{\psi_g \psi_h} \oplus V_{\psi_g \psi_h^\sigma},$$
and we will write
$$\psi_1 = \psi_g \psi_h^\sigma,\ \psi_2 = \psi_g \psi_h.$$
This corresponds to the factorization of $L$-functions
$$L(E, \varrho_{gh}, s) = L(E/K, \psi_1, s) L(E/K, \psi_2, s).$$

Hypothesis~\ref{hyp:epsilon-1} is hence equivalent to the following hypothesis.

\begin{hypothesisalphpr}\label{hyp:A'}
    We have that:
    $$\epsilon(E/K, \psi_1) \cdot \epsilon(E/K, \psi_2) = -1,$$
    i.e. without loss of generality, $\epsilon(E/K, \psi_1) = +1$, $\epsilon(E/K, \psi_2) = -1$.
\end{hypothesisalphpr}

Next, we want to work out when Hypothesis~\ref{hyp:classicality} holds in this case. The values $\alpha_p, \beta_p$ depend on the splitting of $p$ in $K$:
$$\{\alpha_p, \beta_p\} = \begin{cases}
    \{\psi(\sigma_{\mathfrak p}), \psi(\sigma_{\overline {\mathfrak p}})  \} & \text{if $(p) = \mathfrak p \overline{\mathfrak p}$ splits}, \\
    \{\sqrt{\psi(\sigma_p)}, -\sqrt{\psi(\sigma_p)}  \} & \text{if $(p)$ is inert}.
\end{cases}.$$
Therefore, Hypothesis~\ref{hyp:classicality} amounts to the following.

\begin{hypothesisalphpr}\label{hyp:B'}
    For $\psi \in \{\psi_g, \psi_h\}$:
    \begin{itemize}
        \item $\theta_\psi$ is ordinary at $p$,
        \item $\theta_\psi$ is not also the theta series of a character of a real quadratic field in which $p$ splits,
        \item when $p$ splits in $K$, $\psi(\sigma_{\mathfrak p}) \neq \psi(\sigma_{\overline{\mathfrak p}})$. 
    \end{itemize}
\end{hypothesisalphpr}

Finally, Hypothesis~\ref{hyp:Ccond_at_most_2} is that for each prime $\ell$ such that $\epsilon_\ell(E_K, \psi_1) \cdot \epsilon_\ell(E_K, \psi_2) = -1$, we have that $v_\ell(N_f), v_\ell(N_{K/\Q} \mathfrak c_g), v_\ell(N_{K/\Q} \mathfrak c_h)  \leq 2$. For simplicity, we make a slightly stronger assumption. 

\begin{hypothesisalphpr}\label{hyp:C'}
    For each prime $\ell$ of $K$ such that $\epsilon_\ell(E_K, \psi_1) \cdot \epsilon_\ell(E_K, \psi_2) = -1$, we have that:
	\begin{itemize}
		\item $\epsilon_\ell(E_K, \psi_2) = +1$,
		\item $v_\ell(N_f) = 1$,  $v_\ell(N_{K/\Q} \mathfrak c_g) = 2, v_\ell(N_{K/\Q} \mathfrak c_h) = 2$.
	\end{itemize}
    Moreover, for each prime $q$ of $K$ such that $\epsilon_q(E_K, \psi_1) \cdot \epsilon_q(E_K, \psi_2) = +1$, we have that $\epsilon_q(E_K, \psi_1) =+1$ and $\epsilon_q(E_K, \psi_2) =+1$.
\end{hypothesisalphpr}

\begin{remark}\label{rmk: expected result for hyp C}
	We expect that the results of the section still hold under the weaker assumption, by replacing of Appendix~\ref{appendix BDP} by a generalization of the results of Brooks~\cite{Brooks}, and by calculating the local integrals when all three representations are supercuspidal. 
\end{remark}

We consider three arithmetic quantities associated with our data of $f$, $\psi_g$, $\psi_h$:
\begin{itemize}
    \item following~\cite[Section 11]{Gross87} (more generally,~\cite{CaiShuTian}),
    \begin{align}\label{eqn:Gross_point}
    	c_{f, \psi_1} = \pi_f\left( \sum\limits_{\sigma \in \Gal(H/\Q)} \psi_1^{-1}(\sigma) s^\sigma  \right),
    \end{align}
    where $\pi_f$ is the projection onto the $f$-isotypic component of an appropriate Shimura set $S$ and $s \in S$ is a Heegner point for $H_{c(\psi_1)}/K$;

    \item following~\cite[p. 37]{Darmon_Lauder_Rotger_Stark_points} (cf.\ \cite[Definition 3.3.1]{LZZ}),
    \begin{align}\label{eqn:Heegner_pt}
        P_{\psi_2} & = \pi_f\left(\sum\limits_{\sigma \in \Gal(H/K)} \psi_2^{-1}(\sigma) t^\sigma\right) \in E(H)_L^{\psi_2} 
    \end{align}
    where $\pi_f$ is the modular parametrization by an appropriate Shimura curve $X$ and $t \in X$ is a Heegner point for $H_{c(\psi_2)}/K$;

    \item following~\cite[p. 38]{Darmon_Lauder_Rotger_Stark_points}, for $\psi_0 \in \{\psi_g/\psi_g^{\sigma}, \psi_h/\psi_h^{\sigma}\}$, we let
    \begin{align}\label{eqn:Heegner_unit}
        u_{\psi_0} & = \sum\limits_{\sigma \in \Gal(H/K)} \psi_0^{-1}(\sigma) u^{\sigma} \in (\O_H^\times)_L^{\psi_0}
    \end{align}
    where $u$ is an elliptic unit for $H_{c(\psi_0)}/K$ and the unit group is written additively.    
\end{itemize}

\begin{theorem}\label{thm:CM}
    Let $K$ be an imaginary quadratic field and $p$ be a prime which splits in $K$. Then, under Hypothesis~\ref{hyp:A'}, \ref{hyp:B'}, \ref{hyp:C'}, the rank one Elliptic Stark Conjecture~\ref{conj:ES} holds. Explicitly, if $r(E_K, \psi_1) = 0$ and $r(E_K, \psi_2) = 1$, then there is an explicit constant $\lambda \in \sqrt{L^\times}$ such that:
    $$\mathcal L_p^{f, \rm bal}(1,1) = \lambda \cdot  \overbrace{\langle c_{f, \psi_1}, c_{f, \psi_1} \rangle^{1/2}}^{\in \sqrt{L^\times}} \cdot \frac{\log_{E,p}(P_{\psi_2})}{\log_p(u_{\psi_g/\psi_g^\sigma})^{1/2} \log_p(u_{\psi_h/\psi_h^\sigma})^{1/2}},$$
    where $\langle -, - \rangle$ is the height pairing on $S$.
\end{theorem}

\begin{remark}
    The factor $\langle c_{f, \psi_1}, c_{f, \psi_1} \rangle^{1/2} \in L^\times$ could, of course, be combined with $\lambda \in \sqrt{L^\times}$ in the formula. However, we include it here, because it seems to describe the ``arithmetic'' contribution to the special value associated with the character $\psi_1$. It is quite interesting that this factor has a similar form to the other (non-algebraic) factors, and may therefore be relevant to an integral version of the result or to similar results in other settings.
\end{remark}

The proof of this theorem is based on a factorization of the $p$-adic $L$-function corresponding to the factorization
$$L(E, \varrho_{gh}, s) = L(E/K, \psi_1, s) \cdot L(E/K, \psi_2, s).$$

Following \cite[Section 3.2]{Darmon_Lauder_Rotger_Stark_points}, we recall two related $p$-adic $L$-functions.

\subsubsection{The Katz $p$-adic $L$-function for $K$~\cite{Katz-real_analytic}}\label{section: The Katz $p$-adic $L$-function for $K$}

    Let $\Sigma$ be the set of characters of $K$ of conductor dividing a fixed integral ideal $\mathfrak c \subseteq \O_K$. We then define $\Sigma_K = \Sigma_K^{(2)} \cup \Sigma_K^{(2')}$ where:
    \begin{align*}
        \Sigma_K^{(2)} & = \{\psi \in \Sigma \text{ of infinity type $(\kappa_1, \kappa_2)$ with $\kappa_1 \geq 1$, $\kappa_2 \leq 0$} \}, \\
        \Sigma_K^{(2')} & = \{\psi \in \Sigma \text{ of infinity type $(\kappa_1, \kappa_2)$ with $\kappa_1 \leq 0$, $\kappa_2 \geq 1$} \}.
    \end{align*}
    
    Katz defined a $p$-adic $L$-function 
    $$\mathcal L_p(K) \colon \widehat{\Sigma}_K \to \C_p$$ 
    defined on the $p$-adic completion $\widehat{\Sigma}_K$ of $\Sigma_K$ with the interpolation property
    \begin{align}\label{eqn:Katz_interpolation}
        \mathcal L_p(K)(\psi) & = \mathfrak a(\psi) \cdot \mathfrak e(\psi) \cdot \mathfrak f(\psi) \cdot \frac{\Omega_p^{\kappa_1 - \kappa_2}}{\Omega^{\kappa_1 - \kappa_2}} \cdot L_{\mathfrak c}(\psi^{-1}, 0) & \text{for $\psi \in \Sigma^{(2)}_K$},
    \end{align}
    where:
    \begin{itemize}
        \item $\mathfrak a(\psi) = (\kappa_1 - 1)\pi^{-\kappa_2}$, $\mathfrak e(\psi) = (1 - \psi(\mathfrak p) p^{-1})(1 - \psi^{-1}(\overline{\mathfrak p}))$ and $\mathfrak f(\psi) = D_K^{\kappa_2/2} 2^{-\kappa_2}$,
        \item $\Omega \in \C^\times$ and $\Omega_p \in \C_p^\times$ are CM periods attached to $K$, cf.\ \cite[(2-15), (2-17)]{BDP1},
        \item $L_{\mathfrak c}(\psi^{-1}, s)$ is the $\mathfrak c$-depleted Hecke $L$-function associated with $\psi^{-1}$.
    \end{itemize}

    It satisfies a functional equation of the form 
    \begin{align}\label{eqn:Katz_functional_eqn}
        \mathcal L_p(K)(\psi) = \mathcal L_p(K)((\psi^\sigma)^{-1} N_{K/\Q}),
    \end{align}
    relating the values in $\Sigma_K^{(2)}$ to $\Sigma_K^{(2')}$.

    A finite order Hecke character $\psi$ has trivial infinity type and hence lies outside $\Sigma_K$, i.e.\ outside the interpolation range~\eqref{eqn:Katz_interpolation}; however it still belongs to $\widehat{\Sigma}_K$. Katz proves a $p$-adic Kroenecker limit formula~\cite[(47)]{Darmon_Lauder_Rotger_Stark_points}:
    \begin{align}\label{eqn:p-adic_Kronecker}
        \mathcal L_p(\psi) & = \frac{-1}{24 N_{K/\Q} \mathfrak c} \mathfrak e(\psi) \log_p(u_\psi), & \text{for $\psi$ of finite order},
    \end{align}
    for $u_\psi$ defined in~\eqref{eqn:Heegner_unit}.

\begin{figure}[h]
	\begin{tikzpicture}
	\fill[color=darkblue!20]    (1,0) -- (4,0) -- (4,-3) -- (1,-3);
	\fill[color=darkgreen!20]    (0,1) -- (0,4) -- (-3,4) -- (-3,1);
 
 	\draw[->] (-3,0) -- (4,0) node[right] {$\kappa_1$};
	\draw[->] (0, -3) -- (0,4) node[above] {$\kappa_2$};
 
        \draw (0, 1) -- (-3, 1);
        \draw (1,0) -- (1, -3);
 
	\fill (0,0) circle(2.5pt);
        \draw (-1.75, -0.5) node {\begin{tabular}{c} $p$-adic Kronecker \\ formula equation~\eqref{eqn:p-adic_Kronecker} \end{tabular}};
	
	\draw[thick, dashed] (-3,4) -- (4,-3);

	\draw[<->] (0.5, 2) to[bend left] (2, 0.5);
	\draw (1.9,1.3) node[right] {\begin{tabular}{c} functional equation \\ equation~\eqref{eqn:Katz_functional_eqn} \end{tabular}};
        \draw[dotted] (-3,-3) -- (4, 4);

        \draw (2, -2) node {$\Sigma_K^{(2)}$};
        \draw (2.9, -0.5) node {\begin{tabular}{c} Interpolation \\ equation~\eqref{eqn:Katz_interpolation}\end{tabular}};

        \draw (-2, 2) node {$\Sigma_K^{(2')}$};

        \draw (-3, 4) node[above left] {\begin{tabular}{c}
            central \\ critical \\ line
        \end{tabular} };
	\end{tikzpicture}
	
	\caption{The following diagram shows the diagram of infinity types $(\kappa_1, \kappa_2)$ for the characters in $\Sigma$. We indicate interpolation region $\Sigma_K^{(2)}$ for the Katz $p$-adic $L$-function in blue~\eqref{eqn:Katz_interpolation}, the functional equation~\eqref{eqn:Katz_functional_eqn} with the axis of symmetry given by the dotted line, and the point where the $p$-adic Kroenecker limit formula~\eqref{eqn:p-adic_Kronecker} is valid. The dashed line is the central critical line}
\end{figure}
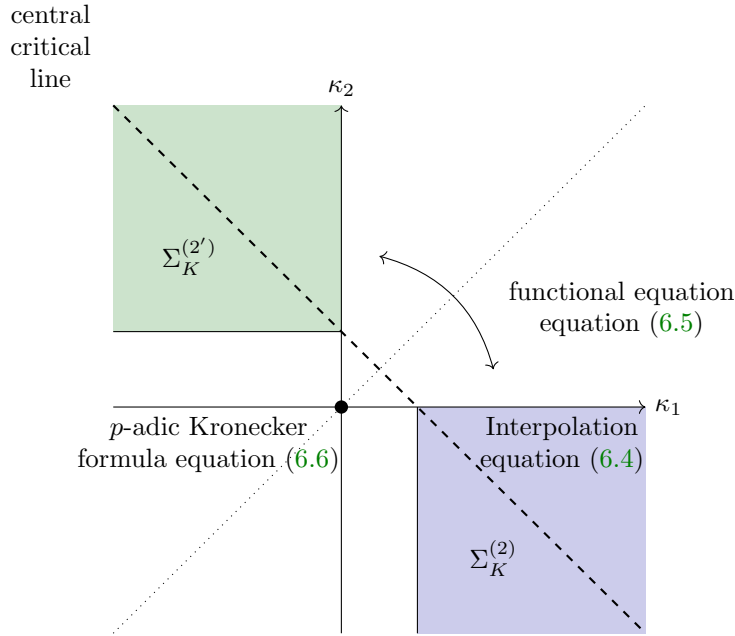

\subsubsection{The BDP $p$-adic Rankin $L$-function for $f$ and $K$}

    For any character $\psi$ of $K$ with infinity type $(\kappa_1, \kappa_2)$, we consider the Rankin--Selberg $L$-function associated with $f$ and $\theta_\psi$:
    $$L(f, \psi, s) = L(f \times \theta_\psi, s - (\kappa_1+ \kappa_2 + 1)/2).$$
    When the conductor of $\psi$ is coprime to the level $N_f$ of $f$, a $p$-adic $L$-function was constructed in~\cite{BDP} and a special value formula outside of the interpolation range was proved.

    Assume Hypothesis~\ref{hyp:weak_Heegner}, i.e. $N_f = N_+ N_-$ and $N_-$ is a square-free product of finite primes. Let $c \geq 1$ be a positive integer relatively prime to $pN^+$ and divisible by $N_-$. Bertolini--Darmon--Prasanna assume that $N_- = 1$, but we prove the analogue of their results in Appendix~\ref{appendix BDP}.
    
    Let $\Sigma_{f, c}$ be the set of characters defined in Definition~\ref{def:Sigma_fc}, so that for $\psi \in \Sigma_{f, c}$, $L(f, \psi, s)$ is self-dual and has $s = 0$ as its central critical point. Note that $\Sigma_{f, c}$ naturally decomposes as
    \begin{align*}
        \Sigma_{f, c} & = \Sigma_{f, c}^{(1)} \cup \Sigma_{f, c}^{(2)} \cup \Sigma_{f, c}^{(2')}, \\
        \Sigma_{f, c}^{(1)} & = \{\psi \text{ of infinity type $(1,1)$} \ | \ \psi \circ \mathrm{Tr} = 1 \}, \\
        \Sigma_{f, c}^{(2)} & = \{\psi \text{ of infinity type $(
        \kappa + 2, -\kappa)$ for $\kappa \geq 1$} \}, \\
        \Sigma_{f, c}^{(2')} & = \{\psi \text{ of infinity type $(
        -\kappa, \kappa + 2)$ for $\kappa \geq 1$} \}.
    \end{align*}
	
	We next summarize Theorem~\ref{thm:BDP_generalization} which is a generalization of the main results of~\cite{BDP}.

    There is a $p$-adic $L$-function
    $$\mathcal L_p(f, K) \colon \widehat{\Sigma}_{f, c} 
    \to \C_p^\times$$
    defined on the $p$-adic completion $\widehat{\Sigma}_{f, c}$ of $\Sigma_{f, c}$ with the interpolation property
    \begin{align}\label{eqn:BDP_interpolation}
        \mathcal L_p(f, K)(\psi) & = \mathfrak a(f, \psi) \cdot \mathfrak e(f, \psi)^2 \cdot \mathfrak f(f, \psi) \cdot \frac{\Omega_p^{4 \kappa + 4}}{\Omega^{4\kappa + 4}} \cdot L(f, \psi^{-1}, 0) & \psi \in \Sigma_{f, c}^{(2)},
    \end{align}
    where
    \begin{itemize}
        \item $\mathfrak a(f, \psi) = \kappa! (\kappa+1)! \pi^{2\kappa + 1}$, $\mathfrak e(f, \psi) = 1 - a_p(f) \psi^{-1}(\overline{\mathfrak p}) + \psi^{-2}(\overline{\mathfrak p}) p$,
        \item $\mathfrak f(f, \psi) = (2/c \sqrt{D_K})^{2 \kappa + 1} \cdot \prod\limits_{q|c_0} \frac{q - \chi_K(q)}{q-1} \cdot \prod\limits_{q|c_-} \frac{q^2}{1-q^2} \cdot \omega(f, \psi)^{-1}$, and $\omega(f, \psi) \in \overline \Q$ is defined in \cite[(5.1.11)]{BDP}.
    \end{itemize}
    
    Next, for a finite order character $\psi$ of $K$, we have that $\psi N_K \in \Sigma_{f, c}^{(1)}$ which is outside of the interpolation range~\eqref{eqn:BDP_interpolation} and Bertolini--Darmon--Prasanna prove the following $p$-adic Gross--Zagier formula:
    \begin{align}\label{eqn:BDP_Gross_Zagier}
        \mathcal L_p(f, K) & = \left( 1 - \frac{a_p(f)}{\psi(\overline{\mathfrak p}) p} + \frac{1}{\psi^2(\overline{\mathfrak p}) p} \right)^2 \cdot \log_p(P_{\psi})^2,
    \end{align}
    for $P_\psi$ defined in~\eqref{eqn:Heegner_pt}.

    \begin{figure}[h]
	\begin{tikzpicture}
	\fill[color=darkblue!20]    (2,0) -- (5,0) -- (5,-3) -- (2,-3);
	\fill[color=darkgreen!20]    (0,2) -- (0,5) -- (-3,5) -- (-3,2);
 
 	\draw[->] (-3,0) -- (5,0) node[right] {$\kappa_1$};
	\draw[->] (0, -3) -- (0,5) node[above] {$\kappa_2$};
 
        \draw (0, 2) -- (-3, 2);
        \draw (2,0) -- (2, -3);
 
	\fill (1,1) circle(2.5pt);
        \draw (-1.5, 1) node {\begin{tabular}{c} $p$-adic \\ Gross--Zagier \\ formula~\eqref{eqn:BDP_Gross_Zagier} \end{tabular}};
        \draw[->] (-0.2,1) -- (0.8,1);
	
	\draw[thick, dashed] (-3,5) -- (5,-3);

	\draw[<->] (0.5, 3) to[bend left] (3, 0.5);
 
	\draw (3,1.3) node[right] {functional equation};
        \draw[dotted] (-3,-3) -- (5, 5);

        \draw (3, -2) node {$\Sigma_{f, c}^{(2)}$};
        \draw (3.9, -0.5) node {\begin{tabular}{c} Interpolation \\ equation~\eqref{eqn:BDP_interpolation}\end{tabular}};

        \draw (-2, 3) node {$\Sigma_{f, c}^{(2')}$};

        \draw (-3, 5) node[above left] {\begin{tabular}{c}
             central \\ critical \\ line
        \end{tabular} };
	\end{tikzpicture}
	
	\caption{The following diagram shows the diagram of infinity types $(\kappa_1, \kappa_2)$ for the character in $\Sigma$. We indicate interpolation region $\Sigma_{f, c}^{(2)}$ for the BDP $p$-adic $L$-function in blue~\eqref{eqn:Katz_interpolation}, the expected functional equation with the axis of symmetry given by the dotted line, and the point where the $p$-adic Gross--Zagier formula~\eqref{eqn:BDP_Gross_Zagier} is valid.}
\end{figure}
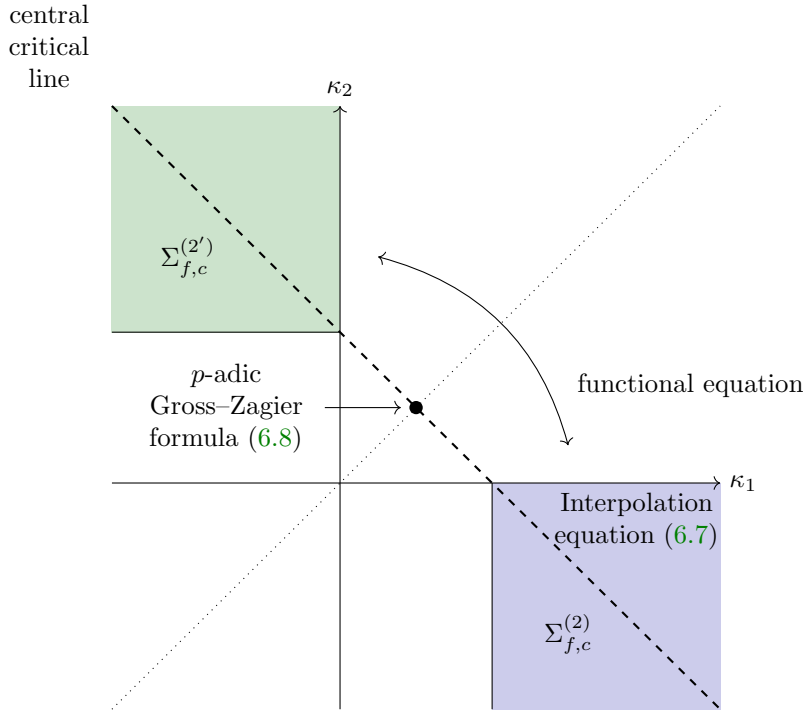

\subsubsection{Gross' formula}

The following formula was originally discovered by Gross~\cite[Proposition 11.2]{Gross87} in a special case and proved in full generality in~\cite[Theorem 1.10]{CaiShuTian}:
\begin{equation}\label{eqn:Gross}
    \frac{L^{(\Sigma)}(f, \psi, 1)}{\pi^2 \langle f, f \rangle} =
    2^{3} \frac{\langle c_{f, \psi}, c_{f, \psi} \rangle}{[\O_c^\times : \Z^\times]^2 c \sqrt{D_K}},
\end{equation}
for $c_{f, \chi}$ defined in~\eqref{eqn:Gross_point} and $\Sigma = \{ v \text{ places dividing } (N, c) \}$. In particular, $\frac{\sqrt{D_K} L(f, \psi, 1)}{\pi^2 \langle f, f \rangle}$ lies in the number field $\Q(f, \psi)$.

\subsubsection{Explicit CM Hida families}

We follow~\cite[Section 3.2]{Darmon_Lauder_Rotger_Stark_points} to recall Hida's construction of an explicit CM Hida family $\mathbf F$ specializing in weight one to $F_\alpha \in \{g_\alpha, h_\alpha\}$. Recall that for a finite order character $\psi$ of $K$ such that $\psi^\sigma \neq \psi$, we have an associated cusp form:
$$\theta_\psi \in S_1(D_K N_{K/\Q}(\mathfrak c), \chi).$$
We fix a character $\lambda$ of infinity type $(0,1)$ and conductor $\overline{\mathfrak p}$ and valued in $\Q(\lambda)$. Let $\Q_p(\lambda)$ be the $p$-adic completion of $\Q(\lambda)$ determined by a fixed embedding $\Q(\lambda) \hookrightarrow \overline{\Q_p}$. We have that
$$\O_{\Q_p(\lambda)}^\times = \mu \times W$$
where $\mu$ is finite and $W$ is a free $\Z_p$-module, and we write
$$\langle - \rangle \colon \O_{\Q_p(\lambda)}^\times \to W$$
for the projection. 

For $\psi \in \{\psi_g, \psi_h\}$ and every integer $k \geq 1$, define
\begin{equation}
    \psi_{k-1}^{(p)} = \psi \langle \lambda \rangle^{k-1}
\end{equation}
which is independent on the choice of $\lambda$, and let
\begin{equation}
    \psi_{k-1}(\mathfrak q) = \begin{cases}
        \psi_{k-1}^{(p)}(\mathfrak q) & \mathfrak q \neq \overline{\mathfrak p}, \\
        \chi_F(p) p^{k-1}/\psi_{k-1}^{(p)}(\mathfrak p) & \mathfrak q = \overline{\mathfrak p}.
    \end{cases}
\end{equation}
Then the ordinary $p$-stabilization of the associated theta series:
\begin{equation}
    F_k = \theta_{\psi_{k-1}} \in S_k(D_K N_{K/\Q} (\mathfrak c(\psi_F)), \chi_F)
\end{equation}
is the weight $k$ specialization of the Hida family $\mathcal F$ and, by definition, $F_1 = F$. 

\subsubsection{Factorization of the $p$-adic $L$-function}\label{section: Factorization of the $p$-adic $L$-function}

We are now ready prove the factorization of the triple product balanced $p$-adic $L$-function. We start by recalling the factorization of the classical $L$-functions, then analyze the periods, and finally deal with the auxillary factors. The resulting statement is Theorem~\ref{thm:CM_factorization} below.

Note that the norm $N_K$ from $K$ to $\Q$ can be regarded as a Hecke character of $K$ of infinity type $(1,1)$. Since $\langle \lambda \rangle$ has infinity type $(0,1)$, $\langle \lambda \lambda^\sigma \rangle$ has infinity type $(1,1)$ and we may identify it with $N_K$. 

Note that:
$$V_{g_\ell} \otimes V_{h_\ell} = V_{\psi_{g, \ell-1}^{(p)}} \otimes V_{\psi_{h, \ell - 1}^{(p)}} \iso V_{\psi_{g, \ell-1}^{(p)} \psi_{h, \ell-1}^{(p)}} \oplus V_{\psi_{g, \ell-1}^{(p)} \psi_{h, \ell-1}^{(p), \sigma}}$$
and
\begin{align*}
    \psi_{g, \ell-1}^{(p)} \cdot \psi_{h, \ell-1}^{(p)} & = \psi_g \cdot \psi_h \cdot \langle \lambda \rangle^{2\ell - 2}  = \psi_2 \langle \lambda \rangle^{2\ell -2} \\
    \psi_{g, \ell-1}^{(p)} \cdot \psi_{h, \ell-1}^{(p), \sigma} & = \psi_g \cdot \psi_h^\sigma \cdot \langle \lambda \lambda^\sigma \rangle^{\ell-1} = \psi_g \psi_h^\sigma  N_K^{\ell - 1} = \psi_1 N_K^{\ell - 1}. 
\end{align*}

Via the Artin formalism, the above results in a factorization of the triple product $L$-function:
\begin{align*}
    L(f \times g_\ell \times h_\ell, s) & = L(f_K \times \psi_2 \langle \lambda \rangle^{2\ell-2}, s) \cdot L(f_K \times \psi_1, s - (\ell - 1)),
\end{align*}
where $\theta_{\psi_1 \langle \lambda \rangle^{2k-2}}$ has weight $2k - 1$ and $\theta_{\psi_2}$ has weight 1.

By Corollary~\ref{cor:interpolation}, we know that $\mathcal L_p^{\bal, f}(\ell, \ell)$ is related to:
\begin{align*}
    L(f \times g_\ell \times h_\ell, \ell) & = L(f_K \times \psi_2 \langle \lambda \rangle^{2\ell-2}, \ell) \cdot L(f_K \times \psi_1, 1) \\
    & = L(f_K \times \psi_2 \langle \lambda \rangle^{2\ell-2} N_K^{-\ell}, 0) \cdot L(f_K \times \psi_1, 1).
\end{align*}
More specifically, dividing by the period $\langle f, f \rangle \langle g_\ell, g_\ell \rangle \langle h_\ell, h_\ell \rangle$, we have:
\begin{align}\label{eqn:fact_initial}
    \frac{L(f \times g_\ell \times h_\ell, \ell)}{\langle f, f \rangle \langle g_\ell, g_\ell \rangle \langle h_\ell, h_\ell \rangle} & = 
    \frac{1}{\langle g_\ell, g_\ell \rangle \langle h_\ell, h_\ell \rangle} \cdot L(f_K \times \psi_2 \langle \lambda \rangle^{2\ell-2} N_K^{-\ell}, 0) \cdot \frac{L(f_K \times \psi_1, 1)}{\langle f, f \rangle}.
\end{align}
Note that 
$$\Psi_{gh}(\ell) = \psi_2^{-1} \langle 
\lambda \rangle^{-2 \ell + 2} N_K^{\ell} \in \Sigma^{(2)}_K$$
has infinity type $(\ell , -\ell+2)$, so by equation~\eqref{eqn:BDP_interpolation}:
\begin{equation}\label{eqn:BDP_interpolation_ourcase}
    L(f_K \times \Psi_{gh}(\ell)^{-1}, 0) =  \mathfrak a(\Psi_{gh}(\ell))^{-1} \cdot \mathfrak e(f, \Psi_{gh}(\ell))^{-1} \cdot \mathfrak f(\Psi_{gh}(\ell))^{-1} \cdot \frac{\Omega^{4\ell - 4}}{\Omega_p^{4\ell - 4}} \cdot \mathcal L_p(f,K)(\Psi_{gh}(\ell)).
\end{equation}
Next, define for $F \in \{g, h\}$
$$\Psi_F(\ell) = \psi_{F, \ell-1}^{-2} \chi_F N_K^\ell$$
and recall fact from \cite[Lemma 3.8]{Darmon_Lauder_Rotger_Stark_points}: for some $K$-admissible functions $\mathfrak f_2(k)$, $\mathfrak f_3(k)$
\begin{align}\label{eqn:Katz_and_Petersson}
    \mathcal L_p(K)(\Psi_F(\ell)) & = \frac{\mathfrak e(\Psi_F(\ell)) \mathfrak f(\Psi_F(\ell))}{\mathfrak f_2(\ell
    ) \cdot \mathfrak f_3(\ell) } \cdot \langle F_
    \ell, F_\ell \rangle \cdot \left( \frac{\pi \Omega_p}{\Omega} \right)^{2 \ell - 2}.
\end{align}
Thus equation~\eqref{eqn:fact_initial} becomes
\begin{align*}
    \frac{\Lambda(f \times g_\ell \times h_\ell, \ell)}{\langle f, f \rangle \langle g_\ell, g_\ell \rangle \langle h_\ell, h_\ell \rangle} \cdot \mathcal L_p(K)(\psi_g(\ell)) \cdot \mathcal L_p(K)(\psi_h(\ell)) & = \frac{\mathfrak e(\Psi_g(\ell)) \mathfrak f(\Psi_g(\ell)) \mathfrak e(\Psi_h(\ell)) \mathfrak f(\Psi_h(\ell))}{\mathfrak f_2(\ell
    )^2  \mathfrak f_3(\ell)^2  \mathfrak e(f, \Psi_{gh}(\ell)) \mathfrak f(\Psi_{gh}(\ell))} \\
    & \hspace{2cm} \cdot \mathcal L_p(f, K)(\Psi_{gh}(\ell)) \cdot \frac{L(f_K \times \psi_1, 1)}{\pi^2 \langle f, f \rangle},
\end{align*}
where we have canceled the factor $\mathfrak a(\Psi_{gh}(\ell))$ with the other powers of $\pi$ and the $\Gamma$-factors for the triple product $L$-function.

Next, using Corollary~\ref{cor:interpolation}, we get that:
$$\mathcal L_p^{\bal, f}(\ell, \ell)^2 \cdot \mathcal L_p(K)(\psi_g(\ell)) \cdot \mathcal L_p(K)(\psi_h(\ell)) =  \mathcal E(\ell) \cdot \mathcal F(\ell) \cdot \mathcal L_p(f, K)(\Psi_{gh}(\ell)) \cdot \frac{L(f_K \times \psi_1, 1)}{\pi^2 \langle f, f \rangle},$$
where
\begin{align}
    \mathcal E(\ell) & = \frac{\mathcal E_p^{\bal}(\mathbf V_\kappa) \cdot  \mathcal E_p(f_\alpha, \Ad) \cdot \mathfrak e(\Psi_g(\ell)) \cdot \mathfrak e(\Psi_h(\ell))}{\mathfrak e(f, \Psi_{gh}(\ell))}, \\
    \mathcal F(\ell) & = 2^{3-2\ell} \frac{\mathfrak f(\Psi_g(\ell)) \mathfrak f(\Psi_h(\ell))}{\mathfrak f_2(\ell
    )^2  \mathfrak f_3(\ell)^2 \mathfrak f(\Psi_{gh}(\ell))}. \label{eqn:F(l)}
\end{align}

\begin{lemma}
    For any $\ell \geq 0$, 
    $$\mathcal E(\ell) =  (1 - \beta_f \psi_1(\mathfrak p) p^{-1})(1 - \beta_f \psi_1(\overline{\mathfrak p}) p^{-1}).$$
\end{lemma}
\begin{proof}
    We recall that
    $$L_p(\theta_{\psi_{\ell-1}}, s) = L_{\mathfrak p}(\psi_{\ell-1}, s) \cdot L_{\overline{\mathfrak p}}(\psi_{\ell-1}, s)$$
    and hence for $F \in \{g, h \}$
    \begin{align*}
       \alpha_F & = \psi_F(\mathfrak p) \langle \lambda(\mathfrak p) \rangle^{\ell-1} \\
       \beta_F & = \chi_F(p) p^{\ell-1}/\alpha_F = p^{\ell-1} \psi_F(\overline{\mathfrak p}) \langle \lambda(\mathfrak p ) \rangle^{1- \ell}.
    \end{align*}

    We first check that for $F \in \{g, h\}$:
    \begin{align}\label{eqn:adjoint_factor_CM_case}
        \mathcal E_p(F_{\alpha, \ell}, \Ad) & = \mathfrak e(\Psi_F(\ell)),
    \end{align}
    Indeed, according to \cite[p.\ 416]{Hsieh2021}
    \begin{align*}
        \mathcal E_p(F_{\alpha, \ell}, \Ad) & = (1 - \alpha_F^{-2} \chi_F(p) p^{\ell-1})  (1 - \alpha_F^{-2} \chi_F(p) p^{\ell-2}) \\
        & = (1 - \psi_F(\mathfrak p)^{-2} \langle \lambda(\mathfrak p) \rangle^{-2\ell+2} \chi_F(p) p^{\ell-1})  (1 - \psi_F(\mathfrak p)^{-2} \langle \lambda(\mathfrak p) \rangle^{-2\ell+2} \chi_F(p) p^{\ell-2}),
    \end{align*}
    while
    \begin{align*}
        \mathfrak e(\Psi_F(\ell)) & = (1 - \Psi_F(\ell)(\mathfrak p) p^{-1}) (1 - \Psi_F(\ell)^{-1}(\overline{\mathfrak p})) \\
        & = (1 - \psi_{F, \ell-1}^{-2}(\mathfrak p) \chi_F(p) p^\ell p^{-1}) (1 -  (\psi_{F, \ell-1}^{-2}(\overline{\mathfrak p}) \chi_F(p) p^\ell)^{-1} ) \\
        & = (1 - \psi_F(\mathfrak p)^{-2} \langle \lambda(\mathfrak p )\rangle^{-2\ell + 2} \chi_F(p) p^{\ell - 1})(1 - (\chi_F(p) p^{\ell-1}/(\psi_F(\mathfrak p) \langle \lambda(\mathfrak p) \rangle)^{\ell - 1} )^{2} \chi_F(p)^{-1} p^{-\ell}  ) \\
        & = (1 - \psi_F(\mathfrak p)^{-2} \langle \lambda(\mathfrak p )\rangle^{-2\ell + 2} \chi_F(p) p^{\ell - 1})(1 - \psi_F(\mathfrak p)^{-2} \langle \lambda(\mathfrak p) \rangle^{-2\ell + 2} \chi_F(p) p^{\ell - 2}  ),
    \end{align*}
    verifying equation~\eqref{eqn:adjoint_factor_CM_case}.
    
    Therefore, we have:
    \begin{align*}
        \mathcal E_p^{\bal}(V_\kappa) \cdot \mathcal E_p(f_\alpha, \Ad) \cdot \mathfrak e(\Psi_g(\ell)) \mathfrak e(\Psi_h(\ell)) & = (1 - \alpha_f \beta_g \beta_h p^{-\ell})(1 - \beta_f \alpha_g \beta_h p^{-\ell})(1 - \beta_f \beta_g \alpha_h p^{-\ell}) (1 - \beta_f \beta_g \beta_h p^{-\ell})
        \end{align*}
    $$=
        (1 - \alpha_f (\psi_g \psi_h)(\overline{\mathfrak p}) \langle \lambda (\mathfrak p) \rangle^{2 - 2\ell} p^{\ell - 2})
        (1 - \beta_f (\psi_g \psi_h^{\sigma})(\mathfrak p) p^{-1})
        (1 - \beta_f (\psi_g \psi_h^\sigma)(\overline{\mathfrak p}) p^{-1}) 
        (1 - \beta_f (\psi_g \psi_h)(\overline{\mathfrak p}) \langle \lambda (\mathfrak p) \rangle^{2 - 2\ell} p^{\ell - 2} )$$
    It remains to observe that:
    \begin{align*}
        \mathfrak e(f, \Psi_{gh}(\ell)) & = (1 - \alpha_f \Psi_{gh}^{-1}(\overline{\mathfrak p}))(1 - \beta_f \Psi_{gh}^{-1}(\overline{\mathfrak p}) )  \\
        & = (1 - \alpha_f \psi_2(\overline{\mathfrak p}) \langle \lambda(\overline{\mathfrak p}) \rangle^{2\ell - 2} p^{-\ell})(1 - \beta_f \psi_2(\overline{\mathfrak p}) \langle \lambda(\overline{\mathfrak p}) \rangle^{2\ell - 2} p^{-\ell} ) \\
        & = (1 - \alpha_f (\psi_g \psi_h)(\overline{\mathfrak p}) \langle \lambda(\overline{\mathfrak p}) \rangle^{2\ell - 2} p^{-\ell})(1 - \beta_f (\psi_g \psi_h)(\overline{\mathfrak p}) \langle \lambda(\overline{\mathfrak p}) \rangle^{2\ell - 2} p^{-\ell} ),
    \end{align*}
    and use $\mathfrak p\overline{\mathfrak p} = (p)$.
\end{proof}

\begin{remark}
    Note that:
    \begin{align*}
    L_p(f_K \times \psi_1, s)^{-1} & = (1 - \alpha_f \psi_1(\overline{\mathfrak p}) p^{-s})(1 - \beta_f \psi_1(\overline{\mathfrak p}) p^{-s}) (1 - \alpha_f \psi_1({\mathfrak p})p^{-s})(1 - \beta_f \psi_1({\mathfrak p}) p^{-s} ) 
    \end{align*}
    and hence the factor $\mathcal E(\ell)$ is part of the local $L$-factor at $p$ of $L(f_K \times \psi_1, 1)$. 
\end{remark}

\begin{lemma}
    The function $\mathcal F(\ell)$ is $K(\psi_g, \psi_h)$-admissible according to \cite[Definition\ 3.5]{Darmon_Lauder_Rotger_Stark_points}, i.e.\ it extends to an element of $\Frac \O(\mathcal U)$  and $\mathcal F(1) \in K(\psi_g, \psi_h)^\times$.
\end{lemma}
\begin{proof}
    This follows from the same argument as the proof of~\cite[Lemma 3.6]{Darmon_Lauder_Rotger_Stark_points}.
\end{proof}

Altogether, this gives the following result.

\begin{theorem}\label{thm:CM_factorization}
    There is a factorization:
    $$\mathcal L_p^{\bal, f}(\ell, \ell)^2 \cdot \mathcal L_p(K)(\psi_g(\ell))  \cdot \mathcal L_p(K)(\psi_h(\ell)) = \mathcal F(\ell) \cdot \mathcal L_p(f, K)(\Psi_{gh}(\ell)) \cdot \frac{L(f_K \times \psi_1, 1)}{\langle f, f \rangle} \cdot \mathcal E(\ell) .$$
\end{theorem}

Then Theorem~\ref{thm:CM} follows by evaluating the above factorization at $\ell = 1$ and using equations~\eqref{eqn:p-adic_Kronecker},~\eqref{eqn:BDP_Gross_Zagier}, and~\eqref{eqn:Gross}.

\medskip

\appendix

\section{A generalization of results of Bertolini--Darmon--Prasanna}\label{appendix BDP}

\subsection{Statement of the results}

Let $f$ be a weight two cuspidal eigenform for $\Gamma_0(N)$, which we assume to have trivial Nebentypus for simplicity. Let $K$ be an imaginary quadratic field of discriminant $-D_K$ and $p$ be a prime which splits in $K$. For any Hecke character $\psi$ of $K$ with infinity type $(\kappa_1, \kappa_2)$, we consider the Rankin--Selberg $L$-function associated with $f$ and $\theta_\psi$:
$$L(f, \psi, s) = L(f \times \theta_\psi, s - (\kappa_1+ \kappa_2 + 1)/2).$$
We will assume the condition:
\begin{equation}\label{eqn:central_critical}
    \psi|_{\A^\times} = | \cdot |^2,
\end{equation}
where $|\cdot|$ is the norm character of $\Q$, which ensures the $L$-function $L(f, \psi, s)$ is self-dual with central critical value at $s = 0$.

Under the Heegner hypothesis (if $q|N$, then $q$ is either split of ramified in $K$ and if $q^2|N$, then $q$ is split in $K$) and the assumption that the conductor of $\psi$ is coprime to the level $N$ of~$f$, Bertolini, Darmon, and Prasanna~\cite{BDP} constructed a $p$-adic $L$-function interpolating the critical values $L(f, \psi^{-1}, 0)$ when $\kappa_1 \geq 1$ and $\kappa_2 \leq 0$, and proved a special value formula for $\kappa_1 = \kappa_2 = 0$. 

The goal of this appendix is to reprove their result, allowing a squarefree product of inert primes to divide $N$.
\begin{hypothesis}\label{hyp:weak_Heegner}
    If $q^2|N$, then $q$ is split in $K$.
\end{hypothesis}

Instead, we make the following assumption on the conductor of $\psi$ under which $\epsilon_q(f, \psi) = +1$ at the inert primes $q$ dividing $N$.

\begin{hypothesis}\label{hyp:cond_psi}
    Let $N_0$ be the product of primes dividing $N$ which are inert in $K$. Then the conductor of $\psi$ is divisible exactly by $N_0$ and coprime to $N/N_0$.
\end{hypothesis}

\begin{remark}\label{rmk: weak heegner hyp}
    The {\em weak Heegner hypothesis} allows a product $N_-$ of an even number of inert primes to divide $N$. Assuming that the conductor of $\psi$ is coprime to $N$ --- the opposite to Hypothesis~\ref{hyp:cond_psi} --- we have that $\epsilon_q(f, \psi) = -1$ for all primes $q | N_-$. The integral representation for the Rankin--Selberg $L$-function is then on a Shimura curve associated with a quaternion algebra of discriminant $N_-$. A generalization of the results of Bertolini--Darmon--Prasanna to this setting was obtained by Brooks~\cite{Brooks}.

    As mentioned above, our Hypothesis~\ref{hyp:cond_psi} instead forces $\epsilon_q(f, \psi) = +1$ at all primes $q | N_0$, and in fact we will assume that $\epsilon_q(f, \psi) = +1$ for all $q$, so that the integral representation for $L(f, \psi, s)$ is still on the modular curve. This is why we do not require the parity assumption.

    Finally, combining the results of this appendix with the results of Brooks, one could presumably allow $N_0 \cdot N_- | N$ where $N_-$ is a product of an even number of inert primes, and the conductor of $\psi$ is divisible exactly by $N_0$, and coprime to $N/N_0$. We decided not to pursue this generality here.
\end{remark}

Under Hypothesis~\ref{hyp:weak_Heegner}, there is an ideal $\mathfrak N$ of $\mathcal O_K$ of norm $N N_0$; we fix such an ideal. Given an integer $c \geq 1$ divisible exactly by $N_0$ and coprime to $(N/N_0)D_K$, we consider an order $\mathcal O_c$ of $\O_K$ of conductor $c$. Setting $\mathfrak N_c = \mathfrak N \cap \mathcal O_c$, we have:
$$\mathcal O_c/\mathfrak N_c \iso \Z/N \Z.$$
Indeed, for $q | N$ coprime to $c$, if $\mathfrak q | q$ then 
$$\O_{c, \mathfrak q}/\mathfrak N_{c, \mathfrak q} \iso \O_{K, \mathfrak q}/\mathfrak N_{\mathfrak q} \iso \Z/(q^{v_q(N)}) \Z.$$ 
Moreover, for $q | N_0$, $\O_{c, q} = \Z + q \O_{K_q}$ and:
$$\O_{c, q}/\mathfrak N_{c, q} = (\Z + q \O_{K_q})/q \O_{K_q} \iso \Z/q\Z.$$
Note that for $q|N_0$, $1 + q\O_{K_q} \subseteq \O_{c, q}^\times \subseteq \O_{K_q}^\times$ and
\begin{center}
\begin{tikzcd}
    \O_{c, q}^\times/(1 + q \O_{K_q}) \ar[r, hook] \ar[d, "\iso"] & \O_{K_q}^\times/(1 + q \O_{K_q}) \ar[r, two heads] \ar[d, "\iso"] & \O_{K_q}^\times/ \O_{c, q}^\times \ar[d, "\iso"] \\
    \F_q^\times \ar[r] & \F_{q^2}^\times \ar[r] &  \F_{q^2}^\times/\F_q^\times
\end{tikzcd}
\end{center}
In particular, a character of $\O_{K_q}^\times$ which is trivial on $\O_{c,q}^\times$ has conductor at most 1 and is trivial on $\F_{q}^\times \subseteq \F_{q^2}^\times$.

\begin{definition}\label{def:Sigma_fc}
    Let $\Sigma_{f, c}$ denote the set of Hecke characters $\psi$ of $K$ such that: 
    \begin{enumerate}
        \item $\psi|_{\widehat \O_c^\times} = 1$, but for $q|N_0$, $\psi_q|_{\O_{K, q}^\times} \neq 1$,
        \item the infinity type $(\kappa_1, \kappa_2)$ of $\psi$ satisfies $\kappa_1 + \kappa_2  = 2$,
        \item $\psi \psi^\sigma = |\cdot|^2$,
        \item $\epsilon_q(f, \psi^{-1}) = +1$ for all finite primes $q$. 
    \end{enumerate}
\end{definition}

By the above discussion, if $\psi \in \Sigma_{f, c}$, then $\psi_q$ for $q|N_0$ has conductor $1$ and $\psi|_{\Z_q^\times} = 1$. In particular, assumption (1) in the definition implies Hypothesis~\ref{hyp:cond_psi}. If $N_0 = 1$, this agrees with $\Sigma_{f, cc}$ in \cite{BDP}. The final condition is automatic except possibly for the primes in the set:
$$S(f) = \{ q \ | \ \text{$q$ divides $N$ and $D_K$}\}.$$
For example, if $N$ is coprime to $D_K$, $S(f) = \emptyset$.

For $\psi \in \Sigma_{f,c}$, $L(f, \psi, s)$ is self-dual and has $s = 0$ as its central critical point. Moreover, $\Sigma_{f, c}$ naturally decomposes as
    \begin{align*}
        \Sigma_{f, c} & = \Sigma_{f, c}^{(1)} \cup \Sigma_{f, c}^{(2)} \cup \Sigma_{f, c}^{(2')}, \\
        \Sigma_{f, c}^{(1)} & = \{\psi \text{ of infinity type $(1,1)$} \}, \\
        \Sigma_{f, c}^{(2)} & = \{\psi \text{ of infinity type $(
        \kappa + 2, -\kappa)$ for $\kappa \geq 1$} \}, \\
        \Sigma_{f, c}^{(2')} & = \{\psi \text{ of infinity type $(
        -\kappa, \kappa + 2)$ for $\kappa \geq 1$} \}.
    \end{align*}

\begin{theorem}\label{thm:BDP_generalization}
	\leavevmode
	\begin{enumerate}
		\item  There is a $p$-adic $L$-function $\mathcal L_p(f, K) \colon \widehat \Sigma_{f, c} \to \C_p^\times$ defined on the $p$-adic completion $\widehat \Sigma_{f, c}$ of $\Sigma_{f, c}$ with the following interpolation property:
		 \begin{align*}
			\frac{\mathcal L_p(f, K)(\psi)}{\Omega_p^{4 \kappa + 4}} & = \mathfrak a(f, \psi) \cdot \mathfrak e(f, \psi)^2 \cdot \mathfrak f(f, \psi) \cdot \frac{L(f, \psi^{-1}, 0)}{\Omega^{4\kappa + 4}} & \psi \in \Sigma_{f, c}^{(2)},
		\end{align*}
		where
		\begin{itemize}
			\item $\mathfrak a(f, \psi) = \kappa! (\kappa+1)! \pi^{2\kappa + 1}$, $\mathfrak e(f, \psi) = 1 - a_p(f) \psi^{-1}(\overline{\mathfrak p}) + \psi^{-2}(\overline{\mathfrak p}) p$,
			\item $\mathfrak f(f, \psi) = (2/c \sqrt{D_K})^{2 \kappa + 1} \prod\limits_{q|c/N_0} \frac{q - \chi_K(q)}{q-1} \prod\limits_{q|N_0} \frac{q^2 - 1}{q^2} \cdot \omega(f, \psi)^{-1}$, and $\omega(f, \psi) \in \overline \Q$ is defined in \cite[(5.1.11)]{BDP}.
		\end{itemize}
	
		\item For a finite order character $\psi$ of $K$, $\psi N_K \in \Sigma_{f, c}^{(1)}$ which is outside of the interpolation range~\eqref{eqn:BDP_interpolation}, there is a $p$-adic Gross--Zagier formula:
		\begin{align*}
			\mathcal L_p(f, K) & = \left( 1 - \frac{a_p(f)}{\psi(\overline{\mathfrak p}) p} + \frac{1}{\psi^2(\overline{\mathfrak p}) p} \right)^2 \cdot \log_p(P_{\psi})^2,
		\end{align*}
		for $P_\psi$ defined by:
		\begin{align*}
			P_{\psi} & = \pi_f\left(\sum\limits_{\sigma \in \Gal(H_c/K)} \psi^{-1}(\sigma) t^\sigma\right) \in E(H_c)_L^{\psi_2} 
		\end{align*}
		for a Heegner point $t$ for $H_c/K$.
	\end{enumerate}
	  
\end{theorem}

\begin{remark}
	We remark that Liu--Zhang--Zhang~\cite{LZZ} proved a formula similar to part (2) of the above theorem for any abelian variety parameterized by a Shimura curve over a totally real field. In particular, in the case of Shimura curves over $\Q$, they removed all the ramification hypothesis in~\cite{BDP, Brooks}.
	
	However, their construction of the $p$-adic $L$-function does not have an interpolation property as explicit as~(1). More specifically, (1) can be interpreted as an equality in $\overline \Q$ (or even an explicit finite extension of $\Q$), having fixed embeddings $\overline \Q \hookrightarrow \C, \C_p$. Crucially, the CM periods $\Omega$ and $\Omega_p$ make both sides algebraic.
	
	The interpolation property of Liu--Zhang--Zhang is not an equality of algebraic numbers and relies on an identification $\C \iso \C_p$. In particular, the CM periods do not show up in their interpolation property, and we have not find a direct way to deduce the above result from their work. 
\end{remark}

\subsection{An explicit Waldspurger formula}

The crux of the proof of Theorem~\ref{thm:BDP_generalization} is to give an explicit version of Waldspurger's formula, generalizing~\cite[Theorem 4.6]{BDP}. 

\begin{theorem}\label{thm:explicit_Waldspurger}
    Let $f$ be a normalized eigenform in $S_2(N)$ and let $\psi \in \Sigma_{f, c}^{(2)}$ be a Hecke character of~$K$ of infinity type $(2+j, -j)$. Suppose also that $c$ and $D_K$ are odd, and let $w_K$ denote the number of roots of unity in $K$. Then:
    $$
    C(f, \psi, c) \cdot L(f, \psi^{-1}, 0) = \left|  \sum_{[\mathfrak a] \in {\rm Pic}(\O_c)} \chi^{-1}(\mathfrak a) N\mathfrak a^{-j} \cdot (\delta_k^j f)(\mathfrak a^{-1}, t_{\mathfrak a})  \right|,
    $$
    where the representatives of the ideal classes in ${\rm Pic}(\O_c)$ are chosen to be coprime to the $\mathfrak N_c$ and the constant $C(f, \psi, c) \in \C$ is given by:
    $$C(f, \chi, c) = \frac{1}{4} \pi^{k + 2j - 1} \Gamma(j+1) \Gamma(k+j) w_K |D_K|^{1/2} \cdot c \vol(\O_c)^{-\ell} \cdot 2^{\# S_f} \cdot \prod_{q|c/N_0} \frac{q-\chi_K(q)}{q-1} \cdot \prod_{q|N_0} \frac{q^2 - 1}{q^2}.$$
\end{theorem}
\begin{proof}
	The proof in \cite[Section 4]{BDP} can be applied verbatim to our case, until the local calculations in Section 4.6. The extra local integral is then computed in the next section, resulting in Proposition~\ref{prop:local_int_BDP}.
\end{proof}

\subsection{The key computation}

The assumption in~\cite{BDP} that the conductor of $\psi$ is coprime to the level $N$ of $f$ is used in the explicit computations of the local zeta integrals in Section 4.6. We use the same notation as loc. cit.\ but consider the following extra setting:
\begin{itemize}
    \item let $\mathbb Q_{q^2} = \mathbb Q_q[\varpi]$ where $\varpi^2 \in \Q_{q}^\times$ is a unit is the unramified quadratic extension of $\mathbb Q_q$, and let $\sigma$ be the non-trivial automorphism of $\mathbb Q_{q^2}$ over $\mathbb Q_q$,
    \item $\psi$ is an anticyclomotic character of $\mathbb Q_{q^2}^\times$  of conductor $r \geq 1$ such that $\psi \neq \psi^\sigma$ and $\psi|_{\Q_q^\times} = 1$,  (here, anticyclotomic means that $\psi \psi^\sigma = 1$; in particular, note that $\psi^2 \neq 1$), 
    \item $\pi$ is the Steinberg representation $\St$ of $\GL_2(\mathbb Q_q)$.
\end{itemize}
Analogously to p. 1118 of loc.\ cit.\ we consider:
$$J(\varsigma, \vartheta) = \int\limits_{ N_{\Q_{q^2}/\Q_q}(\Q_{q^2}^\times)} W_F(d(a)) W_{\Theta, \varsigma}(d(a)) \Phi_\vartheta^s(d(a))|a|^{-1} d^\times a.$$
As written, the equation in~\cite{BDP} involves the integral of $\Q_q^\times$, but tracing back the reference~\cite[Section 3]{Prasanna_Integrality} reveals this is the correct generalization at inert places.

We also recall from Proposition\ 4.24 of loc. cit.\ that:
$$W_{\Theta, \varsigma}(d(a)) = \int_{\Q_{q^2}^{(1)}} \varsigma( h^{-1} (h')^\sigma ) \psi(h h') \, dh $$
for any $h'$ such that $N(h') = a$, where
$$\Q_{q^2}^{(1)} = \{x \in \Q_{q^2} \ | \ N(x) = 1\} \subseteq \Z_{p^2}^\times$$
and the Haar measure is normalized so that $\vol( \Q_{q^2}^{(1)} ) = 1$.

We recall that the group of norm one elements in the unramified quadratic extension of $\Q_p$ has the following description:
\begin{align}
    \Q_{q^2}^{(1)} & = \left\{ \frac{t}{t^\sigma} \ \middle| \ t \in \Z_{q^2}^\times \right\}. \label{eqn:norm_1_elts}
\end{align}
Therefore,
$$W_{\Theta, \varsigma}(d(a)) = \int_{\Z_{q^2}^\times} \varsigma( (th')^\sigma/t ) \psi( th'/t^\sigma ) \, dt.$$

We choose:
\begin{itemize}
    \item $\varsigma = \mathbf I_{\mathbb Z_q + q^r \varpi \Z_q}$, the indicator function of $\mathbb Z_q + q^r \varpi \Z_q \subseteq \Z_q + \varpi \Z_q = \Z_{q^2}$,
    \item $\vartheta = \mathbf I_{(\Z_q + q^r \Z_q \varpi) \mathfrak j}$ where $\mathfrak j = \begin{pmatrix}
    	1 & 0 \\
    	0 & -1
    \end{pmatrix}$ 
    \item $W_F(d(a)) = |a| \mathbf I_{\Z_q}(a)$, since $\pi = \St$ (cf.\ Proposition\ 4.23 of loc. cit.).
\end{itemize}

Then $ \varsigma( (th')^\sigma/t ) = 1$ if and only if 
$$(th')^\sigma/t \in \Z_q + q^r \varpi \Z_q.$$
For $x \in \Z_{q^2}$, the condition $x \in \Z_q + q^r \varpi \Z_q$ is equivalent to $v_q(x - x^\sigma) \geq r$. Therefore, the above condition is simply:
$$v_q((t^\sigma)^2 (h')^\sigma - t^2 h') \geq r.$$

Next, we observe that for $v(h') \geq r$, this condition is automatic, and in this case:
\begin{align*}
    W_{\Theta, \varsigma}(d(a)) & = \psi(h') \int\limits_{\Z_{q^2}^\times} (\psi/\psi^\sigma)(t) \, dt \\
    & = 0 & \text{$\psi/\psi^\sigma \neq 1$}.
\end{align*}

Since $W_F(d(a)) = |a| \mathbf I_{\Z_q}(a)$ and $h' (h')^\sigma = a$, the remaining part of the integral $J(\varsigma, \vartheta)$ has $0 \leq v_q(h') \leq r - 1$. Fix $m = v(h')$ and write $h' = q^m \cdot u$ for $u \in \Z_{q^2}^\times$. Then:
$$(t^\sigma)^2 (h')^\sigma - t^2 h' = q^m ((t^{\sigma})^2 u^\sigma - t^2 u),$$
and the condition on the valuation can be concisely written as
$$t^2 u \equiv (t^2 u)^\sigma \mod q^{r-m}.$$
Dividing both sides by the unit $t t^\sigma$ gives the equivalent condition:
$$tu/t^\sigma \equiv (tu/t^\sigma)^\sigma \mod q^{r-m}.$$

Finally, this shows that:
\begin{align*}
    \int\limits_{v_q(a) = 2m} W_F(d(a)) W_{\Theta, \varsigma}(d(a)) |a|^{s-1} d^\times a & = \iint\limits_{\substack{u, t \in \Z_{q^2}^\times \\ 
 tu/t^\sigma \equiv (tu/t^\sigma)^\sigma \bmod q^{r-m}} } \psi(tu/t^\sigma) \, d^\times u \, d^\times t \\
 & = \iint\limits_{\substack{v, t \in \Z_{q^2}^\times  \\ v \equiv v^\sigma \bmod q^{r-m} }} \psi(v)  \, d^\times v & v = t/t^\sigma u.
\end{align*}
Now, if $m \geq 1$, this final integral is 0, because there exists $v_0 \in \Z_{q^2}^\times$ such that $v_0 \equiv 1 \mod q^{r-m}$ and $\psi(v_0) \neq 1$. 

For $v \in \Z_{q^2}^{\times}$ such that $v \equiv v^\sigma \bmod q^{r}$, there exists $v_0 \in \Z_q^\times$ such that $v_0 \equiv v \bmod q^r$, and hence
$$\psi(v) = \psi(v_0) = 1,$$
since $\psi|_{\Z_q^\times} = 1$. Altogether, this shows that:
$$J(\varsigma, \vartheta) = \vol(\Z_{q^2}^\times) \cdot \vol\{ v \in \Z_{q^\times}^\times \ | \ v \equiv v^\sigma \bmod q^{r} \}.$$

Next, recall that the measure is normalized so that $\vol(\Q_{q^2}^{(1)}) = 1$ and hence by equation~\eqref{eqn:norm_1_elts}, we also have $\vol(\Z_{q^2}^\times) = 1$. To compute $J(\varsigma, \vartheta)$, it remains to compute the index of
$$S = \{v \in \Z_{q^2}^\times \ | \ v \equiv v^\sigma \bmod q^r\}$$ 
in $\Z_{q^2}^\times$. We consider the intersection of $S$ with the filtration $U_n =  1 + q^n \Z_{q^2}^\times$ on $U_0 = \Z_{q^2}^\times$. Clearly, for $n \geq r$, $S_n = S \cap U_n = U_n$ because the condition $v^\sigma \equiv v \bmod p^r$ is automatically satisfied. The successive quotients of the filtration are:
\begin{align*}
	U_0/U_1 & \iso \F_{q^2}^\times \\
	U_n/U_{n+1} & \iso \F_{q^2}
\end{align*}
and for $n < r$ the condition $v \equiv v^\sigma \bmod q^r$ gives:
\begin{align*}
	S_0/S_1 & \iso \F_{q}^\times, \\
	S_n/S_{n+1} & \iso \F_{q}.
\end{align*}
Overall, this shows that:
$$[\Z_{q^2}^\times : S] = \frac{q^2 -1}{q-1} \cdot q^{r-1}.$$
Therefore, finally:
$$J(\varsigma, \vartheta) = \frac{q-1}{q^{r-1}(q^2 - 1)}.$$
We rephrase this result as in Section 4.6 of loc. cit.

\begin{proposition}\label{prop:local_int_BDP}
	For the above choices of $\varsigma$ and $\vartheta$, we have
		$$J(\varsigma, \vartheta) = \frac{1}{q^r}  \frac{q-1}{q+1}  \cdot L_q(\pi_f, \pi_{\overline{\eta'}}, s) L_q(2s, \epsilon_K)^{-1}|_{s = 1/2}.$$
\end{proposition}
\begin{proof}
	We recall that:
	\begin{itemize}
		\item $L_q(1, \epsilon_K) =\frac{1}{1+1/q} = \frac{q}{q+1}$,
		\item by~\cite[Theorem 15.1]{Jacquet1972}, since $\pi = \St$, we have
		$$L(s, \pi \times \pi_\psi) = L(s, \pi_\psi \otimes |\cdot|^{1/2}) = L(s + 1/2, \pi_\psi) = L(s + 1/2, \psi),$$ 
		and hence: 
		$$L(1/2, \pi \times \pi_\psi) = L(1, \psi) = \frac{1}{1 - \psi(q) q^{-2}} = \frac{1}{1 - q^{-2}} = \frac{q^2}{q^2 - 1}.$$
	\end{itemize}
	Thus:
	\begin{align*}
		\frac{I(\varsigma, \vartheta)}{L_q(\pi_f, \pi_{\overline{\eta'}}, s) L_q(2s, \epsilon_K)^{-1}|_{s = 1/2}} & = \frac{q-1}{q^{r-1}(q^2 - 1)} \frac{q^2 - 1}{q^2} \frac{q}{q+1} \\
		& = \frac{1}{q^r} \cdot \frac{q-1}{q+1}. \qedhere
	\end{align*}

\end{proof}

\subsection{Finishing the proof}

Theorem~\ref{thm:BDP_generalization} is again proved by following \cite[Section 5]{BDP} word-for-word, and replacing Theorem 4.6 of loc.\ cit,\ with Theorem~\ref{thm:explicit_Waldspurger} above. 

\bigskip

\section{Examples}\label{section: appendix digest of examples}

The hypotheses of Conjecture~\ref{conj:ES} and Theorem~\ref{thm:CM} may seem quite restrictive at first glance, so in this section we collect many examples where they are satisfied. Even though we have only provided evidence for the conjecture in the CM case so far, we still wish to give instances where the conjecture could be verified numerically by adapting the algorithms in~\cite{DallAva2021Approx}. We hope to carry out this numerical verification in future work.

We divide the examples into four categories, namely:
\begin{enumerate}
        \item $h = g^\ast$, weight 1 forms with CM by the same field $K$;
        \item $h \neq g^\ast$, weight 1 forms with CM by the same field $K$;
        \item $h = g^\ast$, weight 1 forms without CM;
        \item $h \neq g^\ast$, weight 1 forms without CM.
\end{enumerate}
Recall that we denote the form obtained by complex conjugation from $g$ by $g^*$ and that its corresponding representation is the contragredient of $\pi_g$. In the CM cases, $K$ represents the imaginary quadratic field. As in the previous sections, we denote the various analytic ranks by $r(E_K)=\operatorname{ord}_{s=1}(L(E_K, s))$, and $r(E_K,\psi)=\operatorname{ord}_{s=1}(L(E_K,\psi, s))$. We always consider elliptic curves with analytic (and algebraic) rank $0$ over $\Q$. In the first two categories of examples, we keep track of which among the various Hypotheses \ref{hyp:epsilon-1}, \ref{hyp:classicality}, \ref{hyp:Ccond_at_most_2}, \ref{hyp:A'}, \ref{hyp:B'}, and \ref{hyp:C'} is satisfied. All our computations are made with the help of \magma~ and we make extensive use of the LMFDB Database~\cite{lmfdb}; all the labels refer to items listed there. We developed a simple procedure for computing the Hecke characters $\psi_g$ and $\psi_h$ and hence studying order of vanishing of the $L$-functions; these routines can be found at \citeGitHub{2}. In the following tables, we highlight the primes of supercuspidal type in \textcolor{blue}{blue}.

\subsection{CM case, $h = g^\ast$}\label{section: digest h=g^*, CM}

We consider here the case of $h=g^*$ with CM by $K$. In this situation, the two characters $\psi_1$ and $\psi_2$ of Section \ref{section: Application: Elliptic Stark Conjecture in rank one} are, respectively, $\mathds{1}$ and $\psi = \psi_g\psi_h^{\sigma}=\psi_g/\psi_g^\sigma$, for $\sigma$ the generator of $\Gal(K/\Q)$. For each example below, we checked that Hypotheses \ref{hyp:A'}, \ref{hyp:B'}, and \ref{hyp:Ccond_at_most_2} are satisfied, so the Elliptic Stark Conjecture~\ref{conj:ES} applies. The stronger hypothesis~\ref{hyp:C'}, under which Theorem~\ref{thm:CM} applies, is only sometimes satisfied so we separate the examples into Tables~\ref{tab: digest table g eq h, CM, the CM applies} and \ref{tab: digest table g eq h, CM, conj no thm}, accordingly.

\begin{table}[H]
    \centering
    \caption{Examples where $h=g^*$, both with CM by the same field $K$. Hypotheses \ref{hyp:A'}, \ref{hyp:B'}, and \ref{hyp:C'} are satisfied; Theorem~\ref{thm:CM} applies.}
    \label{tab: digest table g eq h, CM, the CM applies}
    {\renewcommand{\arraystretch}{1.5}
    
    \begin{tabular}{|c|c|c|c|c|c|c|}
        
        \hline

        \hline
        \textnormal{$E$} & \textnormal{Level} & \textnormal{$g=h^*$}  & \textnormal{Level}& $r(E_K)+r(E_K,\psi)$ & $K$& $p$ \\
        \hline

        \hline
        \href{https://www.lmfdb.org/EllipticCurve/Q/14/a/2}{14.a2} & $2 \cdot\textcolor{blue}{7}$ & \href{https://www.lmfdb.org/ModularForm/GL2/Q/holomorphic/1911/1/h/c/}{1911.1.h.c} & $3\cdot \textcolor{blue}{7}^2\cdot 13 $ & $0\, + \, 1$  & $\Q(\sqrt{-39})$ & 41\\

        \hline

        \href{https://www.lmfdb.org/EllipticCurve/Q/14/a/2}{14.a.2} & $2 \cdot\textcolor{blue}{7}$ & \href{https://www.lmfdb.org/ModularForm/GL2/Q/holomorphic/1911/1/h/d/}{1911.1.h.d}  & $3\cdot \textcolor{blue}{7}^2\cdot 13 $ & $0\, + \, 1$  & $\Q(\sqrt{-39})$ & 5 \\ 

        \hline
        
        \href{https://www.lmfdb.org/EllipticCurve/Q/17/a/1}{17.a1} & $\textcolor{blue}{17}$ & \href{https://www.lmfdb.org/ModularForm/GL2/Q/holomorphic/2023/1/c/a/}{2023.1.c.a}  & $7\cdot\textcolor{blue}{17}^2$ & $0\, + \, 1$  & $\Q(\sqrt{-7})$ & 23\\
        
        \hline
        
        \href{https://www.lmfdb.org/EllipticCurve/Q/17/a/1}{17.a1} & $\textcolor{blue}{17}$ & \href{https://www.lmfdb.org/ModularForm/GL2/Q/holomorphic/2023/1/c/b/}{2023.1.c.b}  & $7\cdot\textcolor{blue}{17}^2$ & $0\, + \, 1$  & $\Q(\sqrt{-7})$ & 23\\
        
        \hline

        \href{https://www.lmfdb.org/EllipticCurve/Q/17/a/1}{17.a1} & $\textcolor{blue}{17}$ & \href{https://www.lmfdb.org/ModularForm/GL2/Q/holomorphic/2023/1/c/c/}{2023.1.c.c}  & $7\cdot\textcolor{blue}{17}^2$ & $0\, + \, 1$  & $\Q(\sqrt{-7})$ & 11\\ 
        
        \hline
        
        \href{https://www.lmfdb.org/EllipticCurve/Q/17/a/1}{17.a1} & $\textcolor{blue}{17}$ & \href{https://www.lmfdb.org/ModularForm/GL2/Q/holomorphic/2023/1/c/d/}{2023.1.c.d}  & $7\cdot\textcolor{blue}{17}^2$ & $0\, + \, 1$  & $\Q(\sqrt{-7})$ & 11\\
        
        \hline

        \href{https://www.lmfdb.org/EllipticCurve/Q/19/a/1}{19.a.1} & $\textcolor{blue}{19}$ &  \href{https://www.lmfdb.org/ModularForm/GL2/Q/holomorphic/2527/1/d/c/}{2527.1.d.c} & $7\cdot \textcolor{blue}{19}^2 $ & $0\, + \, 1$  & $\Q(\sqrt{-7})$ & 11\\
        
        \hline
        
        \href{https://www.lmfdb.org/EllipticCurve/Q/21/a/4}{21.a4} & $3\cdot \textcolor{blue}{7}$ & \href{https://www.lmfdb.org/ModularForm/GL2/Q/holomorphic/2695/1/l/a/}{2695.1.l.a}  & $5\cdot\textcolor{blue}{7}^2\cdot 11$ & $0\, + \, 1$  & $\Q(\sqrt{-11})$ & 37\\
        
        \hline
        
        \href{https://www.lmfdb.org/EllipticCurve/Q/21/a/4}{21.a4} & $3\cdot \textcolor{blue}{7}$ & \href{https://www.lmfdb.org/ModularForm/GL2/Q/holomorphic/2695/1/g/i/}{2695.1.g.i}  & $5\cdot\textcolor{blue}{7}^2\cdot 11$ & $0\, + \, 1$ & $\Q(\sqrt{-11})$ & 67\\
        
        \hline
        
        \href{https://www.lmfdb.org/EllipticCurve/Q/26/a/1}{26.a1} & $2 \cdot\textcolor{blue}{13}$ & \href{https://www.lmfdb.org/ModularForm/GL2/Q/holomorphic/1183/1/d/a/}{1183.1.d.a}  & $7\cdot \textcolor{blue}{13}^2 $ & $0\, + \, 1$  & $\Q(\sqrt{-7})$ & 11\\
        
        \hline
        
        \href{https://www.lmfdb.org/EllipticCurve/Q/26/b/1}{26.b1} & $2 \cdot\textcolor{blue}{13}$ & \href{https://www.lmfdb.org/ModularForm/GL2/Q/holomorphic/1183/1/d/a/}{1183.1.d.a}  & $7\cdot \textcolor{blue}{13}^2 $ & $0\, + \, 1$  & $\Q(\sqrt{-7})$ & 11\\
        
        \hline

        \href{https://www.lmfdb.org/EllipticCurve/Q/34/a/1}{34.a1} & $2\cdot \textcolor{blue}{17}$ & \href{https://www.lmfdb.org/ModularForm/GL2/Q/holomorphic/2023/1/c/a/}{2023.1.c.a}  & $7\cdot\textcolor{blue}{17}^2$ & $0\, + \, 1$  & $\Q(\sqrt{-7})$ & 11\\
        
        \hline
        
        \href{https://www.lmfdb.org/EllipticCurve/Q/34/a/1}{34.a1} & $2\cdot \textcolor{blue}{17}$ & \href{https://www.lmfdb.org/ModularForm/GL2/Q/holomorphic/2023/1/c/b/}{2023.1.c.b}  & $7\cdot\textcolor{blue}{17}^2$ & $0\, + \, 1$  & $\Q(\sqrt{-7})$ & 11\\
        
        \hline
        
        \href{https://www.lmfdb.org/EllipticCurve/Q/34/a/1}{34.a1} & $2\cdot \textcolor{blue}{17}$ & \href{https://www.lmfdb.org/ModularForm/GL2/Q/holomorphic/2023/1/c/d/}{2023.1.c.d}  & $7\cdot\textcolor{blue}{17}^2$ & $0\, + \, 1$  & $\Q(\sqrt{-7})$ & 11\\
        
        \hline
        
        \href{https://www.lmfdb.org/EllipticCurve/Q/52/a/2}{52.a2} & $2^2 \cdot\textcolor{blue}{13}$ & \href{https://www.lmfdb.org/ModularForm/GL2/Q/holomorphic/1183/1/d/a/}{1183.1.d.a}  & $7\cdot \textcolor{blue}{13}^2 $ & $0\, + \, 1$  & $\Q(\sqrt{-7})$ & 11\\
        
        \hline

        \href{https://www.lmfdb.org/EllipticCurve/Q/55/a/1}{55.a1} & $\textcolor{blue}{5}\cdot 11$ & \href{https://www.lmfdb.org/ModularForm/GL2/Q/holomorphic/175/1/d/a/}{175.1.d.a}  & $\textcolor{blue}{5}^2\cdot 7 $ & $0\, + \, 1$ & $\Q(\sqrt{-7})$ & 23\\
        
        \hline
        
        \href{https://www.lmfdb.org/EllipticCurve/Q/187/b/1}{187.b1} & $11\cdot \textcolor{blue}{17}$ & \href{https://www.lmfdb.org/ModularForm/GL2/Q/holomorphic/2023/1/c/c/}{2023.1.c.c}  & $7\cdot\textcolor{blue}{17}^2$ & $0\, + \, 1$  & $\Q(\sqrt{-7})$ & 11\\ 
        
        \hline
        
        \hline

    \end{tabular}
    }
\end{table}

\begin{table}[H]
    \centering
    \caption{Examples where $h=g^*$, both with CM by $K$. Hypotheses \ref{hyp:A'}, \ref{hyp:B'}, and \ref{hyp:Ccond_at_most_2} are satisfied, so Conjecture~\ref{conj:ES} applies. However, Theorem~\ref{thm:CM} does not apply as Hypothesis~\ref{hyp:C'} is not satisfied.}
    \label{tab: digest table g eq h, CM, conj no thm}
    {\renewcommand{\arraystretch}{1.5}
    
    \begin{tabular}{|c|c|c|c|c|c|c|}

        \hline
        
        \hline
        \textnormal{$E$} & \textnormal{Level} & \textnormal{$g$}  & \textnormal{Level}& $r(E_K)+r(E_K,\psi)$ & $K$& $p$ \\
        \hline

        \hline
        
        \href{https://www.lmfdb.org/EllipticCurve/Q/15/a/3}{15.a3} & $3\cdot \textcolor{blue}{5}$ & \href{https://www.lmfdb.org/ModularForm/GL2/Q/holomorphic/525/1/be/a/}{525.1.be.a}  & $3\cdot\textcolor{blue}{5}^2\cdot7$ & $0\, + \, 1$ & $\Q(\sqrt{-3})$ & 13 \\
        
        \hline
        
        \href{https://www.lmfdb.org/EllipticCurve/Q/15/a/3}{15.a3} & $3\cdot\textcolor{blue}{5}$ & \href{https://www.lmfdb.org/ModularForm/GL2/Q/holomorphic/525/1/k/a/}{525.1.k.a}  & $3\cdot\textcolor{blue}{5}^2\cdot7$ & $0\, + \, 1$ & $\Q(\sqrt{-3})$ & 13 \\
        
        \hline
        
        \href{https://www.lmfdb.org/EllipticCurve/Q/15/a/3}{15.a3} & $\textcolor{blue}{3}\cdot 5$ & \href{https://www.lmfdb.org/ModularForm/GL2/Q/holomorphic/693/1/h/a/}{693.1.h.a}  & $\textcolor{blue}{3}^2\cdot 7 \cdot 11$ & $1\, + \, 0$ & $\Q(\sqrt{-7})$ & 11 \\
        
        \hline
        \href{https://www.lmfdb.org/EllipticCurve/Q/30/a/7}{30.a7} & $2\cdot \textcolor{blue}{3}\cdot 5$ & \href{https://www.lmfdb.org/ModularForm/GL2/Q/holomorphic/693/1/bp/a/}{693.1.bp.a}  & $\textcolor{blue}{3}^2\cdot 7 \cdot 11$ & $1\, + \, 0$  & $\Q(\sqrt{-7})$ & 11 \\
        
        \hline
        
        \href{https://www.lmfdb.org/EllipticCurve/Q/35/a/1}{35.a1} & $\textcolor{blue}{5}\cdot7$ & \href{https://www.lmfdb.org/ModularForm/GL2/Q/holomorphic/525/1/p/a/}{525.1.p.a}  & $3\cdot\textcolor{blue}{5}^2\cdot7$ & $0\, + \, 1$  & $\Q(\sqrt{-3})$ & 13 \\
        
        \hline
        
        \href{https://www.lmfdb.org/EllipticCurve/Q/35/a/1}{35.a1} & $\textcolor{blue}{5}\cdot7$ & \href{https://www.lmfdb.org/ModularForm/GL2/Q/holomorphic/525/1/u/a/}{525.1.u.a}  & $3\cdot\textcolor{blue}{5}^2\cdot7$ & $0\, +  \, 1$ & $\Q(\sqrt{-3})$ & 13 \\
        
        \hline
        
        \href{https://www.lmfdb.org/EllipticCurve/Q/35/a/1}{35.a1} & $\textcolor{blue}{5}\cdot7$ & \href{https://www.lmfdb.org/ModularForm/GL2/Q/holomorphic/525/1/u/b/}{525.1.u.b}  & $3\cdot\textcolor{blue}{5}^2\cdot7$ & $0\, +  \, 1$  & $\Q(\sqrt{-3})$ & 13 \\
        
        \hline

        \href{https://www.lmfdb.org/EllipticCurve/Q/39/a/4}{39.a4} & $3\cdot\textcolor{blue}{13}$ & \href{https://www.lmfdb.org/ModularForm/GL2/Q/holomorphic/1183/1/d/a/}{1183.1.d.a}  & $7\cdot \textcolor{blue}{13}^2 $ & $0\, + \, 1$  & $\Q(\sqrt{-7})$ & 11\\
        
        \hline

        \href{https://www.lmfdb.org/EllipticCurve/Q/49/a/4}{49.a4} & $\textcolor{blue}{7}^2$ & \href{https://www.lmfdb.org/ModularForm/GL2/Q/holomorphic/539/1/c/b/}{539.1.c.b}  & $\textcolor{blue}{7}^2\cdot 11$ & $0\, + \, 1$ & $\Q(\sqrt{-11})$ & 5 \\

        \hline

        \href{https://www.lmfdb.org/EllipticCurve/Q/51/a/1}{51.a.1} & $3\cdot \textcolor{blue}{17}$ & \href{https://www.lmfdb.org/ModularForm/GL2/Q/holomorphic/2023/1/c/d/}{2023.1.c.d}  & $7\cdot\textcolor{blue}{17}^2$ & $1\, + \, 0$  & $\Q(\sqrt{-7})$ & 11 \\
        
        \hline

        \href{https://www.lmfdb.org/EllipticCurve/Q/65/a/2}{65.a2} & $5\cdot\textcolor{blue}{13}$ & \href{https://www.lmfdb.org/ModularForm/GL2/Q/holomorphic/1183/1/d/a/}{1183.1.d.a}  & $7\cdot \textcolor{blue}{13}^2 $ & $1\, + \, 0$  & $\Q(\sqrt{-7})$ & 11\\

        \hline
        
        \href{https://www.lmfdb.org/EllipticCurve/Q/85/a/2}{85.a.2} & $5\cdot \textcolor{blue}{17}$ & \href{https://www.lmfdb.org/ModularForm/GL2/Q/holomorphic/2023/1/c/a/}{2023.1.c.a}  & $7\cdot\textcolor{blue}{17}^2$ & $1\, + \, 0$  & $\Q(\sqrt{-7})$ & 23 \\
        
        \hline
        
        \href{https://www.lmfdb.org/EllipticCurve/Q/85/a/2}{85.a.2} & $5\cdot \textcolor{blue}{17}$ & \href{https://www.lmfdb.org/ModularForm/GL2/Q/holomorphic/2023/1/c/b/}{2023.1.c.b}  & $7\cdot\textcolor{blue}{17}^2$ & $1\, + \, 0$  & $\Q(\sqrt{-7})$ & 23 \\
        
        \hline
        
        \href{https://www.lmfdb.org/EllipticCurve/Q/85/a/2}{85.a.2} & $5\cdot \textcolor{blue}{17}$ & \href{https://www.lmfdb.org/ModularForm/GL2/Q/holomorphic/2023/1/c/d/}{2023.1.c.d}  & $7\cdot\textcolor{blue}{17}^2$ & $1\, + \, 0$  & $\Q(\sqrt{-7})$ & 23 \\
        
        \hline

        \href{https://www.lmfdb.org/EllipticCurve/Q/195/a/1}{195.a1} & $3 \cdot\textcolor{blue}{5} \cdot 13$ & \href{https://www.lmfdb.org/ModularForm/GL2/Q/holomorphic/175/1/d/a/}{175.1.d.a}  & $\textcolor{blue}{5}^2\cdot 7 $ & $0\, + \, 1$ & $\Q(\sqrt{-7})$ & 23 \\
        
        \hline

        \hline 
    \end{tabular}
    }
\end{table}

\FloatBarrier

\subsection{CM case, $h \neq g^\ast$}\label{section: digest h neq g^*, CM}

We consider here the case of $h\neq g^*$ with CM by the same field $K$ and we collect the examples into two tables. Note that, differently from the previous section (Section \ref{section: digest h=g^*, CM}), we have $\psi_1,\psi_2\neq\mathds{1}$, hence it becomes difficult to check Hypothesis~\ref{hyp:C'} for the examples in Table~\ref{tab: digest table g neq h, CM, theorem applies} and Table~\ref{tab: digest table g neq h, CM, conj no thm}. This last condition can be checked by explicitly recovering the Hecke characters $\psi_1$ and $\psi_2$ via \citeGitHub{2} and computing the local $\epsilon$-factors. We list a few forms satisfying all the Hypotheses \ref{hyp:A'}, \ref{hyp:B'}, and \ref{hyp:Ccond_at_most_2} in Table~\ref{tab: digest table g neq h, CM, theorem applies}. In these examples, Hypothesis~\ref{hyp:C'} represents a rather strict assumption but, as observed in Remark~\ref{rmk: weak heegner hyp}, it can be relaxed by extending \cite{Brooks}. Some pairs are related by a twist of a Dirichler character, while others are not. The latter represents the most interesting situation and one can notice that the Artin representation $V_g\otimes V_h$ decomposes as direct sum of two irreducible two-dimensional Artin representations. The second table of this section, that is, Table~\ref{tab: digest table g neq h, CM, conj no thm} presents a few examples not satisfying Hypothesis~\ref{hyp:C'}, i.e.\ ones where Theorem~\ref{thm:CM} does not apply, but Conjecture~\ref{conj:ES} does.


\begin{table}[H]
        \centering
        \caption{Examples where $h\neq g^*$ both with CM by the same field $K$. Hypotheses \ref{hyp:A'}, \ref{hyp:B'}, and \ref{hyp:C'} are satisfied; Theorem~\ref{thm:CM} applies.}
        \label{tab: digest table g neq h, CM, theorem applies}
        {\renewcommand{\arraystretch}{1.5}
        \begin{tabular}{|c|c|c|c|c|c|c|}
            \hline
            
            \hline
            \multirow{2}{*}{\textnormal{$E$}} & \multirow{2}{*}{\textnormal{Level}} & $g$  & \textnormal{Level} & \multirow{2}{*}{$r(E_K,\psi_1)+r(E_K,\psi_2)$} & \multirow{2}{*}{$K$} & \multirow{2}{*}{$p$} \\
            \cline{3-4}
            & & $h$ & \textnormal{Level}  & &  &\\

            \hline

            \hline
         
            \multirow{2}{*}{\href{https://www.lmfdb.org/EllipticCurve/Q/14/a/1}{14.a1}} & \multirow{2}{*}{$2\cdot \textcolor{blue}{7}$} & \href{https://www.lmfdb.org/ModularForm/GL2/Q/holomorphic/539/1/c/a/}{539.1.c.a} & $\textcolor{blue}{7}^2 \cdot 11$ &  \multirow{2}{*}{$0\, + \, 1$} & \multirow{2}{*}{$\Q(\sqrt{-11})$} &  \multirow{2}{*}{5} \\
            \cline{3-4}
            & & \href{https://www.lmfdb.org/ModularForm/GL2/Q/holomorphic/2156/1/h/b/}{2156.1.h.b}  & $2^2\cdot\textcolor{blue}{7}^2 \cdot 5$  & & & \\

            \hline
            
            \multirow{2}{*}{\href{https://www.lmfdb.org/EllipticCurve/Q/15/a/1}{15.a1}} & \multirow{2}{*}{$3\cdot \textcolor{blue}{5}$} & \href{https://www.lmfdb.org/ModularForm/GL2/Q/holomorphic/175/1/d/a/}{175.1.d.a} & $\textcolor{blue}{5}^2 \cdot 7$ &  \multirow{2}{*}{$0\, + \, 1$} & \multirow{2}{*}{$\Q(\sqrt{-7})$} &  \multirow{2}{*}{11} \\
            \cline{3-4}
            & & \href{https://www.lmfdb.org/ModularForm/GL2/Q/holomorphic/1575/1/h/d/}{1575.1.h.d}  & $3^2\cdot\textcolor{blue}{5}^2 \cdot 7$  & & & \\

            \hline

            \multirow{2}{*}{\href{https://www.lmfdb.org/EllipticCurve/Q/70/a/1}{70.a1}} & \multirow{2}{*}{$2\cdot \textcolor{blue}{5}\cdot 7$} & \href{https://www.lmfdb.org/ModularForm/GL2/Q/holomorphic/525/1/u/a/}{525.1.u.a} & $3 \cdot\textcolor{blue}{5}^2 \cdot 7$ &  \multirow{2}{*}{$0\, + \, 1$} & \multirow{2}{*}{$\Q(\sqrt{-3})$} &   \multirow{2}{*}{13}\\
            \cline{3-4}
            & & \href{https://www.lmfdb.org/ModularForm/GL2/Q/holomorphic/2100/1/bn/a/}{2100.1.bn.a}  & $2^2 \cdot 3 \cdot\textcolor{blue}{5}^2 \cdot 7$  & & & \\
            
            \hline
            
            \multirow{2}{*}{\href{https://www.lmfdb.org/EllipticCurve/Q/77/b/1}{77.b1}} & \multirow{2}{*}{$\textcolor{blue}{7} \cdot 11$} & \href{https://www.lmfdb.org/ModularForm/GL2/Q/holomorphic/539/1/c/a/}{539.1.c.a} & $\textcolor{blue}{7}^2 \cdot 11$ &  \multirow{2}{*}{$0\, + \, 1$} & \multirow{2}{*}{$\Q(\sqrt{-11})$} &  \multirow{2}{*}{5} \\
            \cline{3-4}
            & & \href{https://www.lmfdb.org/ModularForm/GL2/Q/holomorphic/539/1/c/b/}{539.1.c.b}  & $\textcolor{blue}{7}^2 \cdot 11$  & & & \\

            \hline

            \hline

        \end{tabular}
        }
\end{table}

\begin{table}[H]
        \centering
        \caption{Examples where $h\neq g^*$ both with CM by the same field $K$. Hypotheses \ref{hyp:A'}, \ref{hyp:B'}, and \ref{hyp:Ccond_at_most_2} are satisfied; Conjecture~\ref{conj:ES} applies.}
        \label{tab: digest table g neq h, CM, conj no thm}
        {\renewcommand{\arraystretch}{1.5}
        \begin{tabular}{|c|c|c|c|c|c|c|}
            \hline
        
            \hline
            \multirow{2}{*}{\textnormal{$E$}} & \multirow{2}{*}{\textnormal{Level}} & $g$  & \textnormal{Level} & \multirow{2}{*}{$r(E_K,\psi_1)+r(E_K,\psi_2)$} & \multirow{2}{*}{$K$} & \multirow{2}{*}{$p$} \\
            \cline{3-4}
            & & $h$ & \textnormal{Level}  & & & \\

            \hline
            
            \hline

            \multirow{2}{*}{\href{https://www.lmfdb.org/EllipticCurve/Q/15/a/1}{15.a1}} & \multirow{2}{*}{$3 \cdot\textcolor{blue}{5}$} & \href{https://www.lmfdb.org/ModularForm/GL2/Q/holomorphic/525/1/u/a/}{525.1.u.a} & $3 \cdot\textcolor{blue}{5}^2 \cdot 7$ &  \multirow{2}{*}{$0\, + \, 1$}  & \multirow{2}{*}{$\Q(\sqrt{-3})$} &   \multirow{2}{*}{13}\\
            \cline{3-4}
            & & \href{https://www.lmfdb.org/ModularForm/GL2/Q/holomorphic/2100/1/bn/a/}{2100.1.bn.a}  & $2^2 \cdot 3 \cdot\textcolor{blue}{5}^2 \cdot 7$  & & & \\
            
            \hline

            \multirow{2}{*}{\href{https://www.lmfdb.org/EllipticCurve/Q/20/a/1}{20.a1}} & \multirow{2}{*}{$2^2\cdot \textcolor{blue}{5}$} & \href{https://www.lmfdb.org/ModularForm/GL2/Q/holomorphic/525/1/u/a/}{525.1.u.a} & $3\cdot\textcolor{blue}{5}^2\cdot7$ &  \multirow{2}{*}{$0\, + \, 1$} & \multirow{2}{*}{$\Q(\sqrt{-3})$} &  \multirow{2}{*}{13} \\
            \cline{3-4}
            & & \href{https://www.lmfdb.org/ModularForm/GL2/Q/holomorphic/525/1/u/b/}{525.1.u.b}  & $3\cdot\textcolor{blue}{5}^2\cdot7$  & & & \\

            \hline
            
            \multirow{2}{*}{\href{https://www.lmfdb.org/EllipticCurve/Q/42/a/1}{42.a1}} & \multirow{2}{*}{$2\cdot 3 \cdot  \textcolor{blue}{7}$} & \href{https://www.lmfdb.org/ModularForm/GL2/Q/holomorphic/539/1/c/a/}{539.1.c.a} & $\textcolor{blue}{7}^2 \cdot 11$ &  \multirow{2}{*}{$0\, + \, 1$} & \multirow{2}{*}{$\Q(\sqrt{-11})$} &  \multirow{2}{*}{11} \\
            \cline{3-4}
            & & \href{https://www.lmfdb.org/ModularForm/GL2/Q/holomorphic/2156/1/h/b/}{2156.1.h.b}  & $2^2\cdot\textcolor{blue}{7}^2 \cdot 5$  & & & \\
            
            \hline

            \multirow{2}{*}{\href{https://www.lmfdb.org/EllipticCurve/Q/55/a/1}{55.a1}} & \multirow{2}{*}{$\textcolor{blue}{5}\cdot 11$} & \href{https://www.lmfdb.org/ModularForm/GL2/Q/holomorphic/525/1/u/a/}{525.1.u.a} & $3\cdot\textcolor{blue}{5}^2\cdot7$ &  \multirow{2}{*}{$0\, + \, 1$} & \multirow{2}{*}{$\Q(\sqrt{-3})$} &  \multirow{2}{*}{13} \\
            \cline{3-4}
            & & \href{https://www.lmfdb.org/ModularForm/GL2/Q/holomorphic/525/1/u/b/}{525.1.u.b}  & $3\cdot\textcolor{blue}{5}^2\cdot7$  & & & \\

            \hline

            \multirow{2}{*}{\href{https://www.lmfdb.org/EllipticCurve/Q/155/b/2}{155.a2}} & \multirow{2}{*}{$\textcolor{blue}{5}\cdot 31$} & \href{https://www.lmfdb.org/ModularForm/GL2/Q/holomorphic/525/1/u/a/}{525.1.u.a} & $3 \cdot\textcolor{blue}{5}^2 \cdot 7$ &  \multirow{2}{*}{$0\, + \, 1$} & \multirow{2}{*}{$\Q(\sqrt{-3})$} &   \multirow{2}{*}{13}\\
            \cline{3-4}
            & & \href{https://www.lmfdb.org/ModularForm/GL2/Q/holomorphic/2100/1/bn/a/}{2100.1.bn.a}  & $2^2 \cdot 3 \cdot\textcolor{blue}{5}^2 \cdot 7$  & & & \\
            
            \hline

            \multirow{2}{*}{\href{https://www.lmfdb.org/EllipticCurve/Q/155/b/2}{155.a2}} & \multirow{2}{*}{$\textcolor{blue}{5}\cdot 31$} & \href{https://www.lmfdb.org/ModularForm/GL2/Q/holomorphic/525/1/u/a/}{525.1.u.a} & $3 \cdot\textcolor{blue}{5}^2 \cdot 7$ &  \multirow{2}{*}{$0\, + \, 1$} & \multirow{2}{*}{$\Q(\sqrt{-3})$} &   \multirow{2}{*}{13}\\
            \cline{3-4}
            & &  \href{https://www.lmfdb.org/ModularForm/GL2/Q/holomorphic/2100/1/bn/b/}{2100.1.bn.b}  & $2^2 \cdot 3 \cdot\textcolor{blue}{5}^2 \cdot 7$  & & & \\

            \hline

            \multirow{2}{*}{\href{https://www.lmfdb.org/EllipticCurve/Q/210/a/1}{210.a1}} & \multirow{2}{*}{$2\cdot 3 \cdot \textcolor{blue}{5}\cdot 7$} & \href{https://www.lmfdb.org/ModularForm/GL2/Q/holomorphic/525/1/u/a/}{525.1.u.a} & $3 \cdot\textcolor{blue}{5}^2 \cdot 7$ &  \multirow{2}{*}{$0\, + \, 1$} & \multirow{2}{*}{$\Q(\sqrt{-3})$} &   \multirow{2}{*}{13}\\
            \cline{3-4}
            & & \href{https://www.lmfdb.org/ModularForm/GL2/Q/holomorphic/2100/1/bn/a/}{2100.1.bn.a}  & $2^2 \cdot 3 \cdot\textcolor{blue}{5}^2 \cdot 7$  & & & \\

            \hline

            \multirow{2}{*}{\href{https://www.lmfdb.org/EllipticCurve/Q/210/d/1}{210.d1}} & \multirow{2}{*}{$2\cdot 3 \cdot \textcolor{blue}{5}\cdot 7$} & \href{https://www.lmfdb.org/ModularForm/GL2/Q/holomorphic/525/1/u/a/}{525.1.u.a} & $3 \cdot\textcolor{blue}{5}^2 \cdot 7$ &  \multirow{2}{*}{$0\, + \, 1$} & \multirow{2}{*}{$\Q(\sqrt{-3})$} &   \multirow{2}{*}{13}\\
            \cline{3-4}
            & & \href{https://www.lmfdb.org/ModularForm/GL2/Q/holomorphic/2100/1/bn/a/}{2100.1.bn.a}  & $2^2 \cdot 3 \cdot\textcolor{blue}{5}^2 \cdot 7$  & & & \\

            \hline

            \multirow{2}{*}{\href{https://www.lmfdb.org/EllipticCurve/Q/490/a/1}{490.a1}} & \multirow{2}{*}{$2\cdot 3 \cdot \textcolor{blue}{5}\cdot 7$} & \href{https://www.lmfdb.org/ModularForm/GL2/Q/holomorphic/525/1/u/a/}{525.1.u.a} & $3 \cdot\textcolor{blue}{5}^2 \cdot 7$ &  \multirow{2}{*}{$0\, + \, 1$} & \multirow{2}{*}{$\Q(\sqrt{-3})$} &   \multirow{2}{*}{13}\\
            \cline{3-4}
            & & \href{https://www.lmfdb.org/ModularForm/GL2/Q/holomorphic/2100/1/bn/a/}{2100.1.bn.a}  & $2^2 \cdot 3 \cdot\textcolor{blue}{5}^2 \cdot 7$  & & & \\

            \hline
            
            \hline
            
        \end{tabular}
        }
\end{table}

\begin{remark}
    It is not automatic that finding the right level and character produces a situation where our work applies. By Proposition~\ref{prop:epsilon=-1}~(2) (\cite[Proposition 8.5]{Prasad1990}), in order to obtain local sign $\varepsilon_\ell(E, \varrho_{gh})=-1$, the supercuspidal representations $\pi_{g,\ell}$ and $\pi_{h,\ell}$ need to satisfy $\pi_{g,\ell}\cong \pi_{h,\ell}^*=\pi_{h^*,\ell}$. There are several examples which fail this final condition:
    \begin{itemize}
        \item $g$: \href{https://www.lmfdb.org/ModularForm/GL2/Q/holomorphic/2023/1/c/a/}{2023.1.c.a}, \href{https://www.lmfdb.org/ModularForm/GL2/Q/holomorphic/2023/1/c/b/}{2023.1.c.b} and $h$:  \href{https://www.lmfdb.org/ModularForm/GL2/Q/holomorphic/2023/1/c/c/}{2023.1.c.c}, \href{https://www.lmfdb.org/ModularForm/GL2/Q/holomorphic/2023/1/c/d/}{2023.1.c.d};
        \item $g$: \href{https://www.lmfdb.org/ModularForm/GL2/Q/holomorphic/2527/1/d/c/}{2527.1.d.c}, \href{https://www.lmfdb.org/ModularForm/GL2/Q/holomorphic/2527/1/d/d/}{2527.1.d.d} and $h$:  \href{https://www.lmfdb.org/ModularForm/GL2/Q/holomorphic/2527/1/d/f/}{2527.1.d.f}.
    \end{itemize}
    
\end{remark}

\FloatBarrier

\subsection{Non-CM case, $h = g^\ast$} In this section we report a few examples to which our Conjecture~\ref{conj:ES} applies, but the form $g$ does not have CM, so Theorem~\ref{thm:CM} does not. We divide the examples depending on the projective image of the Artin representation associated with $g$.

\subsubsection{RM but no CM:} We begin by giving examples of some weight 1 modular forms with RM by a field $F/\Q$ but without CM. When $h = g^\ast$, we once again have
\begin{equation}
    V_g\otimes V_{g^*} = \Ind^F_\Q(\mathds{1}) \oplus V_\psi,
\end{equation}
for $\psi =\psi_g/\psi_g^\sigma$; here $\sigma$ is a generator of $\Gal(F/\Q)$.
\begin{table}[ht!]
    \centering
    \caption{Examples where $h=g^*$, both with RM by $F$; $r(E_F)$ and $r(E_F,\psi)$ are computed.}
    \label{tab:digest_table_geq_h_RM}
    {\renewcommand{\arraystretch}{1.5}
    \begin{tabular}{|c|c|c|c|c|c|c|c|}
        \hline
        \textnormal{$E$} & \textnormal{Level} & $g$ & \textnormal{Level} & $r(E_F) + r(E_F,\psi)$ & $F$ & $p$\\
        \hline
        
        \hline
        \href{https://www.lmfdb.org/EllipticCurve/Q/15/a/1}{15.a1} & $3\cdot\textcolor{blue}{5}$ & \href{https://www.lmfdb.org/ModularForm/GL2/Q/holomorphic/1025/1/i/a/}{1025.1.i.a}  & $\textcolor{blue}{5}^2\cdot 41$ & 0\, +\, 1 & $\Q(\sqrt{5})$ & 11\\
        \hline
        \href{https://www.lmfdb.org/EllipticCurve/Q/15/a/1}{15.a1} & $3\cdot\textcolor{blue}{5}$ & \href{https://www.lmfdb.org/ModularForm/GL2/Q/holomorphic/1025/1/f/a/}{1025.1.f.a}  & $\textcolor{blue}{5}^2\cdot 41$ & 0\, +\, 1 & $\Q(\sqrt{5})$ & 11\\
        
        \hline
        \href{https://www.lmfdb.org/EllipticCurve/Q/20/a/1}{20.a1} & $2^2\cdot\textcolor{blue}{5}$ & \href{https://www.lmfdb.org/ModularForm/GL2/Q/holomorphic/1025/1/i/a/}{1025.1.i.a}  & $\textcolor{blue}{5}^2\cdot 41$ & 0\, +\, 1 & $\Q(\sqrt{5})$ & 11\\
        
        \hline
        \href{https://www.lmfdb.org/EllipticCurve/Q/20/a/1}{20.a1} & $2^2\cdot\textcolor{blue}{5}$ & \href{https://www.lmfdb.org/ModularForm/GL2/Q/holomorphic/1025/1/f/a/}{1025.1.f.a}  & $\textcolor{blue}{5}^2\cdot 41$ & 0\, +\, 1 & $\Q(\sqrt{5})$ & 11\\

        \hline
        \href{https://www.lmfdb.org/EllipticCurve/Q/21/a/1}{21.a1} & $\textcolor{blue}{3}\cdot 7$ & \href{https://www.lmfdb.org/ModularForm/GL2/Q/holomorphic/396/1/d/a/}{396.1.d.a}  & $2^2\cdot \textcolor{blue}{3}^2\cdot 11$ & 1\, +\, 0 & $\Q(\sqrt{11})$ & 5\\ 
        
        \hline
        \href{https://www.lmfdb.org/EllipticCurve/Q/30/a/1}{30.a1} & $2\cdot 3 \cdot\textcolor{blue}{5}$ & \href{https://www.lmfdb.org/ModularForm/GL2/Q/holomorphic/1025/1/i/a/}{1025.1.i.a}  & $\textcolor{blue}{5}^2\cdot 41$ & 0\, +\, 1 & $\Q(\sqrt{5})$ & 11\\
        
        \hline
        \href{https://www.lmfdb.org/EllipticCurve/Q/30/a/1}{30.a1} & $2\cdot 3 \cdot\textcolor{blue}{5}$ & \href{https://www.lmfdb.org/ModularForm/GL2/Q/holomorphic/1025/1/f/a/}{1025.1.f.a}  & $\textcolor{blue}{5}^2\cdot 41$ & 0\, +\, 1 & $\Q(\sqrt{5})$ & 11\\

        \hline
        \href{https://www.lmfdb.org/EllipticCurve/Q/39/a/1}{39.a1} & $\textcolor{blue}{3}\cdot 7$ & \href{https://www.lmfdb.org/ModularForm/GL2/Q/holomorphic/396/1/d/a/}{396.1.d.a}  & $2^2\cdot \textcolor{blue}{3}^2\cdot 11$ & 0\, +\, 1 & $\Q(\sqrt{11})$ & 5\\ 
        
        \hline

        \hline
    \end{tabular}
    }
\end{table}

\FloatBarrier

\subsubsection{$A_4$:} To compute these examples, we make wide use of the technical remarks in \cite[Section 5.1]{Darmon_Lauder_Rotger_Stark_points}. In this situation, $V_g\otimes V_{g^*}$ decomposes as $\mathds{1} \oplus \Ad_g$, for $\Ad_g$ the adjoint representation of $g$. Therefore it is enough to compute the order of vanishing $\mathrm{ord}_{s=1}L(E,\Ad_g,s)$. Let $H$ be the Galois field defined by the \emph{projective} Artin representation associated with $g$, let $h(x)$ be a degree 4 polynomial whose splitting field is $H$. Denoting by $M$ the field generated by a single root of $h(x)$, by \cite[Section 5.1.1]{Darmon_Lauder_Rotger_Stark_points}, we have
\begin{equation}
    r(E,\Ad_g)=\mathrm{ord}_{s=1}L(E,\Ad_g,s)= r(E_M) - r(E).
\end{equation}
Hence, $\mathrm{ord}_{s=1} L(f\times g \times h, 1) = r(E_M)$.

\begin{table}[ht!]
    \centering
    \caption{Examples where $h=g^*$ is an exotic form of projective type $A_4$; $r(E_M)$ and $r(E_H)$ are computed.}
    \label{tab:digest_table_geq_h_A4}
    {\renewcommand{\arraystretch}{1.5}
    \begin{tabular}{|c|c|c|c|c|c|c|c|c|}
        \hline
        \textnormal{$E$} & \textnormal{Level} & \textnormal{$g$} & \textnormal{Level} & $r(E)$ & $r(E_M)$ & $r(E_H)$ & Polynomial of $M$ & $p$ \\
        \hline
        
        \hline
        
        \href{https://www.lmfdb.org/EllipticCurve/Q/15/a/1}{15.a1} & $3\cdot\textcolor{blue}{5}$ & \href{https://www.lmfdb.org/ModularForm/GL2/Q/holomorphic/325/1/u/a/}{325.1.u.a} & $\textcolor{blue}{5}^2\cdot 13$ & 0 & 1 & $3$ & $x^4 - x^3 - 3\cdot x + 4$ & 11 \\
        \hline
        
        \href{https://www.lmfdb.org/EllipticCurve/Q/26/a/1}{26.a1} & $2\cdot\textcolor{blue}{13}$ & \href{https://www.lmfdb.org/ModularForm/GL2/Q/holomorphic/1183/1/x/a/}{1183.1.x.a} & $7\cdot \textcolor{blue}{13}^2$ & 0 & 1 & $3$ & $x^4 - x^3 + 5\cdot x^2 - 4\cdot x + 3$ & 11 \\
        \hline
        
        \href{https://www.lmfdb.org/EllipticCurve/Q/26/a/1}{26.a1} & $2\cdot\textcolor{blue}{13}$ & \href{https://www.lmfdb.org/ModularForm/GL2/Q/holomorphic/1183/1/z/a/}{1183.1.z.a} & $7\cdot\textcolor{blue}{13}^2$ & 0 & 1 & $3$ & $x^4 - x^3 + 5\cdot x^2 - 4\cdot x + 3$ & 11 \\
        \hline
        
        \href{https://www.lmfdb.org/EllipticCurve/Q/26/a/1}{26.a1} & $2\cdot\textcolor{blue}{13}$ & \href{https://www.lmfdb.org/ModularForm/GL2/Q/holomorphic/1183/1/bd/a/}{1183.1.bd.a} & $7\cdot\textcolor{blue}{13}^2$ & 0 & 1 & $3$ & $x^4 - x^3 + 5\cdot x^2 - 4\cdot x + 3$ & 11 \\
        \hline
        
        \href{https://www.lmfdb.org/EllipticCurve/Q/39/a/1}{39.a1} & $3\cdot\textcolor{blue}{13}$ & \href{https://www.lmfdb.org/ModularForm/GL2/Q/holomorphic/1183/1/x/a/}{1183.1.x.a} & $7\cdot \textcolor{blue}{13}^2$ & 0 & 1 & $5$ & $x^4 - x^3 + 5\cdot x^2 - 4\cdot x + 3$ & 11 \\
        \hline
        
        \href{https://www.lmfdb.org/EllipticCurve/Q/39/a/1}{39.a1} & $3\cdot\textcolor{blue}{13}$ & \href{https://www.lmfdb.org/ModularForm/GL2/Q/holomorphic/1183/1/z/a/}{1183.1.z.a} & $7\cdot\textcolor{blue}{13}^2$ & 0 & 1 & $5$ & $x^4 - x^3 + 5\cdot x^2 - 4\cdot x + 3$ & 11 \\
        \hline
        
        \href{https://www.lmfdb.org/EllipticCurve/Q/39/a/1}{39.a1} & $3\cdot\textcolor{blue}{13}$ & \href{https://www.lmfdb.org/ModularForm/GL2/Q/holomorphic/1183/1/bd/a/}{1183.1.bd.a} & $7\cdot\textcolor{blue}{13}^2$ & 0 & 1 & $5$ & $x^4 - x^3 + 5\cdot x^2 - 4\cdot x + 3$ & 11 \\
        \hline
        
        \hline
    \end{tabular}
    }
\end{table}

\FloatBarrier

\subsubsection{$S_4$:} As above, $V_g\otimes V_h$ decomposes as $\mathds{1} \oplus \Ad_g$.
\begin{table}[H]
    \centering
    \caption{Examples where $h=g^*$ is an exotic form of projective type $S_4$; $r(E)$ and $r(E,\Ad_g)$ are computed.}
    \label{tab:digest_table_geq_h_S4}
    {\renewcommand{\arraystretch}{1.5}
    
    \begin{tabular}{|c|c|c|c|c|c|}
        
        \hline
        
        \textnormal{$E$} & \textnormal{Level} & \textnormal{$g$}  & \textnormal{Level}& $r(E)+r(E,\Ad_g)$ & $p$ \\
        
        \hline

        \hline
        
        \href{https://www.lmfdb.org/EllipticCurve/Q/11/a/1}{11.a1} & $\textcolor{blue}{11}$ & \href{https://www.lmfdb.org/ModularForm/GL2/Q/holomorphic/968/1/h/a/}{968.1.h.a}  & $2^3\cdot \textcolor{blue}{11}^2$ & $0\, + \, 1$ & 5 \\
        
        \hline

        \href{https://www.lmfdb.org/EllipticCurve/Q/33/a/1}{33.a1} & $3\cdot \textcolor{blue}{11}$ & \href{https://www.lmfdb.org/ModularForm/GL2/Q/holomorphic/968/1/h/a/}{968.1.h.a}  & $2^3\cdot \textcolor{blue}{11}^2$ & $0\, + \, 1$ & 5 \\
        
        \hline

        \href{https://www.lmfdb.org/EllipticCurve/Q/55/a/1}{55.a1} & $5\cdot \textcolor{blue}{11}$ & \href{https://www.lmfdb.org/ModularForm/GL2/Q/holomorphic/968/1/h/a/}{968.1.h.a}  & $2^3\cdot \textcolor{blue}{11}^2$ & $0\, + \, 1$ & 23 \\
        
        \hline
        
        \hline
    \end{tabular}
    }
\end{table}

We report here a few more modular forms to which our setting applies, but computing the corresponding analytic ranks seems computationally demanding: \href{https://www.lmfdb.org/ModularForm/GL2/Q/holomorphic/1224/1/m/a/}{1224.1.m.a}, \href{https://www.lmfdb.org/ModularForm/GL2/Q/holomorphic/1224/1/m/b/}{1224.1.m.b}, \href{https://www.lmfdb.org/ModularForm/GL2/Q/holomorphic/1800/1/l/a/}{1800.1.l.a}, and \href{https://www.lmfdb.org/ModularForm/GL2/Q/holomorphic/1800/1/l/b/}{1800.1.l.b}.

\FloatBarrier

\subsubsection{$A_5$:}

Unfortunately, computing examples with forms of projective image $A_5$
seems out of reach at the present moment, due to the difficulty of computing Artin representations over degree 60 number fields. Even trying to address the problem as in \cite[Section 5.1.2]{Darmon_Lauder_Rotger_Stark_points}, as the computation of the unique subfield of degree 2 contained in the Artin field is demanding. However, we list here a few modular forms to which our setting should apply: \href{https://www.lmfdb.org/ModularForm/GL2/Q/holomorphic/1825/1/y/a/}{1825.1.y.a}, \href{https://www.lmfdb.org/ModularForm/GL2/Q/holomorphic/2079/1/dd/a/}{2079.1.dd.a}, \href{https://www.lmfdb.org/ModularForm/GL2/Q/holomorphic/3069/1/cd/a/}{3069.1.cd.a}, \href{https://www.lmfdb.org/ModularForm/GL2/Q/holomorphic/3069/1/cd/b/}{3069.1.cd.b}, and \href{https://www.lmfdb.org/ModularForm/GL2/Q/holomorphic/3168/1/cb/a/}{3168.1.cb.a}.

\subsection{Non-CM case, $h \neq g^\ast$}
Computing examples in this setting is rather challenging, but we present a few examples of dihedral forms with RM and no CM, and exotic ones with $S_4$ projective image. In both cases, we only consider forms $g$ and $h$ with are related by a twist by Dirichlet character.

\subsubsection{RM but no CM:} In this situation, we consider $g$ and $h$ with RM by the same field $\Q(\sqrt{5})$.
We write $V_g\otimes V_h=\Ind^\Q_{\Q(\sqrt{5})}(\psi_g)\otimes\Ind^\Q_{\Q(\sqrt{5})}(\psi_h) \Ind^\Q_{\Q(\sqrt{5})}(\psi_g \psi_h)\oplus\Ind^\Q_{\Q(\sqrt{5})}(\psi_g \psi_h^\sigma)$. 

\begin{table}[ht!]
    \centering
    \caption{Examples where $h\neq g^*$ are dihedral forms, both with RM by $\Q(\sqrt{5})$ and without CM.}
    \label{tab:digest_table_gneq_h_RM}
    {\renewcommand{\arraystretch}{1.5}
    \begin{tabular}{|c|c|c|c|c|c|}
        \hline
        \multirow{2}{*}{\textnormal{$E$}} & \multirow{2}{*}{\textnormal{Level}} & $g$ & \textnormal{Level} & \multirow{2}{*}{$r(E_F,\psi_g\psi_h) + r(E_F,\psi_g\psi_h^\sigma)$} & \multirow{2}{*}{$p$} \\
        \cline{3-4}
        & & $h$ & \textnormal{Level}  & & \\

        \hline
        
        \hline
        \multirow{2}{*}{\href{https://www.lmfdb.org/EllipticCurve/Q/15/a/1}{15.a1}} & \multirow{2}{*}{$3 \cdot\textcolor{blue}{5}$} & \href{https://www.lmfdb.org/ModularForm/GL2/Q/holomorphic/1025/1/i/a/}{1025.1.i.a} & $\textcolor{blue}{5}^2\cdot 41$ & \multirow{2}{*}{0\, +\, 1} & \multirow{2}{*}{11}\\
        \cline{3-4}
        & & \href{https://www.lmfdb.org/ModularForm/GL2/Q/holomorphic/1025/1/f/a/}{1025.1.f.a}  & $\textcolor{blue}{5}^2\cdot 41$  & & \\
        
        \hline
        \multirow{2}{*}{\href{https://www.lmfdb.org/EllipticCurve/Q/20/a/1}{20.a1}} & \multirow{2}{*}{$2^2\cdot\textcolor{blue}{5}$} & \href{https://www.lmfdb.org/ModularForm/GL2/Q/holomorphic/1025/1/i/a/}{1025.1.i.a} & $\textcolor{blue}{5}^2\cdot 41$ & \multirow{2}{*}{1\, +\, 0} & \multirow{2}{*}{11}\\
        \cline{3-4}
        & & \href{https://www.lmfdb.org/ModularForm/GL2/Q/holomorphic/1025/1/f/a/}{1025.1.f.a}  & $\textcolor{blue}{5}^2\cdot 41$  & & \\
        
        \hline
        \multirow{2}{*}{\href{https://www.lmfdb.org/EllipticCurve/Q/30/a/1}{30.a1}} & \multirow{2}{*}{$2\cdot 3 \cdot\textcolor{blue}{5}$} & \href{https://www.lmfdb.org/ModularForm/GL2/Q/holomorphic/1025/1/i/a/}{1025.1.i.a} & $\textcolor{blue}{5}^2\cdot 41$ & \multirow{2}{*}{0\, +\, 1} & \multirow{2}{*}{11}\\
        \cline{3-4}
        & & \href{https://www.lmfdb.org/ModularForm/GL2/Q/holomorphic/1025/1/f/a/}{1025.1.f.a}  & $\textcolor{blue}{5}^2\cdot 41$  & & \\
        
        \hline
        
        \hline
    \end{tabular}
    }
\end{table}

\FloatBarrier

\subsubsection{$S_4$:} We consider the case of $g\neq h^*$ but $h=g\otimes \chi$, for $\chi$ a Dirichlet character. In this situation, the representation $V_g\otimes V_h$ decomposes as $\rho_1\oplus \rho_3$ for $\rho_d$ a $d$-dimensional Artin representation, $d=1,3$.
\begin{table}[ht!]
    \centering
    \caption{Examples where $h\neq g^*$, $h=g\otimes \chi$,  exotic forms of projective type $S_4$; $r(E,\rho_1)$ and $r(E,\rho_3)$ are computed.}
    \label{tab:digest_table_gneq_h_non-CM_non-twisted}
    {\renewcommand{\arraystretch}{1.5}
    \begin{tabular}{|c|c|c|c|c|c|}
        \hline
        \multirow{2}{*}{\textnormal{$E$}} & \multirow{2}{*}{\textnormal{Level}} & $g$ & \textnormal{Level} & \multirow{2}{*}{$r(E,\rho_1)+r(E,\rho_3)$} & \multirow{2}{*}{$p$} \\
        \cline{3-4}
        & & $h$ & \textnormal{Level}  & & \\

        \hline
        
        \hline
        
        \multirow{2}{*}{\href{https://www.lmfdb.org/EllipticCurve/Q/15/a/1}{15.a1}} & \multirow{2}{*}{$\textcolor{blue}{3}\cdot 5$} & \href{https://www.lmfdb.org/ModularForm/GL2/Q/holomorphic/981/1/d/a/}{981.1.d.a} & $\textcolor{blue}{3}^2 \cdot 109$ &  \multirow{2}{*}{$0\, + \, 1$} & \multirow{2}{*}{7} \\
        \cline{3-4}
        & & \href{https://www.lmfdb.org/ModularForm/GL2/Q/holomorphic/981/1/d/b/}{981.1.d.b}  & $\textcolor{blue}{3}^2 \cdot 109$  & & \\
        
        \hline
        
        \multirow{2}{*}{\href{https://www.lmfdb.org/EllipticCurve/Q/24/a/1}{24.a1}} & \multirow{2}{*}{$2^3\cdot \textcolor{blue}{3}$} & \href{https://www.lmfdb.org/ModularForm/GL2/Q/holomorphic/981/1/d/a/}{981.1.d.a} & $\textcolor{blue}{3}^2 \cdot 109$ &  \multirow{2}{*}{$1\, + \, 0$} & \multirow{2}{*}{7} \\
        \cline{3-4}
        & & \href{https://www.lmfdb.org/ModularForm/GL2/Q/holomorphic/981/1/d/b/}{981.1.d.b}  & $\textcolor{blue}{3}^2 \cdot 109$  & & \\
        
        \hline
        
        \multirow{2}{*}{\href{https://www.lmfdb.org/EllipticCurve/Q/30/a/1}{30.a1}} & \multirow{2}{*}{$2\cdot\textcolor{blue}{3}\cdot 5$} & \href{https://www.lmfdb.org/ModularForm/GL2/Q/holomorphic/981/1/d/a/}{981.1.d.a} & $\textcolor{blue}{3}^2 \cdot 109$ &  \multirow{2}{*}{$1\, + \, 0$} & \multirow{2}{*}{7} \\
        \cline{3-4}
        & & \href{https://www.lmfdb.org/ModularForm/GL2/Q/holomorphic/981/1/d/b/}{981.1.d.b}  & $\textcolor{blue}{3}^2 \cdot 109$  & & \\
        
        \hline
        
        \multirow{2}{*}{\href{https://www.lmfdb.org/EllipticCurve/Q/33/a/1}{33.a1}} & \multirow{2}{*}{$\textcolor{blue}{3}\cdot 11$} & \href{https://www.lmfdb.org/ModularForm/GL2/Q/holomorphic/981/1/d/a/}{981.1.d.a} & $\textcolor{blue}{3}^2 \cdot 109$ &  \multirow{2}{*}{$1\, + \, 0$} & \multirow{2}{*}{7} \\
        \cline{3-4}
        & & \href{https://www.lmfdb.org/ModularForm/GL2/Q/holomorphic/981/1/d/b/}{981.1.d.b}  & $\textcolor{blue}{3}^2 \cdot 109$  & & \\
        
        \hline
        
        \hline
    \end{tabular}
    }
\end{table}

\FloatBarrier

{\setstretch{1.0}
    \bibliographystyle{amsalpha}
	\bibliography{Bibliography}
    \addtocontents{toc}{\protect\setcounter{tocdepth}{-1}}\bibliographystyleGitHub{unsrt}
    \bibliographyGitHub{GitHubRepositories}
 }
{   \hypersetup{hidelinks}	
	\Addresses}

\end{document}